\documentclass[11pt, reqno]{amsart}

\usepackage[utf8]{inputenc}
\usepackage[T1]{fontenc}
\usepackage[english]{babel}

\usepackage{newpxtext}
\usepackage[vvarbb]{newpxmath}
\usepackage{microtype}

\usepackage[margin=1.15in]{geometry}
\usepackage{amsmath, amsthm}
\usepackage{mathtools}
\usepackage{esint}
\usepackage{dsfont}

\usepackage{enumitem}
\usepackage{xcolor}
\usepackage{cite}
\usepackage{tikz}
\usetikzlibrary{arrows.meta,decorations.pathreplacing,calc}
\usepackage{wrapfig}
\usepackage{xspace}

\usepackage[colorlinks=true,
            linkcolor=blue!70!black,
            citecolor=teal,
            urlcolor=magenta]{hyperref}
\usepackage[capitalize]{cleveref}

\usepackage{aliascnt}
\newcommand{\sharedthm}[3]{%
  \newaliascnt{#1}{theorem}\newtheorem{#1}[#1]{#2}%
  \aliascntresetthe{#1}\crefname{#1}{#2}{#3}}
\theoremstyle{plain}
\newtheorem{theorem}{Theorem}[section]
\crefname{theorem}{Theorem}{Theorems}
\sharedthm{lemma}{Lemma}{Lemmas}
\sharedthm{proposition}{Proposition}{Propositions}
\sharedthm{corollary}{Corollary}{Corollaries}
\newtheorem{maintheorem}{Theorem}

\crefname{maintheorem}{Theorem}{Theorems}

\theoremstyle{definition}
\sharedthm{definition}{Definition}{Definitions}

\theoremstyle{remark}
\sharedthm{remark}{Remark}{Remarks}
\sharedthm{example}{Example}{Examples}

\numberwithin{equation}{section}

\newcommand{\R}{\mathbb{R}}
\newcommand{\N}{\mathbb{N}}
\newcommand{\Z}{\mathbb{Z}}
\newcommand{\Zp}{\mathbb{Z}_{+}}
\newcommand{\Bn}{\mathbb{B}^n}
\newcommand{\Hn}{\mathbb{H}^n}
\newcommand{\Un}{\mathbb{U}^n}
\newcommand{\Sph}{\mathbb{S}^{n-1}}
\newcommand{\eps}{\varepsilon}
\newcommand{\avg}[2]{\langle #1\rangle_{#2}}
\newcommand{\Chi}{\mbox{\large$\chi$}}
\newcommand{\Jd}{\mathcal{J}}
\newcommand{\cells}{\mathcal{C}}
\newcommand{\Ab}{\mathbb{A}}
\newcommand{\Aloc}{A_{p,\mathrm{loc}}(\Bn)}
\newcommand{\Minf}{M^{\infty}}
\newcommand{\Mz}{M^{0}}
\newcommand{\Mcoarse}{M^{\sharp}}
\newcommand{\ann}{\mathcal{A}^{\circ}}
\newcommand{\gp}[2]{(#1|#2)_o}
\newcommand{\ang}[2]{\angle(#1,#2)}
\newcommand{\SCap}{\mathrm{Cap}}
\newcommand{\ip}[1]{\lfloor #1\rfloor}
\newcommand{\WP}{\eqref{eq:WP}\xspace}
\newcommand{\WQ}{\eqref{eq:WQ}\xspace}
\newcommand{\Ro}{\hyperref[eq:R1R2]{\textup{(R1)}}\xspace}
\newcommand{\Rt}{\hyperref[eq:R1R2]{\textup{(R2)}}\xspace}
\newcommand{\Rtw}{\hyperref[eq:R2w]{\textup{(R2w)}}\xspace}
\newcommand{\PP}{\mathfrak{P}}
\newcommand{\QQ}{\mathfrak{Q}}
\newcommand{\Sop}{\mathcal{S}}
\newcommand{\Hst}{\mathcal{H}^{*}}

\DeclareMathOperator{\artanh}{artanh}
\DeclareMathOperator*{\esssup}{ess\,sup}
\DeclareMathOperator*{\essinf}{ess\,inf}

\newcommand{\ww}{\mathsf w}

\title[Muckenhoupt weights via shadow slices on hyperbolic spaces]
{Muckenhoupt weights for the Hardy-Littlewood maximal operator on hyperbolic spaces}

\author[Rana]{Tapendu Rana}
\address{Department of Mathematics: Analysis, Logic and Discrete
Mathematics, Ghent University, Krijgslaan 281, Building S8, B 9000
Ghent, Belgium}
\email{tapendurana@gmail.com, tapendu.rana@ugent.be}

\subjclass[2020]{Primary 43A85, 42B25; Secondary  43A90, 43A15}
\keywords{Hardy--Littlewood maximal operator, Muckenhoupt weights,
weighted norm inequalities, two-layer weights, real hyperbolic space,
exponential volume growth, shadow slices}

\begin{document}

\begin{abstract}
On the real hyperbolic space $\mathbb{H}^n$, the classical Muckenhoupt condition does not describe the weights $w$ for which the Hardy--Littlewood maximal operator $M$ is bounded on $L^p(w)$.  We show that, for $1<p<\infty$, $M$ is bounded on $L^p(w)$ if and only if $w$ is a local $A_p$ weight and satisfies a two-layer condition, which compares $w$ and its dual weight on two layers of the shadow of a point.  We also treat the weak type case, radial weights and the range $p\le1$, and we prove a weighted Herz principle for radial convolution operators.
\end{abstract}

\maketitle

\section{Introduction}\label{sec:intro}

The study of weighted norm inequalities for the Hardy--Littlewood
maximal operator goes back to the works of Fefferman and Stein
\cite{FS71} and of Muckenhoupt \cite{Muc72}.  Muckenhoupt showed that on $\R^n$ the maximal operator is bounded on $L^p(w)$, $1<p<\infty$, exactly when the weight $w$ satisfies the $A_p$ condition, which compares the averages of $w$ and of its dual weight over one, and the same ball. This characterisation remains true on spaces of homogeneous type \cite{ST89}, and its proof relies on the doubling property of the underlying measure, through covering lemmas and the Calder\'on--Zygmund decomposition.

 On real hyperbolic spaces, and more generally on Riemannian symmetric spaces of noncompact type, the volume of balls grows exponentially, the measure is not doubling, and none of these tools is available.  Nevertheless, Clerc and Stein \cite{CS74} proved that the centered maximal operator $M$ is bounded on $L^p$, $1<p\le\infty$, by means of a convolution inequality on semisimple Lie groups, and Str\"omberg \cite{Str81} proved the weak type $(1,1)$ estimate; Anker, Damek and Yacoub \cite{ADY96} extended the latter to harmonic $NA$ groups.  Li and Lohou\'e \cite{LL12} obtained sharp bounds with respect to the dimension on real hyperbolic spaces, Ionescu \cite{Ion00}  showed that the  uncentered operator behaves differently: in rank one it is bounded on $L^p$ only for $p>2$.
 More recently, Chalmoukis, Meda and Santagati \cite{CMS26} obtained sharp $L^p$ bounds on Gromov hyperbolic spaces; see also \cite{CMM09,CMPS26} and, for trees and graphs,
\cite{ST19,LMSV25}. 

Weighted estimates in these exponential volume growth settings are much more recent. Ombrosi et al. \cite{ORR22,AO25} showed that in these settings the classical $A_p$ condition is not the right one: since most of the volume of a large ball lies near its boundary, the $A_p$ condition, which compares $w$ and $\sigma$ on the same ball, sees only its outer shell, whereas an average centred at an interior point sees how $w$ is distributed inside. A correct condition has to compare $w$ and $\sigma$ on different parts of the space. Antezana and Ombrosi \cite{AO25} gave sufficient conditions of two-set type on $\Hn$ and found that the weak and the strong type no longer coincide; these results were extended to harmonic $NA$ groups by Ganguly, Rana, and Sarkar \cite{GRS25}. However, no characterisation of the admissible weights is known, and whether the two-set conditions are also necessary was left open; for the weak type this was asked in \cite[Remark~1.4]{AO25}.

In this article, we characterise the weights $w$ for which the maximal operator $M$ is
bounded on $L^p(w)$, $1<p<\infty$, by a two-layer Muckenhoupt type condition defined through shadow slices (\cref{thmA});  and show in particular, the strong-type
two-set condition of \cite{AO25,GRS25} with $\alpha=1$ is not only
sufficient but also necessary.  For the weak type, we obtain sufficient conditions with partial converses (\cref{thmB}), and a characterisation for radial
weights (\cref{thmC}). Our approach is motivated by the infinite rooted $\mathsf k$-ary tree, a discrete analogue of $\Hn$ on which the weighted theory was developed in \cite{NT10,ORRS21,ORR22}, and on which such a characterisation was found very recently by Yang, Yuan and Zhang \cite{YYZ26}, by means of two layers of descendants of a common vertex.  Our class is the continuous analogue of theirs, but on $\Hn$ new phenomena appear, caused by the interaction of the exponential large scales with the Euclidean small scales (see \cref{subsec:new}).

Let us fix the notation.  Let $n\ge2$ and $\rho:=\frac{n-1}2$, let $d$ and $\mu$ be the distance and the volume of $\Hn$, $|E|:=\mu(E)$, and let $o$ be a base point, with $|x|:=d(o,x)$; a ball of radius $R\ge1$ has volume comparable to $e^{2\rho R}$.  We write $B_x:=B(x,1)$ for the unit ball about $x$.  A \emph{weight} is a positive locally integrable function $w$; we write $w(E):=\int_Ew\,d\mu$, $L^p(w):=L^p(\Hn,w\,d\mu)$ and $\sigma:=w^{-1/(p-1)}$ for $1<p<\infty$.  With $\avg fE:=|E|^{-1}\int_E|f|\,d\mu$ and $A_rf(x):=\avg f{B(x,r)}$,
\begin{equation}\label{def_class_Ap}
   Mf(x):=\sup_{r>0}A_rf(x),\qquad [w]_{A_p}:=\sup_B\,\avg wB\,
   \avg\sigma B^{p-1},
\end{equation}
the supremum being taken over all balls.  Since $\Hn$ behaves like $\R^n$ at small scales and exponentially at large ones, we split $M$, as in \cite{ADY96,AO25,GRS25}, into
\[
   M=\max\{\Mz,\Minf\},\qquad \Mz:=\sup_{0<r\le2}A_r,\qquad
   \Minf:=\sup_{r>2}A_r .
\]
We realise $\Hn$ as the Poincar\'e ball $\Bn$ with $o$ the origin (see \cref{sec:prelim} for more details), and say that $w\in\Aloc$ if \eqref{def_class_Ap} holds with the supremum restricted to balls of radius at most $1$ (for $p=1$: $\avg wB\le C\essinf_Bw$ for such balls).
Finally, for $\ell\in\N$ we use the \emph{annular averages}
\begin{equation}\label{eq:annular}
   \ann_\ell f(x):=\frac1{|B(x,\ell)|}\int_{\{y:\ \ell-1<d(x,y)\le
   \ell\}}|f(y)|\,d\mu(y).
\end{equation}
Since a fixed proportion of the volume of a large ball lies in its outermost unit shell, $\Minf f\lesssim\sup_\ell\ann_\ell f$ (see \cref{subsec_Aver_max_oper}): at large scales $M$ is an operator of averages over spheres, and the problem is to understand how a sphere about $x$ sees the weight.

\subsection{Shadow slices}\label{subsec:shadows}
Imagine that we stand at the origin $o$ of $\Bn$ and shine a light towards a point $z$.  A ball of bounded radius about $z$ casts a shadow, and we look at the part of the shell $\{y:\ |y|\approx|z|+m\}$ which lies in this shadow.  Formally, for $z\in\Bn$ and $m\in\Zp$ we define the \emph{shadow slice}
\begin{equation}\label{eq:shadow}
   \Omega_z(m):=\bigl\{y\in\Bn:\ m-1\le|y|-|z|\le m+1,\ \ d(z,y)\le
   m+s_0\bigr\},
\end{equation}

\begin{wrapfigure}[10]{r}{0.34\textwidth}
\vspace{-12pt}
\centering
\begin{tikzpicture}[scale=1.5,>=Stealth]
 \fill[orange!12] (0,0) -- (2:1) arc[start angle=2,end angle=38,radius=1] -- cycle;
 \draw[thick] (0,0) circle (1);
 \draw[gray!60,densely dashed] (0,0) circle (0.80);
 \draw[gray!60,densely dashed] (0,0) circle (0.955);
 \fill[orange!70] (2:0.80) arc[start angle=2,end angle=38,radius=0.80]
   -- (38:0.955) arc[start angle=38,end angle=2,radius=0.955] -- cycle;
 \draw[orange!80!black,densely dashed] (0,0) -- (2:1);
 \draw[orange!80!black,densely dashed] (0,0) -- (38:1);
 \draw[blue!60!black,fill=blue!10] (20:0.548) circle (0.168);
 \fill (0,0) circle (0.6pt) node[below left] {\small $o$};
 \fill[blue!70!black] (20:0.548) circle (0.6pt) node[below=1pt] {\small $z$};
 \node[orange!40!black] at (58:0.93) {\small $\Omega_z(m)$};
 \draw[->,orange!50!black] (58:0.86) -- (36:0.87);
\end{tikzpicture}
\par\vspace{2pt}
\refstepcounter{figure}\label{fig:shadow}
{\small\textsc{Figure}~\thefigure.\ The shadow slice $\Omega_z(m)$ in the
shadow of a ball about $ z$.}
\vspace{-6pt}
\end{wrapfigure}
\noindent where $s_0\ge6$ is a fixed constant; its value is immaterial (\cref{cor:slack}).  In the ball model, $\Omega_z(m)$ is, up to bounded changes of the constants, the part of a Euclidean annulus which lies in the Euclidean cone with vertex $0$, axis through $z$ and aperture comparable to $1-\|z\|$ (\cref{fig:shadow}, \cref{lem:shadow}), and $|\Omega_z(m)|\simeq e^{2\rho m}$.  On the tree, $\Omega_z(m)$ corresponds to the vertices $m$ levels below $z$ whose path to the root passes through $z$, the descendant layers of \cite{YYZ26}.

Why should such sets appear?  Let $x,y\in\Bn$ with $d(x,y)\approx\ell$, and let $z$ be the point of the segment $[0,x]$ at which the geodesics from $o$ to $x$ and to $y$ separate.  Then $x$ lies in a slice $\Omega_z(j)$ and $y$ in a slice $\Omega_z(m)$ with $j+m\approx\ell$ (\cref{fig:decomp}): as on a tree, the way from $x$ to $y$ leads back towards $o$ and then outwards again.  Hence, the average $\ann_\ell f(x)$ sees $w$ on one layer in the shadow of $z$ and $f$ on another layer of the same shadow, and a condition characterising $M$ has to compare $w$ and $\sigma$ on two \emph{different} layers in the shadow of the \emph{same} point.  This leads to the following class.

\subsection{Main results}\label{subsec:main}
\begin{definition}[Two-layer Muckenhoupt class]\label{def:class}
Let $p\in(1,\infty)$ and $\eps>0$.  A weight $w$ on $\Bn$ belongs to the class $\Ab_{p,\eps}$ if
\[
   [w]_{\Ab_{p,\eps}}:=\sup_{z\in\Bn}\ \sup_{j,m\in\Zp}\
   \frac{w\bigl(\Omega_z(j)\bigr)\,\sigma\bigl(\Omega_z(m)\bigr)^{p-1}}
        {e^{2\rho(p-\eps)(j+m)}}<\infty ,
\]
and we put $\Ab_p=\Ab_p(\Bn):=\bigcup_{\eps>0}\Ab_{p,\eps}$.
\end{definition}

On a doubling space, the $A_p$ condition is equivalent to a two-ball condition, which bounds $w(B(x,r))\,\sigma(B(x,s))^{p-1}$ by a multiple of $|B(x,r+s)|^{p}$ \cite[Remark~1.2]{YYZ26}.  \Cref{def:class} is this condition with the two balls about $x$ replaced by two slices in the shadow of $z$, with $|B(x,r+s)|$ replaced by $|\Omega_z(j+m)|\simeq e^{2\rho(j+m)}$, and with the additional gain $|\Omega_z(j+m)|^{-\eps}$.  Since $w(\Omega_z(i))\,\sigma(\Omega_z(i))^{p-1}\ge|\Omega_z(i)|^p$ by H\"older's inequality, this gain is admissible exactly for $\eps\le p/2$ (\cref{prop:basic}).  The class involves only the metric and the measure, and although the slices are seen from $o$, it does not depend on the base point (\cref{rem:invariance}). Equipped with this definition, our first main result characterises the strong type boundedness for general weights.

\begin{maintheorem}[General weights, strong type]\label{thmA}
Let $p\in(1,\infty)$ and let $w$ be a weight on $\Bn$.  The following are equivalent.
\begin{enumerate}[label=\textup{(\roman*)}, leftmargin=2.3em]
\item\label{thmA:i} $M$ is bounded on $L^p(w)$.
\item\label{thmA:ii} $w\in\Aloc\cap\Ab_p$.
\item\label{thmA:iii} $w\in\Aloc$ and there are $\eps,C>0$ such that      $\|\ann_\ell f\|_{L^p(w)}\le Ce^{-\eps\ell}\|f\|_{L^p(w)}$ for all   $\ell\in\N$.
\item\label{thmA:iv} \textup{(Global Sawyer testing)} $\|M(\sigma\Chi_B)\|^p_{L^p(w)} \lesssim\sigma(B)$ for every ball $B\subset\Bn$.
\item\label{thmA:v} $w\in\Aloc$ and there is $\beta\in(0,1)$ such that, for all
      measurable $E,F\subset\Bn$ and all $r\ge1$,
      \begin{equation}\label{eq:AOcond}
         \int_FA_r(\Chi_E)\,w\,d\mu\lesssim e^{2\rho(\beta-1)r}\,
         w(E)^{1/p}\,w(F)^{1-1/p}.
      \end{equation}
\end{enumerate}
Moreover, $\Minf$ alone is bounded on $L^p(w)$ if and only if $w\in\Ab_p$, and for every weight $w$ and every $\eps\in(0,p/2]$,
\begin{equation}\label{eq:buckley-intro}
   \|\Minf f\|_{L^p(w)}\le C(n,p)\,\eps^{-2}\,[w]_{\Ab_{p,\eps}}^{1/p}\,
   \|f\|_{L^p(w)},
\end{equation}
for each fixed $\eps$ the exponent $1/p$ is optimal.  Finally, $\Aloc\cap\Ab_p$ is open: if $M$ is bounded on $L^p(w)$, then it is bounded on $L^q(w)$ for all $q\in[p-\eta_0,\infty)$, for some $\eta_0>0$.
\end{maintheorem}
Condition \ref{thmA:iv} is the analogue of Sawyer's testing condition for the two-weight maximal inequality \cite{Saw82}, with $\sigma$ as the second weight; as usual, if $\sigma(B)=\infty$ it is understood with $\sigma$ replaced by $\min\{\sigma,N\}$, uniformly in $N\in\N$.
As on the tree \cite[Theorem~4.3]{ORR22}, the global integral cannot be replaced by the local one $\int_BM(\sigma\Chi_B)^pw\,d\mu$ (\cref{prop:localsawyer}).  Condition \ref{thmA:v} is the strong-type condition of \cite[Theorem~1.2]{AO25} and \cite[Theorem~1.2]{GRS25} with $\alpha=1$; thus that sufficient condition is also necessary.
\begin{remark}
Neither of the two simpler conditions suggested by the classical theory is adequate.  \emph{One slice is not enough:} for $w=e^{\alpha|x|}$ the one-slice product $\avg w{\Omega_z(i)}\avg\sigma{\Omega_z(i)}^{p-1}$ is bounded for \emph{every} $\alpha\in\R$ (\cref{rem:onewindow}), whereas $M$ is bounded on $L^p(w)$ only for $\alpha\in(-2\rho,2\rho(p-1))$ (\cref{thmC}).  \emph{Balls are not the right sets:} for the same weights the ball condition \eqref{def_class_Ap} fails when $\alpha\in(-2\rho,-\rho)\cup(\rho(p-1),2\rho(p-1))$, although $M$ is bounded (\cref{rem:ballAp} and \cite[Example~4.1]{AO25}).  Indeed, a ball $B(x,R)$ with $|x|=2R$ contains a unit ball about a point $z$ with $|z|\approx R$ and a part of the shadow of $z$ at depth $\approx2R$; the ball condition bounds the product of $w$ and $\sigma$ on it by $e^{2\rho pR}$, whereas \cref{def:class} allows $e^{4\rho(p-\eps)R}$ for two such layers.
\end{remark}

For the weak type we use two conditions.  For $1\le p<\infty$, $z\in\Bn$ and $m\in\Zp$ let
\begin{equation}\label{eq:alpha-intro}
   \PP_m[w](z):=\begin{cases}e^{-2\rho pm}\,w(B_z)\,
   \sigma\bigl(\Omega_z(m)\bigr)^{p-1}, & p>1,\\
   e^{-2\rho m}\,w(B_z)\,
   \esssup_{\Omega_z(m)}w^{-1}, & p=1,\end{cases}
\end{equation}
and consider the conditions
\begin{align}
   [w]_{\mathcal W_Q}&:=\sup_{y\in\Bn}\ \sup_{N\in\Zp}\
   \frac{e^{-2\rho pN}}{w(B_y)}\sum_{j=0}^{N}w\bigl(\Omega_y(j)
   \bigr)<\infty, \tag{$\mathcal W_Q$}\label{eq:WQ}\\
   \PP_m[w](z)&\le Ce^{-2\rho p\eps m}\ \text{ for all }z\in\Bn,\
   m\in\Zp,\ \text{for some }C,\eps>0. \tag{$\mathcal W_P$}\label{eq:WP}
\end{align}

\begin{maintheorem}[General weights, weak type]\label{thmB}
Let $p\in[1,\infty)$ and let $w$ be a weight on $\Bn$.
\begin{enumerate}[label=\textup{(\alph*)}, leftmargin=2.3em]
\item\label{thmB:a} If $w\in A_{p,\mathrm{loc}}(\Bn)$ and $w$ satisfies \WQ and   \WP, then $M:L^p(w)\to L^{p,\infty}(w)$ is bounded.
\item\label{thmB:b} Conversely, if $M:L^p(w)\to L^{p,\infty}(w)$ is bounded, then $w\in A_{p,\mathrm{loc}}(\Bn)$, $w$ satisfies \WQ, and      $\sup_{z,m}\PP_m[w](z)<\infty$.
\item\label{thmB:c} If $p>1$ and $M$ is bounded on $L^p(w)$, then \WP holds.  If $w$ is radial and $M:L^p(w)\to L^{p,\infty}(w)$ is bounded, then \WP holds.
\end{enumerate}
\end{maintheorem}
Whether \WP itself is necessary for the weak type is open for general weights.  For radial weights $w=w_0(|\cdot|)$, the two layers collapse to a single interval.
We write $W_N:=w(\{x:\ N\le|x|<N+1\})$, $N\in\Zp$, and $dV(t)= \Jd(t)\,dt$, $\Jd(t)=\omega_{n-1}\sinh^{n-1}t$, for the radial volume measure on $\R_+=[0,\infty)$.  For $1<p<\infty$ we say that $w_0\in A_p(\R_+,V)$ if
\begin{equation}\label{eq:ApV}
   [w_0]_{A_p(V)}:=\sup_{I}\Bigl(\frac1{V(I)}\int_Iw_0\,dV\Bigr)
   \Bigl(\frac1{V(I)}\int_I\sigma_0\,dV\Bigr)^{p-1}<\infty,
   \qquad\sigma_0:=w_0^{-1/(p-1)},
\end{equation}
the supremum being taken over all compact intervals $I\subset\R_+$, and for $1\le p<\infty$ we put
\begin{equation}\label{eq:Aprw}
   [w_0]_{A^{\mathrm{KT}}_p(V)}:=\sup\Bigl\{\Bigl(\frac{V(I)}{V(J)}
   \Bigr)^p\,\frac{\int_Jw_0\,dV}{\int_Iw_0\,dV}:\ I\subset J\
   \text{compact intervals},\ |I|\ge1\Bigr\}.
\end{equation}
Our basic examples are the radial weights
\begin{equation}\label{eq:examples}
   \ww_{\alpha,\gamma}(x):=e^{\alpha|x|}(1+|x|)^{\gamma},\qquad
   \ww^\circ_{\alpha,\gamma}(x):=e^{\alpha|x|}|x|^{\gamma},\qquad
   \alpha,\gamma\in\R .
\end{equation}

\begin{maintheorem}[Radial weights]\label{thmC}
Let $w=w_0(|\cdot|)$ be a radial weight on $\Bn$.
\begin{enumerate}[label=\textup{(\alph*)}, leftmargin=2.3em]
\item\label{thmC:a} For $p\in(1,\infty)$ the following are equivalent:
      \begin{enumerate}[label=\textup{(\roman*)}, leftmargin=2em]
      \item\label{thmC:a:i} $M$ is bounded on $L^p(w)$;
      \item\label{thmC:a:ii} $M$ is bounded on the subspace of radial functions of    $L^p(w)$;
      \item\label{thmC:a:iii} $w\in\Aloc$ and
      \begin{equation}\label{eq:R1R2}
        \textup{(R1)}\ \ \sup_{N}\frac1{W_N}\sum_{i=0}^{N}W_i<\infty,
        \qquad
        \textup{(R2)}\ \ \sup_{N}\frac1{W_N}\sum_{i\ge0}
        e^{-2\rho p i}\,W_{N+i}<\infty ;
      \end{equation}
      \item\label{thmC:a:iv} $w_0\in A_p(\R_+,V)$.
      \end{enumerate}
\item\label{thmC:b} For $p\in[1,\infty)$ the following are equivalent:
      \begin{enumerate}[label=\textup{(\roman*)}, leftmargin=2em]
      \item\label{thmC:b:i} $M:L^p(w)\to L^{p,\infty}(w)$ is bounded;
      \item\label{thmC:b:ii} $w\in A_{p,\mathrm{loc}}(\Bn)$, \Ro holds, and
      \begin{equation}\label{eq:R2w}
        \textup{(R2w)}\ \ \sup_{N,i}e^{-2\rho pi}\,W_{N+i}/W_N<\infty;
      \end{equation}
      \item\label{thmC:b:iii} $w\in A_{p,\mathrm{loc}}(\Bn)$ and
            $[w_0]_{A^{\mathrm{KT}}_p(V)}<\infty$.
      \end{enumerate}
\item\label{thmC:c} Let $\alpha,\gamma\in\R$.  For $p>1$, $M$ is bounded      on $L^p(\ww_{\alpha,\gamma})$ iff $\alpha\in(-2\rho,2\rho(p-1))$, and on $L^p(\ww^\circ_{\alpha,\gamma})$ iff moreover      $\gamma\in(-n,n(p-1))$.  For $p\ge1$, $M$ is of weak type $(p,p)$ with respect to $\ww_{\alpha,\gamma}$ iff either       $\alpha\in(-2\rho,2\rho(p-1))$, or $\alpha=2\rho(p-1)$ and
      $\gamma\le0$; with respect to $\ww^\circ_{\alpha,\gamma}$ the same
      holds with the additional requirement $\gamma\in(-n,n(p-1))$ for
      $p>1$, respectively $\gamma\in(-n,0]$ for $p=1$.
\end{enumerate}
\end{maintheorem}
Condition \Ro says that most of the mass of a ball centered at $o$ lies in its outer shell, and \Rt that the shell masses grow strictly slower than $e^{2\rho pN}$.  Condition \ref{thmC:a}\ref{thmC:a:iv} is Muckenhoupt's condition \cite{Muc72} for the non-doubling Jacobi measure $V$, so that \ref{thmC:a} is a Muckenhoupt theorem for the Jacobi hypergroup \cite{FJK73,Koo84}, and \ref{thmC:b}\ref{thmC:b:iii} is an interval analogue of the condition of Kerman and Torchinsky \cite{KT82}; the case $p=1$, $w\equiv1$ of \ref{thmC:b} is the weak type $(1,1)$ of \cite{Str81,ADY96}.

For $p\le1$, we use the class $\Ab_1=\bigcup_{\eps>0}\Ab_{1,\eps}$ of weights with
\[
   [w]_{\Ab_{1,\eps}}:=\sup_{z\in\Bn}\ \sup_{j,m\in\Zp}\
   \frac{w\bigl(\Omega_z(j)\bigr)\,\|w^{-1}\|_{L^\infty(\Omega_z(m))}}
        {e^{2\rho(1-\eps)(j+m)}}<\infty .
\]

\begin{maintheorem}[The range $0<p\le1$]\label{thmD}
\begin{enumerate}[label=\textup{(\alph*)}, leftmargin=2.3em]
\item\label{thmD:a} If $0<p<1$, there is no weight $w$ such that
      $\Minf:L^p(w)\to L^{p,\infty}(w)$ is bounded; a fortiori the same
      holds for $M$.
\item\label{thmD:b} There is no weight $w$ such that $M$ is bounded on $L^1(w)$.
\item\label{thmD:c} For every weight $w$ the following are equivalent:
      \begin{enumerate}[label=\textup{(\roman*)}, leftmargin=2em]
      \item\label{thmD:c:i} $\Minf$ is bounded on $L^1(w)$;
      \item\label{thmD:c:ii} $\int_{\Bn}e^{-2\rho d(x,y)}w(y)\,d\mu(y)\le Cw(x)$ for a.e.\
            $x\in\Bn$;
      \item\label{thmD:c:iii} there is $\eps>0$ with
            $\int_{\Bn}e^{-(2\rho-\eps)d(x,y)}w(y)\,d\mu(y)\le Cw(x)$ for
            a.e.\ $x\in\Bn$;
      \item\label{thmD:c:iv} $\|\ann_\ell\|_{L^1(w)\to L^1(w)}\le Ce^{-\eps\ell}$ for some
            $\eps>0$ and all $\ell\in\N$;
      \item\label{thmD:c:v} $w\in\Ab_1$.
      \end{enumerate}
\end{enumerate}
In particular, $\Minf$ is bounded on $L^1(\ww_{\alpha,\gamma})$ if and only if $-2\rho<\alpha<0$, and on $L^1(\ww^\circ_{\alpha,\gamma})$ if and only if, in addition, $-n<\gamma\le0$.
\end{maintheorem}

\begin{remark}
    On the tree, admissible weights exist exactly for $p>\frac12$ \cite[Theorem~1.3]{YYZ26}; on $\Bn$ the critical index is $1$, for two reasons: for $p<1$, $L^p(w)$ contains functions which are not locally integrable, and at $p=1$, a unit ball is Euclidean.  The index $\frac12$ reappears for functions which are constant on cells of unit size
(\cref{rem:coarse}).
\end{remark}

Our last result uses the Helgason--Fourier analysis of $\Hn$.  Herz's \emph{principe de majoration} \cite{Her70} states that a convolution operator with a positive radial kernel on a semisimple Lie group is bounded on $L^2$ if and only if its kernel is integrable against the elementary spherical function $\varphi_0$, the norm being this integral; it lies behind the Kunze--Stein phenomenon \cite{KS60,Cow78}.
We show that the $L^2$-norm is a lower bound on every weighted $L^p$ space and identify the weights for which it is attained.  For a radial kernel $K:[0,\infty)\to[0,\infty)$ let $T_Kf(x):=\int_{\Bn}K(d(x,y)) f(y)\,d\mu(y)$, let $\widehat K(\lambda):=\int_{\Bn}K(|x|) \varphi_\lambda(x)\,d\mu(x)$ be the spherical transform of $K(|\cdot|)$, where $\varphi_\lambda$ is Harish-Chandra's elementary spherical function (\cref{subsec:phi0}), and let $P(x,\xi)$ be the Poisson kernel of $\Bn$.

\begin{maintheorem}[Weighted Herz principle and Poisson weights]
\label{thmE}
Let $1\le p<\infty$ and let $K\ge0$ be a radial kernel.
\begin{enumerate}[label=\textup{(\alph*)}, leftmargin=2.3em]
\item\label{thmE:a} For every weight $w$,
      $\|T_K\|_{L^p(w)\to L^p(w)}\ge\|T_K\|_{L^2\to L^2}=\widehat K(0)$, with equality for $w=\varphi_0^{2-p}$.
\item\label{thmE:b} Let $\xi\in\Sph$, $s\in\R$ and $\eta:=\frac{2\rho(1-s)}p-\rho$.
      Then $\|T_K\|_{L^p(P(\cdot,\xi)^s)\to L^p(P(\cdot,\xi)^s)}= \widehat K(i\eta)$.  For $p>1$, $M$ is bounded on $L^p(P(\cdot,\xi)^s)$ if and only if $-(p-1)<s<1$, and $\Minf$ is bounded on $L^1(P(\cdot,\xi)^s)$ if and only if $0<s<1$.
\end{enumerate}
\end{maintheorem}
The weights $P(\cdot,\xi)^s$ are constant on horospheres, and they are our explicit non-radial examples; for $s=0$, \ref{thmE:b} is \cite[Theorem~3.3]{ADY96}.

\subsection{Ideas of the proofs and new ingredients}\label{subsec:new}
For $\ell  \in \N$, let $\ell-1<d(x,y)\le\ell$ and let $z$ be as in \cref{subsec:shadows}, at distance $j$ from $x$.  Then $x\in\Omega_z(j)$ and $y\in\Omega_z(m)$ with $\ell-3\le j+m\le\ell$ (\cref{lem:decomp}), so that $\ann_\ell f(x)$ is bounded by a sum of $O(\ell)$ averages of $f$ over slices $\Omega_z(m)$ with $j+m\approx\ell$.  By H\"older's inequality, the norm of each of them on $L^p(w)$ is at most a constant times
$\bigl(e^{-2\rho p(j+m)}\sup_{z}w(\Omega_z(j))\,
\sigma(\Omega_z(m))^{p-1}\bigr)^{1/p}$ (\cref{lem:normbounds}), which is $\lesssim[w]_{\Ab_{p,\eps}}^{1/p}e^{-2\rho\eps(j+m)/p}$ for $w\in\Ab_{p,\eps}$.  Summing over the pairs $(j,m)$ and over $\ell$ gives the boundedness of $\Minf$, since $\Minf f\lesssim\sum_\ell \ann_\ell f$.  For the necessity, we follow \cite{YYZ26} and consider
the descendant and the ancestor averages
\[
   P_mf(x):=\avg f{\Omega_x(m)},\qquad
   Q_jf(x):=e^{-2\rho j}\avg f{B(a_j(x),\frac12)},
\]
where $a_j(x)$ is the point of $[0,x]$ at distance $j$ from $x$. Testing $\Minf$ on the functions $\sigma\Chi_B$, we show that the norms of $P_m$ and the tails of the ancestor sums decay exponentially (\cref{prop:Pgeom,lem:Qtail}), and the two decays are combined by H\"older's inequality on a unit ball (\cref{prop:necessity}).  Compared with the tree, three new ingredients are needed.

\emph{There is no exact semigroup law.}  On the tree the descendant and ancestor averages satisfy $P_kP_m=P_{k+m}$ and $Q_jQ_i=Q_{j+i}$, and these identities turn summability into exponential decay.  On $\Bn$, ancestors of nearby points are only close to each other, and such identities hold only after enlarging the sets.  We prove a one-sided semigroup inequality with the enlargement on the outer factor only (\cref{lem:semigroup}), and replace the second identity by two self-improvement arguments: a covering and averaging argument for the ancestor tails (\cref{lem:Qtail}), and, at $p=1$, a binomial lower bound for convolution powers (\cref{lem:convpowers}).

\emph{Points have measure zero.}  Arguments which on the tree are applied to single vertices have to be applied to unit balls.  The two layers are linked by H\"older's inequality on a unit ball (\cref{prop:necessity}), and the openness of the class requires a reverse H\"older inequality on unit balls; without the local condition $\Ab_p$ is not open (\cref{rem:openness}).

\emph{Local and global scales interact.}  $\Mz$ is governed by the Euclidean local $A_p$ condition and $\Minf$ by $\Ab_p$, and the two conditions are independent: there are radial weights in $\Ab_p$ for which $M$ is unbounded (\cref{ex:loc}).  The local scale also explains why no weight is admissible for $p\le1$ (\cref{thmD}), and why the $\ell^1$-valued Fefferman--Stein inequality fails for $M$ (\cref{thm:FS}), unlike on the tree \cite[Theorem~5.5]{YYZ26}.

\subsection{Organization}\label{subsec:strategy}
\Cref{sec:prelim} collects preliminaries.  \Cref{sec:geometry} studies shadow slices and decomposes spheres into pairs of slices, and \cref{sec:directional} contains the directional operators, the one-sided semigroup inequality, the norm bounds and the basic properties of $\Ab_p$.  \Cref{thmA,thmB,thmC,thmD} are proved in \cref{sec:thmA,sec:weak,sec:radial,sec:small}, and \cref{thmE} in \cref{sec:applications}, together with the vector-valued inequalities and the weighted Herz principle; \cref{sec:remarks} lists open questions.

\section{Preliminaries}\label{sec:prelim}

The geometric facts below are classical and can be found, for instance, in \cite{Bea83,Rat19,BH99,Sto16}; the few which we could not locate in the form needed here are proved in \cref{subsec:cosines}.
\subsection*{Notations}  We write $\Zp=\N\cup\{0\}$, and use $c$, $C$, $C_1$, etc., for positive constants whose value may change from line to line. $X\lesssim Y$ means $X\le CY$ with $C$ depending only on $n$, $s_0$ and $p$, unless indicated otherwise, and $X\simeq Y$ means $X\lesssim Y\lesssim X$.  For $r\ge0$, $\ip r$ is the integer part of $r$, and $r_+:=\max\{r,0\}$ for $r\in\R$.  For $x\in\R^n$, $\|x\|$ is its Euclidean norm, while $|x|:=d(o,x)$ and $B_x=B(x,1)$ for $x\in\Bn$. For an operator $T$, $\|T\|$ denotes its norm on $L^p(w)$ when $p$ and $w$ are clear from the context.  A set $X\subset\Bn$ is \emph{$\delta$-separated} if $d(x,x')\ge\delta$ for all distinct $x,x'\in X$.  A maximal $\delta$-separated set $X$ is countable, the balls $B(x,\delta)$, $x\in X$, cover $\Bn$, and, since the balls $B(x,\frac\delta2)$ are disjoint, the balls $B(x,R)$, $x\in X$, have bounded overlap for every $R>0$.

\subsection{The Poincar{\'e} ball model}\label{subsec:ball}
Let $n\ge2$ and let $\Bn:=\{x\in\R^n:\ \|x\|<1\}$ be the unit ball, endowed with the Riemannian metric
\[
   ds^2=\frac{4\,\|dx\|^2}{(1-\|x\|^2)^2}.
\]
This is the Poincar\'e ball model of the real hyperbolic space $\Hn$ \cite[\S4.5]{Rat19}.  The distance and the volume measure are given by
\begin{equation}\label{eq:distance}
   \cosh d(x,y)=1+\frac{2\|x-y\|^2}{(1-\|x\|^2)(1-\|y\|^2)},
   \qquad
   d\mu(x)=\frac{2^n\,dx}{(1-\|x\|^2)^n},
\end{equation}
where $dx$ denotes the Lebesgue measure \cite[\S4.5]{Rat19}, \cite[Ch.~2]{Sto16}.  We denote by $o$ the origin, and we write $|x|:=d(o,x)$; by \eqref{eq:distance},
\begin{equation}\label{eq:radius}
   |x|=\log\frac{1+\|x\|}{1-\|x\|},\qquad\text{that is,}\qquad
   \|x\|=\tanh\frac{|x|}2 .
\end{equation}
The group of isometries of $\Bn$ is the group of M\"obius transformations of $\Bn$ \cite[\S4.5]{Rat19}, \cite[Ch.~3]{Bea83}; it acts transitively, and the stabilizer of $o$ is the orthogonal group $O(n)$.  Every hyperbolic ball is a Euclidean ball, and on $B(o,L)$ the hyperbolic distance and volume are comparable to the Euclidean ones, with constants depending only on $L$.

For $\theta\in\Sph$ the geodesic ray from $o$ in the direction $\theta$ is the Euclidean radius
\begin{equation}\label{eq:geodesic}
   \gamma_\theta(t):=\tanh\Bigl(\frac t2\Bigr)\,\theta ,\qquad t\ge0,
\end{equation}
parametrized by arc length, so that $|\gamma_\theta(t)|=t$.  For $x\ne o$ we write $\theta_x:=x/\|x\|\in\Sph$, and $\gamma_x:=\gamma_{\theta_x}$ restricted to $[0,|x|]$ is the unit speed geodesic from $o$ to $x$; it is the Euclidean segment $[0,x]$.  For $x,y\ne o$ we denote by $\ang xy\in[0,\pi]$ the Euclidean angle between the vectors $x$ and $y$; since the metric is conformal, it is also the Riemannian angle at $o$ of the geodesic triangle with vertices $o,x,y$.  For $\theta\in\Sph$ and $0<\phi\le\pi$ we denote by
$\SCap(\theta,\phi):=\{\theta'\in\Sph:\ \angle(\theta,\theta')<\phi\}$
the open spherical cap, and by $\nu$ the rotation-invariant probability measure on $\Sph$.  Then
\begin{equation}\label{eq:caps}
   \nu\bigl(\SCap(\theta,\phi)\bigr)=\frac{\omega_{n-2}}{\omega_{n-1}}
   \int_0^\phi\sin^{n-2}u\,du\ \simeq\ \phi^{\,n-1}=\phi^{2\rho},
   \qquad 0<\phi\le\pi,
\end{equation}
where $\omega_{k}$ is the area of the unit sphere of $\R^{k+1}$.

Writing $x=\gamma_\theta(t)=\tanh(t/2)\theta$ and using \eqref{eq:distance}, one obtains the polar coordinate formula \cite[Ch.~I, Theorem~5.8]{Hel84}, \cite[Theorem~4.5.6]{Rat19}
\begin{equation}\label{eq:polar}
   \int_{\Bn}F\,d\mu=\int_0^\infty\!\!\int_{\Sph}F\bigl(\gamma_\theta(t)
   \bigr)\,d\nu(\theta)\,\Jd(t)\,dt ,\qquad
   \Jd(t):=\omega_{n-1}\sinh^{n-1}t ,
\end{equation}
for all integrable functions $F$.  The density $\Jd$ is increasing and
\begin{equation}\label{eq:density}
   \Jd(t)\simeq t^{\,n-1}\ (0<t\le1),\qquad
   \Jd(t)\simeq e^{2\rho t}\ (t\ge1).
\end{equation}
We write $V(I):=\int_I\Jd(t)\,dt$ for intervals $I\subset\R_+$.  Since the isometries act transitively, $|B(x,r)|=|B(o,r)|=V([0,r])$, and so
\begin{equation}\label{eq:ballvol}
   |B(x,r)|\simeq r^{n}\quad(0<r<1),\qquad
   |B(x,r)|\simeq e^{2\rho r}\quad(r\ge1).
\end{equation}
A function $F$ on $\Bn$ is \emph{radial} if $F(x)=F_0(|x|)$ for a function $F_0$ on $\R_+$; then $\int_{\Bn}F\,d\mu=\int_0^\infty F_0\,dV$.

\subsection{Hyperbolic distance estimates}\label{subsec:cosines}
The basic tool is the hyperbolic law of cosines \cite[\S7.12]{Bea83}, \cite[\S3.5]{Rat19}: for $x,y\ne o$,
\begin{align}
   \cosh d(x,y)&=\cosh|x|\cosh|y|-\sinh|x|\sinh|y|\cos\ang xy
   \label{eq:cosines}\\
   &=\cosh\bigl(|x|-|y|\bigr)+2\sinh|x|\sinh|y|\,
   \sin^2\frac{\ang xy}2 . \label{eq:cosines2}
\end{align}
For $x,y\in\Bn$ the Gromov product with respect to $o$ is
\begin{equation}\label{eq_gro_pro}
    \gp xy:=\frac12\bigl(|x|+|y|-d(x,y)\bigr),
\end{equation}
and $0\le\gp xy\le\min\{|x|,|y|\}$.  It is, up to additive constants, the time at which the geodesics from $o$ to $x$ and to $y$ separate, and the next lemma shows that this time is measured by the angle at $o$; on general Gromov hyperbolic spaces the corresponding facts are expressed by visual metrics \cite[III.H.3]{BH99}.

\begin{lemma}\label{lem:angle}
Let $x,y\in\Bn\setminus\{o\}$ and put $\phi:=\ang xy$.
\begin{enumerate}[label=\textup{(\alph*)}, leftmargin=2.3em]
\item\label{lem:angle:a} $\displaystyle\frac1{81}\bigl(e^{-2\min\{|x|,|y|\}}+\phi^2\bigr)
      \le e^{-2\gp xy}\le2\bigl(e^{-2\min\{|x|,|y|\}}+\phi^2\bigr)$.
\item\label{lem:angle:b} For every $s\ge0$,
      $\sinh\bigl(\frac12d(\gamma_{\theta_x}(s),\gamma_{\theta_y}(s))\bigr)
      =\sinh s\,\sin\frac\phi2$.  In particular, $s\mapsto
      d(\gamma_{\theta_x}(s),\gamma_{\theta_y}(s))$ is increasing, and
      \begin{equation}\label{eq:fellow}
        d\bigl(\gamma_x(s),\gamma_y(s)\bigr)\le2\qquad\text{whenever }
        0\le s\le\gp xy .
      \end{equation}
\item\label{lem:angle:c} If $0\le s\le\min\{|x|,|y|\}$ and $\phi\le e^{-s}$, then
      $\gp xy\ge s-\log2$.  If $\gp xy\ge s$ for some $s\in\R$, then
      $\phi\le9e^{-s}$.
\end{enumerate}
\end{lemma}

\begin{proof}
We put $a:=\min\{|x|,|y|\}$ and $b:=\max\{|x|,|y|\}$, so that
$|x|+|y|=a+b$ and $e^{-2\gp xy}=e^{d(x,y)-a-b}$.

(a) By \eqref{eq:cosines2}, together with $4\sinh|x|\sinh|y|\le e^{a+b}$ and $\sin(\phi/2)\le\phi/2$,
\[
   e^{d(x,y)}\le2\cosh d(x,y)\le2e^{b-a}+e^{a+b}\frac{\phi^2}4 ,
\]
and multiplying by $e^{-a-b}$ gives the upper bound.  For the lower bound, suppose first that $a\ge1$.  Then $\sinh t\ge\frac{1-e^{-2}}2e^t$
for $t\ge a$, and $\sin(\phi/2)\ge\phi/\pi$, so that, by
\eqref{eq:cosines2},
\[
   e^{d(x,y)}\ge\cosh d(x,y)\ge\frac12e^{b-a}+\frac{(1-e^{-2})^2}{2\pi^2}
   e^{a+b}\phi^2\ge\frac1{30}\bigl(e^{b-a}+e^{a+b}\phi^2\bigr),
\]
and the lower bound follows.  If $a<1$, then $\gp xy\le a<1$, so that
$e^{-2\gp xy}>e^{-2}$, while $e^{-2a}+\phi^2\le1+\pi^2$; hence
$e^{-2\gp xy}\ge\frac{1}{(1+\pi^2)e^2}(e^{-2a}+\phi^2)\ge\frac1{81}
(e^{-2a}+\phi^2)$, since $(1+\pi^2)e^2<81$.

(b) The points $\gamma_{\theta_x}(s)$ and $\gamma_{\theta_y}(s)$ lie at the same distance $s$ from $o$ and make the angle $\phi$ at $o$, so that the first assertion is \eqref{eq:cosines2} applied to them, together with $\cosh u=1+2\sinh^2(u/2)$.  Let now $0\le s\le\gp xy$.
If $a<1$, then $s\le\gp xy\le a<1$, and $d(\gamma_x(s),\gamma_y(s))\le 2s<2$ by the triangle inequality through $o$.  If $a\ge1$, we use $\sinh s\le e^s/2$ and, by \eqref{eq:cosines2},
$e^{d(x,y)}\ge\cosh d(x,y)\ge2\sinh a\sinh b\,\sin^2(\phi/2)$; hence
\[
   \sinh^2s\,\sin^2\frac\phi2\le\frac{e^{2\gp xy}}4\sin^2\frac\phi2
   =\frac{e^{a+b}}{4e^{d(x,y)}}\sin^2\frac\phi2\le
   \frac{e^{a+b}}{8\sinh a\sinh b}\le\frac1{2(1-e^{-2})^2}<1 .
\]
Thus $\sinh(\frac12d(\gamma_x(s),\gamma_y(s)))<1<\sinh1$, which is
\eqref{eq:fellow}.

(c) If $s\le a$ and $\phi\le e^{-s}$, the upper bound in (a) gives
$e^{-2\gp xy}\le2(e^{-2s}+e^{-2s})=4e^{-2s}$, that is, $\gp xy\ge
s-\log2$.  If $\gp xy\ge s$, the lower bound in (a) gives
$\phi^2\le81e^{-2\gp xy}\le81e^{-2s}$.
\end{proof}

Inequality \eqref{eq:fellow} is the fellow traveller property
\cite[III.H.1]{BH99}.  The next lemma makes \cref{fig:shadow} precise.

\begin{lemma}[Shadow of a ball]\label{lem:shadowball}
Let $x\in\Bn$ and $0<r<|x|$.  The set of directions
$\{\theta_y:\ y\in B(x,r)\}$ is the cap $\SCap(\theta_x,\phi)$, where
$\phi\in(0,\frac\pi2)$ is given by $\sin\phi=\sinh r/\sinh|x|$.  In
particular, if $|x|\ge r+1$, then $\phi\simeq_r e^{-|x|}$.
\end{lemma}

\begin{proof}
We fix a direction $\theta$ with $\angle(\theta,\theta_x)=\psi$.  By
\eqref{eq:cosines}, the ray $\gamma_\theta$ meets $B(x,r)$ if and only
if $\inf_{t\ge0}f(t)<\cosh r$, where
$f(t):=\cosh|x|\cosh t-\sinh|x|\cos\psi\,\sinh t$.  If $\psi\ge\pi/2$,
then $f(t)\ge\cosh|x|>\cosh r$.  If $\psi<\pi/2$, then
$A:=\cosh|x|>B:=\sinh|x|\cos\psi\ge0$, and the minimum of $A\cosh t-
B\sinh t$ over $t\in\R$ is $\sqrt{A^2-B^2}$, attained at
$t=\artanh(B/A)\ge0$; hence $\inf_{t\ge0}f(t)^2=1+\sinh^2|x|\sin^2\psi$.
Thus $\gamma_\theta$ meets $B(x,r)$ if and only if
$\sin\psi<\sinh r/\sinh|x|$, which is the first assertion.  The second
follows from $\sinh|x|\simeq e^{|x|}$ for $|x|\ge1$ and $\sin\phi\simeq
\phi$.
\end{proof}

\subsection{Fourier transforms and spherical functions}\label{subsec:phi0}
For $\xi\in\Sph$ and $x\in\Bn$ let
\begin{equation}\label{eq:poisson}
   \langle x,\xi\rangle:=\log\frac{1-\|x\|^2}{\|x-\xi\|^2},\qquad
   P(x,\xi):=e^{2\rho\langle x,\xi\rangle}=\Bigl(\frac{1-\|x\|^2}
   {\|x-\xi\|^2}\Bigr)^{n-1}.
\end{equation}
Here $\langle x,\xi\rangle$ is the horocyclic bracket, that is, the
signed distance from $o$ to the horosphere through $x$ tangent to
$\Sph$ at $\xi$, and $P$ is the Poisson kernel of $\Bn$
\cite[Introduction, \S4]{Hel84}, \cite[Ch.~5]{Sto16}.  The functions
$x\mapsto e^{\lambda\langle x,\xi\rangle}$ are eigenfunctions of the
Laplace--Beltrami operator $\Delta$ \cite[Ch.~3]{Sto16}:
\begin{equation}\label{eq:horoeigen}
   \Delta e^{\lambda\langle\cdot,\xi\rangle}=\lambda(\lambda-2\rho)\,
   e^{\lambda\langle\cdot,\xi\rangle},\qquad\lambda\in\mathbb C .
\end{equation}
The Helgason--Fourier transform of a function $f$ on $\Bn$ is
\begin{equation}\label{eq_helgaosn_F}
   \widetilde f(\lambda,\xi):=\int_{\Bn}f(x)\,e^{(-i\lambda+\rho)\langle
   x,\xi\rangle}\,d\mu(x),\qquad\lambda\in\mathbb C,\ \xi\in\Sph,
\end{equation}
provided the integral exists \cite[Ch.~III]{Hel84}.  For a radial
function $f$ it does not depend on $\xi$, and it reduces to the
\emph{spherical transform}
\begin{equation}\label{eq:sphtransform}
   \widehat f(\lambda):=\int_{\Bn}f(x)\,\varphi_\lambda(x)\,d\mu(x),
   \qquad
   \varphi_\lambda(x):=\int_{\Sph}e^{(i\lambda+\rho)\langle x,\xi\rangle}
   \,d\nu(\xi),
\end{equation}
where $\varphi_\lambda$ is the spherical function, the unique radial
eigenfunction of $\Delta$ with eigenvalue $-(\lambda^2+\rho^2)$ and
$\varphi_\lambda(o)=1$ \cite[Ch.~IV]{Hel84}.  It is even in $\lambda$,
$|\varphi_\lambda|\le\varphi_0$ for $\lambda\in\R$, its radial profile
is the Jacobi function $\phi^{(\frac{n-2}2,-\frac12)}_\lambda$, and
$\varphi_{i\eta}>0$ for $\eta\in\R$, with \cite{Koo84,ADY96,Hel84}
\begin{equation}\label{eq:phieta}
   \varphi_{i\eta}(x)\simeq_\eta e^{(|\eta|-\rho)|x|}\quad(\eta\ne0),
   \qquad
   \varphi_0(x)\simeq(1+|x|)\,e^{-\rho|x|}.
\end{equation}
In particular $\varphi_0(y)\simeq\varphi_0(y')$ whenever $d(y,y')\le R$,
with constants depending on $R$.  Finally, since $\Hn$ is a symmetric
(in particular, harmonic) space, the mean of an eigenfunction of
$\Delta$ over a geodesic sphere of radius $t$ about $x$ equals its value
at $x$ times the spherical function with the same eigenvalue,
evaluated at $t$ \cite[Ch.~IV]{Hel84}, \cite{Wil50}; equivalently,
$\varphi*\psi=\widehat\psi(\lambda)\,\varphi$ for radial $\psi$, see
\cite{Pus13} for Damek--Ricci spaces.  With \eqref{eq:polar} this gives
the \emph{mean value identity}: if $\Delta\varphi=-(\lambda^2+\rho^2)
\varphi$, then
\begin{equation}\label{eq:meanvalue}
   \int_{\{r_0<d(x,y)\le r_1\}}\varphi(y)\,d\mu(y)
   =\varphi(x)\int_{r_0}^{r_1}\varphi_\lambda(t)\,\Jd(t)\,dt ,
   \qquad x\in\Bn,\ 0\le r_0<r_1 .
\end{equation}

\subsection{Averages over balls and annuli}\label{subsec_Aver_max_oper}
Recall the averages $A_r$, the operators $M$, $\Mz$, $\Minf$ and the
annular averages $\ann_\ell$ defined in \eqref{eq:annular}.  Clearly
$\ann_\ell f\le A_\ell f\le Mf$.  Conversely, for $N\ge1$ we split $B(x,N)$
into the annuli $\{\ell-1<d(x,\cdot)\le\ell\}$, $1\le\ell\le\ip N+1$.  Since
\[
   \frac1{|B(x,N)|}\int_{\{\ell-1<d(x,\cdot)\le\ell\}}|f|\,d\mu=
   \frac{|B(x,\ell)|}{|B(x,N)|}\,\ann_\ell f(x)
   \quad\text{and}\quad\frac{|B(x,\ell)|}{|B(x,N)|}\le Ce^{2\rho(\ell-N)}
\]
for $1\le\ell\le\ip N+1$ by \eqref{eq:ballvol}, and since
$\sum_{\ell\le\ip N+1}e^{2\rho(\ell-N)}\le e^{2\rho}(1-e^{-2\rho})^{-1}$, we
obtain
\begin{equation}\label{eq:ballfromannuli}
   A_Nf(x)\le C\sum_{\ell=1}^{\ip N+1}e^{2\rho(\ell-N)}\,\ann_\ell f(x),
   \quad\text{hence}\quad
   \Minf f(x)\le C\sup_{\ell\ge1}\ann_\ell f(x)\le C\sum_{\ell\ge1}\ann_\ell f(x).
\end{equation}
\subsection{Local Muckenhoupt weights}\label{subsec:local}
For $1<p<\infty$ and $L>0$ let
\[
   [w]_{A_{p,\mathrm{loc}},L}:=\sup_{x\in\Bn,\ 0<r\le L}
   \avg w{B(x,r)}\,\avg\sigma{B(x,r)}^{p-1}.
\]
Thus $w\in\Aloc$ if and only if $[w]_{A_{p,\mathrm{loc}},1}<\infty$, and
then $[w]_{A_{p,\mathrm{loc}},L}<\infty$ for every $L>0$
\cite[(1.3)]{GRS25}.  Since every hyperbolic ball is a
Euclidean ball and the balls of radius at most $L$ are, uniformly,
bi-Lipschitz equivalent to Euclidean balls (\cref{subsec:ball}), they
are uniformly spaces of homogeneous type, and the following facts are
standard \cite{Muc72,ST89,Ryc01,GRS25}:
\begin{enumerate}[label=\textup{(L\arabic*)}, leftmargin=2.6em]
\item\label{item_L1} $\Mz$ is bounded on $L^p(w)$ ($1<p<\infty$) iff $w\in\Aloc$;
      $\Mz:L^p(w)\to L^{p,\infty}(w)$ ($1\le p<\infty$) iff
      $w\in A_{p,\mathrm{loc}}(\Bn)$.
\item\label{item_L21} If $w\in\Aloc$, then $w$ and $\sigma$ are doubling at scales
      $\le L$ (and so is $w$ if $w\in A_{1,\mathrm{loc}}(\Bn)$, since
      $A_{1,\mathrm{loc}}\subset A_{2,\mathrm{loc}}$), and for every
      measurable $E\subset B(x,R)$ with $R\le L$ and $|E|\ge c|B(x,R)|$,
      $\avg wE\avg\sigma E^{p-1}\le C(c,L)[w]_{A_{p,\mathrm{loc}},L}$.
\item\label{item_L31} \textup{(Local openness and reverse H\"older.)}  If $w\in\Aloc$,
      then $w\in A_{q,\mathrm{loc}}(\Bn)$ for some $q<p$, and there is
      $\eta>0$ with $\avg{\sigma^{1+\eta}}B^{1/(1+\eta)}\le
      C\avg\sigma B$ for all balls $B$ of radius $\le L$.
\end{enumerate}

\begin{proposition}[Necessity of the local condition]
\label{prop:locnec}
Let $1<p<\infty$.  If $\|M(\sigma\Chi_B)\|^p_{L^p(w)}\le C_0\,
\sigma(B)$ for every ball $B$ of radius $\le1$ (in particular, if $M$
is bounded on $L^p(w)$), then $w\in\Aloc$ with
$[w]_{A_{p,\mathrm{loc}},1}\lesssim C_0$.  If $1\le p<\infty$ and
$M:L^p(w)\to L^{p,\infty}(w)$ is bounded, then
$w\in A_{p,\mathrm{loc}}(\Bn)$.
\end{proposition}

\begin{proof}
This is Muckenhoupt's argument \cite{Muc72}.  Let $B$ be a ball of
radius $r\le1$.  For $y\in B$ the ball of
radius $2r$ about $y$ contains $B$ and has comparable volume by
\eqref{eq:ballvol}, so that $M(f\Chi_B)\ge c\avg fB$ on $B$ for every
$f\ge0$.  With $f=\sigma_N:=\min\{\sigma,N\}$ and
$\int_B\sigma_N^pw\,d\mu\le\sigma_N(B)$, the hypothesis gives
$\avg wB\avg{\sigma_N}B^{p-1}\le c^{-p}C_0$ for all $N$, and we let
$N\to\infty$.  For the weak type we take $f=\Chi_E$ with $E\subset B$
and obtain $w(B)\le C(|B|/|E|)^pw(E)$; for $p=1$ this is the
$A_{1,\mathrm{loc}}$ condition, as one sees by letting $E$ run through
the sets $\{y\in B:\ w(y)\le\essinf_Bw+\delta\}$, and for $p>1$ the
choice $f=\sigma_N$ at the level $\frac c2\avg{\sigma_N}B$ gives the
$A_{p,\mathrm{loc}}$ condition as above.
\end{proof}

\section{Geometry of shadow slices}\label{sec:geometry}

Two geodesic rays from $o$ which make an angle $\phi$ stay within
bounded distance of each other up to time $\log(1/\phi)$ and then
separate at unit speed (\cref{lem:angle}\ref{lem:angle:b}).  Hence the points at
distance about $t$ from $o$ whose directions lie in a cap of aperture
$e^{-s}$, $s\le t$, form a \emph{polar box} of volume $\simeq
e^{2\rho(t-s)}$, which stays near a single geodesic up to time $s$;
this is the elementary form, for $\Hn$, of Sullivan's shadow lemma
\cite{Sul79}.  The shadow slices are polar boxes of this kind.  We show
that they have volume $\simeq e^{2\rho m}$ and bounded overlap
(\cref{lem:shadow}), that their width can be changed at the cost of a
constant provided that all thin slices are attached to the same point
(\cref{lem:slack}), and that every sphere is covered by slices attached
to the points of a single segment $[0,x]$ (\cref{lem:decomp}).

\subsection{Polar boxes and cones}
For an interval $I\subset\R_+$ and a measurable set $\Theta\subset\Sph$
of directions we put
\begin{equation}\label{eq:polarbox}
   E(I,\Theta):=\{y\in\Bn\setminus\{o\}:\ |y|\in I,\ \theta_y\in
   \Theta\}.
\end{equation}
By \eqref{eq:polar}, for every radial $F=F_0(|\cdot|)\ge0$,
\begin{equation}\label{eq:polarboxint}
   \int_{E(I,\Theta)}F\,d\mu=\nu(\Theta)\int_IF_0\,dV,\qquad\text{in
   particular}\qquad |E(I,\Theta)|=\nu(\Theta)\,V(I). 
\end{equation}
The polar boxes adapted to a point $z\in\Bn$ are those in which $\Theta$ is a
cap about $\theta_z$ of aperture proportional to $e^{-|z|}$: for
$\kappa>0$ we define
\begin{equation}\label{eq:polarboxz}
    \mathcal E_z(I;\kappa):=\begin{cases}
        E\bigl(I,\SCap(\theta_z,\kappa e^{-|z|})\bigr)
        =\bigl\{y\in\Bn\setminus\{o\}:\ |y|\in I,\
   \ang yz<\kappa e^{-|z|}\bigr\}, &  z\ne o,\\[3pt]
   \{y\in\Bn:\ |y|\in I\}, & z= o.
    \end{cases}
\end{equation}
In the ball model, $\mathcal E_z(I;\kappa)$ is the part of the
Euclidean annulus $\{|y|\in I\}$ lying in the Euclidean cone with
vertex $0$, axis through $z$ and aperture $\kappa e^{-|z|}\simeq
\kappa(1-\|z\|)$ (see \eqref{eq:radius}).  By \eqref{eq:polarboxint},
\eqref{eq:caps} and \eqref{eq:density},
\begin{equation}\label{eq:polarboxsize}
   |\mathcal E_z(I;\kappa)|=\nu\bigl(\SCap(\theta_z,\min\{\kappa
   e^{-|z|},\pi\})\bigr)\,V(I)\ \simeq_\kappa\ e^{-2\rho|z|}\,V(I),
\end{equation}
and, more generally, $\int_{\mathcal E_z(I;\kappa)}F\,d\mu\simeq_\kappa
e^{-2\rho|z|}\int_IF_0\,dV$ for radial $F=F_0(|\cdot|)\ge0$.

The next lemma is \cref{lem:shadowball} restricted to one unit shell:
the polar box over $[t,t+1)$ in the direction of $x$ has the size of a
ball of radius $t-|x|$, and it is contained in a ball of radius
$t-|x|+O(1)$ about $x$.

\begin{lemma}[Cones]\label{lem:cone}
Let $\kappa\ge0$, $x\in\Bn$ and $t\ge|x|$.  Then
\[
   \bigl|\mathcal E_x([t,t+1);e^\kappa)\bigr|\simeq_\kappa
   e^{2\rho(t-|x|)}\qquad\text{and}\qquad
   \mathcal E_x([t,t+1);e^\kappa)\subset B\bigl(x,\,t-|x|+1+2\kappa+
   2\log2\bigr).
\]
\end{lemma}

\begin{proof}
The first assertion follows from \eqref{eq:polarboxsize}, since
$V([t,t+1))\simeq e^{2\rho t}$ by \eqref{eq:density}.  For the inclusion
let $y\in\mathcal E_x([t,t+1);e^\kappa)$.  If $|x|\le\kappa$, then
$d(x,y)\le|x|+|y|<(t-|x|+1)+2\kappa$.  If $|x|>\kappa$, we put
$s:=|x|-\kappa$.  Since $|x|\le t\le|y|$, we have $0<s\le\min\{|x|,
|y|\}$, and $\ang yx<e^{\kappa-|x|}=e^{-s}$.  Hence $\gp xy\ge
|x|-\kappa-\log2$ by \cref{lem:angle}\ref{lem:angle:c}, and
\[
   d(x,y)=|x|+|y|-2\gp xy\le|y|-|x|+2\kappa+2\log2< t+1-|x|+2\kappa+
   2\log2 .\qedhere
\]
\end{proof}

\subsection{Shadow slices and geodesic ancestors}
For $x\in\Bn$ and $t\ge0$ let
\begin{equation}\label{eq:ancestor}
   a_t(x):=\gamma_x\bigl((|x|-t)_+\bigr)
   =\tanh\Bigl(\frac{(|x|-t)_+}{2}\Bigr)\,\theta_x ,
   \qquad a_t(o):=o,
\end{equation}
be the point of the segment $[0,x]$ at distance $\min\{t,|x|\}$ from
$x$.  Thus $d(x,a_t(x))=\min\{t,|x|\}$ and $|a_t(x)|=(|x|-t)_+$.
The ancestors of two points of a unit ball are only close to each
other, not equal; this is the source of the changes of width in
\cref{lem:slack} below.  Besides the slices
\eqref{eq:shadow} we need slices of larger width: for $s\ge s_0$,
$z\in\Bn$ and $m\in\Zp$ let
\begin{equation}\label{eq:slack}
   \Omega^{(s)}_z(m):=\bigl\{y\in\Bn:\ m-s\le|y|-|z|\le m+s,\ \
   d(z,y)\le m+s\bigr\},
\end{equation}
and let $I^{(s)}_{z,m}:=[|z|+m-s,\,|z|+m+s]\cap\R_+$ be the
corresponding interval of radii, with $I_{z,m}:=I^{(1)}_{z,m}$.  Clearly
$\Omega_z(m)\subset\Omega^{(s_0)}_z(m)\subset\Omega^{(s)}_z(m)
\subset\Omega^{(s')}_z(m)$ for $s_0\le s\le s'$.  We refer to $s$ as
the \emph{width} of the slice.

\begin{lemma}\label{lem:shadow}
Let $s\ge s_0$, $z\in\Bn$ and $m\in\Zp$.
\begin{enumerate}[label=\textup{(\alph*)}, leftmargin=2.3em]
\item\label{lem:shadow:a} If $y\in\Omega^{(s)}_z(m)$, then $\gp yz\ge|z|-s$.
\item\label{lem:shadow:b} \textup{(Polar boxes.)}
      $\mathcal E_z(I_{z,m};e)\subset\Omega_z(m)$ and
      $\Omega^{(s)}_z(m)\subset\mathcal E_z(I^{(s)}_{z,m};9e^{s})\cup
      \{o\}$.
\item\label{lem:shadow:c} $|\Omega_z(m)|\simeq|\Omega^{(s)}_z(m)|\simeq e^{2\rho m}$, with
      constants depending on $s$.
\item\label{lem:shadow:d} \textup{(Bounded overlap.)}  For every $y\in\Bn$, the set of
      centres $\{x\in\Bn:\ y\in\Omega^{(s)}_x(m)\}$ is contained in the
      ball $B(a_{m+2s}(y),3s+2)$.  Consequently it has measure at most
      $C_s$, and every $\frac14$-separated set $X\subset\Bn$ contains at
      most $C_s$ points $x$ with $y\in\Omega^{(s)}_x(m)$.
\item\label{lem:shadow:e} If $d(z,z')\le\delta$, then $\Omega^{(s)}_z(m)\subset
      \Omega^{(s+\delta)}_{z'}(m)$.
\item\label{lem:shadow:f} If $x\in\Bn$ and $j\in\Zp$, $j\le|x|$, then
      $x\in\Omega_{a_j(x)}(j)$.
\end{enumerate}
\end{lemma}

\begin{proof}
(a) For $y\in\Omega^{(s)}_z(m)$ we have $|y|\ge|z|+m-s$ and
$d(y,z)\le m+s$, so that
$\gp yz=\frac12(|y|+|z|-d(y,z))\ge\frac12(|z|+m-s+|z|-m-s)=|z|-s$.

(b) Let $y\in\mathcal E_z(I_{z,m};e)$, so that $|y|-|z|\in[m-1,m+1]$;
we have to show that $d(z,y)\le m+s_0$.  If $|z|\le1$ (in particular if
$z=o$), then $d(z,y)\le|z|+|y|\le2|z|+m+1\le m+3$.  If $|z|>1$, we put
$t:=|z|-1$.  Then $t\le\min\{|y|,|z|\}$ and $\ang yz<e^{1-|z|}=e^{-t}$,
so that $\gp yz\ge|z|-1-\log2$ by \cref{lem:angle}\ref{lem:angle:c}, and
$d(z,y)=|y|+|z|-2\gp yz\le|y|-|z|+2+2\log2\le m+3+2\log2<m+s_0$.
Conversely, let $y\in\Omega^{(s)}_z(m)\setminus\{o\}$ and $z\ne o$.
Then $|y|\in I^{(s)}_{z,m}$ by definition, and by (a) and
\cref{lem:angle}\ref{lem:angle:c}, $\ang yz\le9 e^{s}e^{-|z|}$ (if $|z|<s$
this is trivial, since then $9e^{s-|z|}>\pi$).

(c) Since $\Omega^{(s)}_z(m)\subset B(z,m+s)$, \eqref{eq:ballvol} gives
$|\Omega^{(s)}_z(m)|\lesssim_s e^{2\rho m}$.  Conversely, by (b) and
\eqref{eq:polarboxsize}, $|\Omega_z(m)|\gtrsim e^{-2\rho|z|}
V(I_{z,m})$, and $V(I_{z,m})\gtrsim e^{2\rho(|z|+m)}$, because
$I_{z,m}$ contains an interval of length $1$ with left endpoint
$(|z|+m-1)_+$.

(d) Let $y\in\Omega^{(s)}_x(m)$.  We first locate $x$ with respect to the
geodesic from $o$ to $y$.  By (a), $\gp xy\ge|x|-s$, so that
$t:=(|x|-s)_+\le\gp xy$, and \eqref{eq:fellow} gives
$d(\gamma_x(t),\gamma_y(t))\le2$ (if $x=o$ or $y=o$, then $t=0$ and
this is trivial).  Since $d(x,\gamma_x(t))=|x|-t\le s$, we get
$d(x,\gamma_y(t))\le s+2$.  Next, $|x|\in[|y|-m-s,|y|-m+s]$, so that
$t$ differs from $t_0:=(|y|-m-2s)_+$ by at most $2s$, and
$\gamma_y(t_0)=a_{m+2s}(y)$.  Hence
\[
   d\bigl(x,a_{m+2s}(y)\bigr)\le d\bigl(x,\gamma_y(t)\bigr)+|t-t_0|
   \le3s+2 .
\]
Thus every point $x$ whose slice $\Omega^{(s)}_x(m)$ contains $y$ lies
in the ball $B(a_{m+2s}(y),3s+2)$, whose measure depends only on $s$ by
\eqref{eq:ballvol}.  If $X$ is $\frac14$-separated, the balls of radius
$\frac18$ about its points are pairwise disjoint, and those about the
points $x$ in question lie in $B(a_{m+2s}(y),3s+3)$; comparing volumes
gives the last assertion.

(e) If $y\in\Omega^{(s)}_z(m)$, then $|y|-|z'|$ differs from $|y|-|z|$ by
at most $\delta$, and $d(z',y)\le d(z,y)+\delta\le m+s+\delta$, which implies $y\in\Omega^{(s+\delta)}_{z'}(m)$.

(f) We have $|a_j(x)|=|x|-j$ and $d(a_j(x),x)=j$, hence $x\in\Omega_{a_j(x)}(j)$.
\end{proof}

In particular, by (b) and \eqref{eq:polarboxsize}, the slices are
comparable to polar boxes, in the sense that for every radial
$F=F_0(|\cdot|)\ge0$,
\begin{equation}\label{eq:radialslice}
   e^{-2\rho|z|}\int_{I_{z,m}}F_0\,dV\ \lesssim\ \int_{\Omega_z(m)}F\,d\mu
   \le\int_{\Omega^{(s)}_z(m)}F\,d\mu\ \lesssim_s\ e^{-2\rho|z|}
   \int_{I^{(s)}_{z,m}}F_0\,dV .
\end{equation}

The next lemma shows that the width of the slices can be changed at the
expense of a constant.  The point is that all the thin slices are
attached to the \emph{same} anchor, namely a geodesic ancestor of $z$;
this is what allows us to compare the two layers of \cref{def:class}
simultaneously.

\begin{lemma}[Changing the width]\label{lem:slack}
Let $s\ge s_0$ and $h_s:=\ip s+2$.  For
$z\in\Bn$ let $z':=a_{h_s}(z)$.  Then $d(z,z')\le h_s$, and for every
$m\in\Zp$,
\[
   \Omega^{(s)}_z(m)\subset\bigcup_{\substack{i\in\Z,\ |i|\le2\ip s+4\\
   m+h_s+i\ge0}}\Omega_{z'}(m+h_s+i).
\]
\end{lemma}

\begin{proof}
Let $y\in\Omega^{(s)}_z(m)$, and let $J:=\ip{|y|-|z'|}$ if $|y|\ge|z'|$.
We show that $y\in\Omega_{z'}(J)$ and $|J-(m+h_s)|\le2\ip s+4$.

Suppose first that $|z|\ge h_s$, so that $|z'|=|z|-h_s\le|z|-s-1$.  By
\cref{lem:shadow}\ref{lem:shadow:a}, $|z'|<|z|-s\le\gp yz$; in particular $|y|>|z'|$.
By \eqref{eq:fellow} applied at time $|z'|$, and since
$\gamma_z(|z'|)=z'$, we have $d(\gamma_y(|z'|),z')\le2$, and
therefore $d(z',y)\le2+(|y|-|z'|)$.  Now $|y|-|z'|\in[J,J+1)\subset
[J-1,J+1]$ and $d(z',y)<J+3\le J+s_0$, that is, $y\in\Omega_{z'}(J)$.
Moreover $|y|-|z'|=(|y|-|z|)+h_s\in[m+h_s-s,m+h_s+s]$, so that
$J\in(m+h_s-s-1,m+h_s+s]$, and since $J$ is an integer,
$|J-(m+h_s)|\le\ip s+1$.

Suppose now that $|z|<h_s$, so that $z'=o$ and
$\Omega_o(J)=\{y:\ J-1\le|y|\le J+1\}$ with $J=\ip{|y|}$.  Then
$y\in\Omega_o(J)$, and since $|y|\in[|z|+m-s,|z|+m+s]\subset
[m-s,m+h_s+s]$, we get $J\in(m-s-1,m+h_s+s]$, that is,
$|J-(m+h_s)|\le h_s+\ip s+1<2\ip s+4$.
\end{proof}

For $p>1$ it is convenient to have a name for the two-layer quantities
of \cref{def:class}: for $j,m\in\Z$ we put
\begin{equation}\label{def_Lambda_m,j}
   \Lambda_{j,m}[w]:=\sup_{z\in\Bn}\ e^{-2\rho p(j+m)}\,
   w\bigl(\Omega_z(j)\bigr)\,\sigma\bigl(\Omega_z(m)\bigr)^{p-1}
   \qquad(\Lambda_{j,m}[w]:=0\ \text{ if }\min\{j,m\}<0),
\end{equation}
so that $w\in\Ab_{p,\eps}$ if and only if $\Lambda_{j,m}[w]\le
[w]_{\Ab_{p,\eps}}e^{-2\rho\eps(j+m)}$ for all $j,m\in\Zp$.

\begin{corollary}\label{cor:slack}
Let $1<p<\infty$, $s\ge s_0$, and let $h_s=\ip s+2$ be as in
\cref{lem:slack}.  Then for all $z\in\Bn$ and $j,m\in\Zp$,
\begin{equation}\label{eq:slackLambda}
   e^{-2\rho p(j+m)}\,w\bigl(\Omega^{(s)}_z(j)\bigr)\,
   \sigma\bigl(\Omega^{(s)}_z(m)\bigr)^{p-1}\le C_s\sum_{|i|,|i'|\le2\ip s+4}
   \Lambda_{j+h_s+i,\,m+h_s+i'}[w].
\end{equation}
In particular, if $\eps>0$ and
\[
   [w]^{(s)}_{\Ab_{p,\eps}}:=\sup_{z\in\Bn}\ \sup_{j,m\in\Zp}\
   \frac{w\bigl(\Omega^{(s)}_z(j)\bigr)\,\sigma\bigl(\Omega^{(s)}_z(m)
   \bigr)^{p-1}}{e^{2\rho(p-\eps)(j+m)}},
\]
then $[w]_{\Ab_{p,\eps}}\le[w]^{(s)}_{\Ab_{p,\eps}}\le C_s[w]_{\Ab_{p,\eps}}$,
with $C_s$ depending only on $n$, $s$ and $p$.
\end{corollary}

\begin{proof}
We use the following elementary inequality: if $E\subset\bigcup_{i=1}^K
E_i$ and $F\subset\bigcup_{i'=1}^KF_{i'}$, then
\begin{equation}\label{eq:unionsplit}
   w(E)\,\sigma(F)^{p-1}\le\Bigl(\sum_iw(E_i)\Bigr)
   \Bigl(K\max_{i'}\sigma(F_{i'})\Bigr)^{p-1}
   \le K^{p-1}\sum_{i,i'}w(E_i)\,\sigma(F_{i'})^{p-1}.
\end{equation}
Let $z'=a_{h_s}(z)$ be as in \cref{lem:slack}.  By \cref{lem:slack}, the
slices $\Omega^{(s)}_z(j)$ and $\Omega^{(s)}_z(m)$ are covered by the
slices $E_i:=\Omega_{z'}(j+h_s+i)$ and $F_{i'}:=\Omega_{z'}(m+h_s+i')$,
$|i|,|i'|\le2\ip s+4$, respectively, which are all attached to the
\emph{same} point $z'$.  By \eqref{eq:unionsplit} with $K=4\ip s+9$ and by
\eqref{def_Lambda_m,j} at the point $z'$,
\[
   w\bigl(\Omega^{(s)}_z(j)\bigr)\,\sigma\bigl(\Omega^{(s)}_z(m)\bigr)^{p-1}
   \le(4\ip s+9)^{p-1}\sum_{|i|,|i'|\le2\ip s+4}e^{2\rho p(j+m+2h_s+i+i')}
   \Lambda_{j+h_s+i,\,m+h_s+i'}[w].
\]
Since $2h_s+i+i'\le6\ip s+12$, this gives \eqref{eq:slackLambda} with
$C_s:=(4\ip s+9)^{p-1}e^{12\rho p(\ip s+2)}$.

The first inequality of the last assertion is obvious.  For the second
we may assume $\eps\le p$, since by H\"older's inequality and
\cref{lem:shadow}\ref{lem:shadow:c}, $w(\Omega_z(i))\sigma(\Omega_z(i))^{p-1}\ge
ce^{2\rho pi}$, so that $[w]_{\Ab_{p,\eps}}=\infty$ if $\eps>p$.  Since
$(j+h_s+i)+(m+h_s+i')\ge j+m-2h_s$ in \eqref{eq:slackLambda}, each term there
is at most $e^{4\rho ph_s}[w]_{\Ab_{p,\eps}}e^{-2\rho\eps(j+m)}$, and the
claim follows, with a constant independent of $\eps\in(0,p]$.
\end{proof}

\subsection{Decomposition of spheres}
The next lemma is the geometric input of the sufficiency part of
\cref{thmA}: every annulus $\{y:\ \ell-1<d(x,y)\le
\ell\}$ is covered by shadow slices $\Omega_{a_j(x)}(m)$ attached to the
points of the segment $[0,x]$, with $j+m=\ell$ up to a bounded error.  The
role of the last common ancestor is played by the point of $[0,x]$ at
time $\gp xy$, at which, by \eqref{eq:fellow}, the geodesics from $o$ to
$x$ and to $y$ separate (\cref{fig:decomp}).  The error in $j+m$ is
harmless, and the change of width needed later is absorbed by
\cref{cor:slack}.

\begin{figure}[t]
\centering
\begin{minipage}[t]{0.48\textwidth}
\centering
\begin{tikzpicture}[scale=2.1,>=Stealth]
 \fill[teal!22] (18:0.87) arc[start angle=18,end angle=106,radius=0.87]
   -- (106:0.981) arc[start angle=106,end angle=18,radius=0.981] -- cycle;
 \draw[thick] (0,0) circle (1);
 \draw[gray!50,densely dashed] (0,0) -- (18:1);
 \draw[gray!50,densely dashed] (0,0) -- (106:1);
 \draw[gray!70!black,densely dashed,thick] ([shift={(163.32:0.37008)}]0.79111,0.71491)
   arc[start angle=163.32,end angle=284.03,radius=0.37008];
 \draw[thick,blue!70!black] (0,0) -- (62:0.93);
 \draw[thick,red!70!black] (0,0) -- (22:0.95);
 \fill (0,0) circle (0.5pt) node[left] {\small $o$};
 \fill[blue!70!black] (62:0.93) circle (0.6pt) node[above=0pt] {\small $x$};
 \fill[red!70!black] (22:0.95) circle (0.6pt) node[right] {\small $y$};
 \fill[blue!70!black] (62:0.4875) circle (0.7pt) node[left=1pt] {\small $a_j(x)$};
 \draw[decorate,decoration={brace,amplitude=3pt}]
   ($(62:0.50)+(152:0.04)$) -- node[above left=-1pt] {\scriptsize $j$} ($(62:0.91)+(152:0.04)$);
 \node[teal!40!black] at (122:0.80) {\small $\Omega_{a_j(x)}(m)$};
 \node[gray!40!black] at ([shift={(222.1:0.25)}]0.79111,0.71491) {\scriptsize $[x,y]$};
\end{tikzpicture}
\caption{Decomposition of spheres: $y\in\Omega_z(m)$, $z=a_j(x)$,
$j+m\approx d(x,y)$.}
\label{fig:decomp}
\end{minipage}\hfill
\begin{minipage}[t]{0.48\textwidth}
\centering
\begin{tikzpicture}[scale=2.1,>=Stealth]
 \fill[teal!22] (-28:0.90) arc[start angle=-28,end angle=62,radius=0.90]
   -- (62:0.97) arc[start angle=62,end angle=-28,radius=0.97] -- cycle;
 \fill[orange!30] (-28:0.56) arc[start angle=-28,end angle=62,radius=0.56]
   -- (62:0.67) arc[start angle=62,end angle=-28,radius=0.67] -- cycle;
 \draw[thick] (0,0) circle (1);
 \draw[gray!50,densely dashed] (0,0) -- (-28:1);
 \draw[gray!50,densely dashed] (0,0) -- (62:1);
 \draw[blue!60!black,densely dashed] (0,0) -- (13:0.97);
 \draw[blue!60!black,densely dashed] (0,0) -- (37:0.97);
 \draw[thick,red!70!black] (0,0) -- (25:0.935);
 \fill (0,0) circle (0.5pt) node[left] {\small $o$};
 \fill (17:0.30) circle (0.6pt) node[below=1pt] {\small $x$};
 \fill[red!70!black] (25:0.935) circle (0.6pt) node[above right=-2pt] {\small $x'$};
 \draw[blue!70!black,fill=blue!10] (25:0.615) circle (0.045);
 \fill[blue!70!black] (25:0.615) circle (0.4pt);
 \node[blue!60!black] at (68:0.42) {\scriptsize $B(y',\frac12)$};
 \draw[->,blue!60!black] (60:0.45) -- (29:0.60);
 \node[orange!40!black] at (-40:0.50) {\scriptsize $\Omega^{(s')}_x(k)$};
 \node[teal!40!black] at (-43:0.86) {\scriptsize $\Omega^{(s)}_x(k+m)$};
\end{tikzpicture}
\caption{The semigroup inequality: $B(y',\frac12)\subset
\Omega^{(s')}_x(k)$ and $x'\in\Omega^{(s)}_y(m)$ for $y\in B(y',\frac12)$.}
\label{fig:semigroup}
\end{minipage}
\end{figure}

\begin{lemma}[Decomposition of spheres]\label{lem:decomp}
Let $x,y\in\Bn$ and $\ell\in\N$ with $\ell-1<d(x,y)\le\ell$.  Then there are
integers $0\le j\le|x|$ and $m\ge0$ with
\[
   \ell-3\le j+m\le\ell\qquad\text{and}\qquad y\in\Omega_{a_j(x)}(m).
\]
Consequently, if for $s\ge s_0$ and $j,m\in\Zp$ we define
\begin{equation}\label{eq:Sjm}
   \Sop^{(s)}_{j,m}f(x):=e^{-2\rho(j+m)}\int_{\Omega^{(s)}_{a_j(x)}(m)}
   |f|\,d\mu\quad(j\le|x|),\qquad \Sop^{(s)}_{j,m}f(x):=0\quad(j>|x|),
\end{equation}
and $\Gamma_\ell:=\{(j,m)\in\Zp^2:\ \ell-3\le j+m\le\ell\}$, then
$\#\Gamma_\ell\le4(\ell+1)$ and
\begin{equation}\label{eq:decomp}
   \ann_\ell f(x)\le C\sum_{(j,m)\in\Gamma_\ell}\Sop^{(s_0)}_{j,m}f(x),
   \qquad x\in\Bn,\ \ell\in\N .
\end{equation}
\end{lemma}

\begin{proof}
Let $t:=\gp xy\in[0,\min\{|x|,|y|\}]$ be the time at which the
geodesics from $o$ to $x$ and to $y$ separate, let $j:=\ip{|x|-t}$ and
$z:=a_j(x)$.  Then $0\le j\le|x|$ and $|z|=|x|-j\in[t,t+1)$; in
particular $|y|-|z|>|y|-t-1\ge-1$.  By \eqref{eq:fellow} applied at
time $t$,
\begin{equation}\label{eq_d(z,y)_leq_|y|-|z|}
   d(z,y)\le d\bigl(z,\gamma_x(t)\bigr)+d\bigl(\gamma_x(t),\gamma_y(t)
   \bigr)+d\bigl(\gamma_y(t),y\bigr)<1+2+(|y|-t)
   \le|y|-|z|+4 .
\end{equation}
We now choose the depth $m$.  If $|y|-|z|\ge-\frac12$, let
$m:=\ip{|y|-|z|+\frac12}$; then $|y|-|z|\in[m-\frac12,m+\frac12)$, and
\eqref{eq_d(z,y)_leq_|y|-|z|} gives $d(z,y)\le m+\frac92\le m+s_0$.  If
$-1<|y|-|z|<-\frac12$, let $m:=0$; then $d(z,y)<\frac72\le s_0$ by
\eqref{eq_d(z,y)_leq_|y|-|z|}.  In both cases $y\in\Omega_z(m)$.
Moreover, in both cases $m\in(|y|-t-\frac32,|y|-t+\frac12]$, while
$j\in(|x|-t-1,|x|-t]$; since $d(x,y)=|x|+|y|-2t$, we obtain
$j+m\in(d(x,y)-\frac52,d(x,y)+\frac12]$.  As $\ell-1<d(x,y)\le\ell$ and
$j+m$ is an integer, $\ell-3\le j+m\le\ell$.

For \eqref{eq:decomp}, the first part shows that the annulus
$\{y:\ \ell-1<d(x,y)\le\ell\}$ is covered by the slices
$\Omega_{a_j(x)}(m)\subset\Omega^{(s_0)}_{a_j(x)}(m)$ with
$(j,m)\in\Gamma_\ell$ and $j\le|x|$.  Since $|B(x,\ell)|\ge ce^{2\rho \ell}$ and
$e^{-2\rho \ell}\le e^{-2\rho(j+m)}$ on $\Gamma_\ell$, the claim follows.
Finally, for each of the at most four values of $j+m$ there are at most
$\ell+1$ pairs $(j,m)$, so that $\#\Gamma_\ell\le4(\ell+1)$.
\end{proof}

\section{Directional operators}\label{sec:directional}

In this section we introduce the directional operators, namely the
descendant averages $P^{(s)}_m$, which look $m$ units outwards, and the
ancestor averages $Q_j$, which look $j$ units back towards $o$.  Their
main properties are the one-sided semigroup inequality
(\cref{lem:semigroup}) and the norm bounds of \cref{lem:normbounds},
which replace the exact semigroup law and the norm identity of
\cite{YYZ26}.  We then record some basic properties of the two-layer
class.

We recall that $\avg fE=|E|^{-1}\int_E|f|\,d\mu$.  For $s\ge s_0$,
$m,j\in\Zp$, measurable $f$ and $x\in\Bn$ we put
\begin{equation}\label{eq_def_P_j_Q_m}
   P^{(s)}_mf(x):=\avg f{\Omega^{(s)}_x(m)},\qquad
   Q_jf(x):=\begin{cases}e^{-2\rho j}\,\avg f{B(a_j(x),\frac12)},&
   j\le|x|,\\ 0,& j>|x|,\end{cases}
\end{equation}
and $P_mf(x):=\avg f{\Omega_x(m)}$ is the average over the slices
$\Omega_x(m)$, and we define the \emph{ancestor maximal
operator}
\begin{equation}\label{def_anc_maximal}
   \Hst f:=\sup_{j\in\Zp}Q_jf.
\end{equation}
The operators $\Sop^{(s)}_{j,m}$ of \eqref{eq:Sjm} are dominated by
compositions of these operators:

\begin{lemma}\label{lem:factor}
Let $s\ge s_0$.
\begin{enumerate}[label=\textup{(\alph*)}, leftmargin=2.3em]
\item\label{lem:factor:a} For all $j,m\in\Zp$, $\Sop^{(s)}_{j,m}f\le C_s\,Q_j\bigl(
      P^{(s+1)}_mf\bigr)$ pointwise.
\item\label{lem:factor:b} For every $f$, $\Minf f\le C\sum_{m\in\Zp}\Hst\bigl(
      P^{(s_0+1)}_mf\bigr)$ pointwise.
\end{enumerate}
\end{lemma}

\begin{proof}
(a) Let $j\le|x|$ and $z=a_j(x)$.  For $z'\in B(z,\frac12)$,
\cref{lem:shadow}\ref{lem:shadow:e} gives $\Omega^{(s)}_z(m)\subset\Omega^{(s+1)}_{z'}
(m)$, hence, by \cref{lem:shadow}\ref{lem:shadow:c},
\[
   \int_{\Omega^{(s)}_z(m)}|f|\,d\mu\le|\Omega^{(s+1)}_{z'}(m)|\,
   P^{(s+1)}_mf(z')\le C_se^{2\rho m}P^{(s+1)}_mf(z').
\]
Averaging over $z'\in B(z,\frac12)$ and recalling \eqref{eq:Sjm} and
\eqref{eq_def_P_j_Q_m}, we obtain (a).

(b) Fix $N>2$.  By \eqref{eq:ballfromannuli}, \eqref{eq:decomp} and (a),
\[
   A_Nf(x)\le C\sum_{\ell=1}^{\ip N+1}e^{2\rho(\ell-N)}
   \sum_{(j,m)\in\Gamma_\ell}Q_j\bigl(P^{(s_0+1)}_mf\bigr)(x).
\]
For $(j,m)\in\Gamma_\ell$ we have $\ell\le j+m+3$, so that
$e^{2\rho(\ell-N)}\le e^{6\rho}e^{2\rho(j+m-N)}$, and $\ell\le\ip N+1$
forces $j+m<N+1$; moreover each pair $(j,m)$ belongs to at most four
sets $\Gamma_\ell$.  Hence
\[
   A_Nf(x)\le C\sum_{m\in\Zp}\Hst\bigl(P^{(s_0+1)}_mf\bigr)(x)
   \sum_{j<N+1-m}e^{2\rho(j+m-N)}\le C'\sum_{m\in\Zp}
   \Hst\bigl(P^{(s_0+1)}_mf\bigr)(x),
\]
with $C'$ independent of $N$.
\end{proof}

\subsection{A one-sided semigroup inequality}
The next lemma replaces the identity $P_kP_m=P_{k+m}$ of the tree.  The
enlargement of the slices falls only on the outer operator, and the
inequality holds at every point of $\Bn$, including the points near
the origin.

\begin{lemma}\label{lem:semigroup}
Let $s\ge s_0$ and $s':=3s+3$.  Then for every measurable $f$, every
$x\in\Bn$, every $m\in\Zp$ and every integer $k\ge s$,
\[
   P^{(s)}_{k+m}f(x)\le C_s\,P^{(s')}_k\bigl(P^{(s)}_mf\bigr)(x).
\]
Consequently, for $1\le p<\infty$ and every weight $w$, the operator
norms on $L^p(w)$ satisfy
$\|P^{(s)}_{k+m}\|\le C_s\|P^{(s')}_k\|\,\|P^{(s)}_m\|$.
\end{lemma}

\begin{proof}
By the definition of $P^{(s)}_m$ and \cref{lem:shadow}\ref{lem:shadow:c},
\begin{align*}
   P^{(s')}_k\bigl(P^{(s)}_mf\bigr)(x)&=\frac1{|\Omega^{(s')}_x(k)|}
   \int_{\Omega^{(s')}_x(k)}\frac1{|\Omega^{(s)}_y(m)|}
   \int_{\Omega^{(s)}_y(m)}|f(x')|\,d\mu(x')\,d\mu(y)\\
   &\ge c_se^{-2\rho(k+m)}\int_{\Omega^{(s')}_x(k)}\int_{\Bn}|f(x')|\,
   \Chi_{\Omega^{(s)}_y(m)}(x')\,d\mu(x')\,d\mu(y).
\end{align*}
Interchanging the order of integration (Fubini's theorem), we get
\begin{equation}\label{eq_P_lP_m>}
   P^{(s')}_k\bigl(P^{(s)}_mf\bigr)(x)\ge c_s\,e^{-2\rho(k+m)}
   \int_{\Bn}|f(x')|\,\bigl|\{y\in\Omega^{(s')}_x(k):\ x'\in
   \Omega^{(s)}_y(m)\}\bigr|\,d\mu(x').
\end{equation}
On the other hand, again by \cref{lem:shadow}\ref{lem:shadow:c},
$P^{(s)}_{k+m}f(x)\le C_se^{-2\rho(k+m)}\int_{\Omega^{(s)}_x(k+m)}|f|\,
d\mu$.  Therefore, it suffices to show that, for every
$x'\in\Omega^{(s)}_x(k+m)$, the set in \eqref{eq_P_lP_m>} contains a ball
of radius $\frac12$.

Let $x'\in\Omega^{(s)}_x(k+m)$.  Then $|x'|-m\ge|x|+k-s\ge0$, and we let
$y':=a_m(x')=\gamma_{x'}(|x'|-m)$ be the point of $[0,x']$ at distance $m$
from $x'$ (\cref{fig:semigroup}).  We claim that every
$y\in B(y',\frac12)$ belongs to the set in \eqref{eq_P_lP_m>}.

First, $|x'|-|y|\in[m-\frac12,m+\frac12]$ and $d(y,x')\le d(y',x')+\frac12
=m+\frac12$, so that $x'\in\Omega^{(s)}_y(m)$.  Next, $|y|-|x|$ lies in
$[|x'|-m-|x|-\frac12,|x'|-m-|x|+\frac12]\subset[k-s-\frac12,k+s+\frac12]$.
Finally, by \cref{lem:shadow}\ref{lem:shadow:a}, $\gp x{x'}\ge|x|-s$, so that the point
$a_s(x)=\gamma_x((|x|-s)_+)$ satisfies $d(a_s(x),\gamma_{x'}((|x|-s)_+))\le
2$ by \eqref{eq:fellow}.  Since $(|x|-s)_+\le|x|\le|y'|$, going from $x$
to $a_s(x)$, then to $\gamma_{x'}((|x|-s)_+)$, and then along $[0,x']$ to
$y'$, we get
\[
   d(x,y')\le s+2+\bigl(|y'|-(|x|-s)_+\bigr)\le s+2+(|x'|-m-|x|+s)
   \le k+3s+2 ,
\]
and so $d(x,y)\le k+3s+\frac52\le k+s'$.  Hence $y\in
\Omega^{(s')}_x(k)$, which proves the claim.  By \eqref{eq_P_lP_m>},
\[
   P^{(s')}_k\bigl(P^{(s)}_mf\bigr)(x)\ge c_s\,|B(o,\tfrac12)|\,
   e^{-2\rho(k+m)}\int_{\Omega^{(s)}_x(k+m)}|f|\,d\mu\ge c'_s\,
   P^{(s)}_{k+m}f(x),
\]
whence the norm inequality follows.
\end{proof}

\begin{remark}\label{rem:exactness}
On the tree \cite[Lemma~3.3]{YYZ26} iterates $P_kP_m=P_{k+m}$ and
$Q_jQ_i=Q_{j+i}$ with the \emph{same} operator in the inner position.
On $\Bn$ a covering argument gives only $P_{k+m}\le CP^\sharp_k
P^\sharp_m$ with enlarged slices in both factors, and
\cref{lem:semigroup} removes the enlargement from the inner factor.  For the ancestor averages, \cref{lem:angle}\ref{lem:angle:b} gives
$Q_{j+i}\le CQ'_jQ''_i$ only with a strictly larger ball in the inner
operator $Q''_i$, which cannot be iterated; this is why in
\cref{subsec:Qtail} we work directly with the sums
$\sum_je^{-2\rho pj}w(\Omega_u(j))$.
\end{remark}

\subsection{Norm bounds}
For $s\ge s_0$, $z\in\Bn$ and $j,m\in\Zp$ we put, for $1<p<\infty$,
\begin{equation}\label{eq:alpha}
\begin{aligned}
   \PP^{(s)}_m[w](z)&:=e^{-2\rho pm}\,w(B_{z})\,
   \sigma\bigl(\Omega^{(s)}_z(m)\bigr)^{p-1},\\
   \QQ^{(s)}_j[w](z)&:=e^{-2\rho pj}\,w\bigl(\Omega^{(s)}_z(j)\bigr)\,
   \sigma(B_{z})^{p-1},
\end{aligned}
\end{equation}
and for $p=1$ we replace $\sigma(\Omega^{(s)}_z(m))^{p-1}$ by
$\esssup_{\Omega^{(s)}_z(m)}w^{-1}$ in the definition of $\PP$.  We
write $\PP_m[w]$ and $\QQ_j[w]$ for the same quantities defined with the
slices $\Omega_z(m)$, $\Omega_z(j)$ of \eqref{eq:shadow}, so that
$\PP_m[w]$ is the quantity \eqref{eq:alpha-intro}.  Thus $\PP$ compares
$w$ near $z$ with $\sigma$ on a descendant slice of $z$, and $\QQ$
compares $w$ on a descendant slice with $\sigma$ near $z$.  By
H\"older's inequality $|B_z|^p\le w(B_z)\sigma(B_z)^{p-1}$, and
so
\begin{equation}\label{eq:PQlink}
   e^{-2\rho p(j+m)}\,w\bigl(\Omega_z(j)\bigr)\,\sigma\bigl(\Omega_z(m)
   \bigr)^{p-1}\le|B_o|^{-p}\,\QQ_j[w](z)\,\PP_m[w](z).
\end{equation}
This H\"older inequality on a single unit ball is the only link between
the two layers of \cref{def:class} in our proof of necessity.

\begin{lemma}[Norm bounds]\label{lem:normbounds}
Let $1\le p<\infty$, $s\ge s_0$ and let $w$ be a weight.
\begin{enumerate}[label=\textup{(\alph*)}, leftmargin=2.3em]
\item\label{lem:normbounds:a} $\|P^{(s)}_m\|^p\le C_s\sup_{z\in\Bn}\PP^{(s+1)}_m[w](z)$ for all
      $m\in\Zp$.
\item\label{lem:normbounds:b} $\|P^{(s+1)}_m\|^p\ge c_s\,\PP^{(s)}_m[w](z)$ for all $z\in\Bn$
      and $m\in\Zp$.
\item\label{lem:normbounds:c} Let $p>1$, and let $h_{s+1}=\ip s+3$ be
      the number of \cref{lem:slack} for the width $s+1$.  Then for all $j,m\in\Zp$,
      \[
         \begin{aligned}
         \|\Sop^{(s)}_{j,m}\|^p&\le C_s\,e^{-2\rho p(j+m)}\sup_{z\in\Bn}
         w\bigl(\Omega^{(s+1)}_z(j)\bigr)\,\sigma\bigl(\Omega^{(s+1)}_z(m)
         \bigr)^{p-1}\\
         &\le C_s\sum_{|i|,|i'|\le2\ip s+6}
         \Lambda_{j+h_{s+1}+i,\,m+h_{s+1}+i'}[w].
         \end{aligned}
      \]
      In particular, if $w\in\Ab_{p,\eps}$, then
      \begin{equation}\label{eq:Sjmnorm}
         \|\Sop^{(s)}_{j,m}\|\le C_s\bigl([w]_{\Ab_{p,\eps}}
         e^{-2\rho\eps(j+m)}\bigr)^{1/p}
         =C_s\Bigl(\frac{[w]_{\Ab_{p,\eps}}e^{2\rho(p-\eps)(j+m)}}
         {e^{2\rho p(j+m)}}\Bigr)^{1/p}.
      \end{equation}
\end{enumerate}
\end{lemma}

\begin{proof}
We give the proofs for $p>1$; for $p=1$ one replaces H\"older's
inequality below by $\int_\Omega|f|\,d\mu\le(\esssup_\Omega w^{-1})
\int_\Omega|f|w\,d\mu$, and the test function in (b) by $\Chi_E$, where
$E\subset\Omega^{(s)}_z(m)$, $|E|>0$, and $w^{-1}\ge(1-\eta)
\esssup_{\Omega^{(s)}_z(m)}w^{-1}$ on $E$ (if $\operatorname{ess\,sup}_{\Omega}w^{-1}=\infty$, take $E\subset\{w^{-1}\ge N\}$ and send $N\to\infty$).

(a) Let $X$ be a maximal $\frac12$-separated subset of $\Bn$, so that
the balls $B(z,\frac12)$, $z\in X$, cover $\Bn$.  If $x\in B(z,\frac12)$,
then $\Omega^{(s)}_x(m)\subset\Omega^{(s+1)}_z(m)$ by
\cref{lem:shadow}\ref{lem:shadow:e}, and by \cref{lem:shadow}\ref{lem:shadow:c} and H\"older's
inequality,
\[
   \bigl(P^{(s)}_mf(x)\bigr)^p\le C_se^{-2\rho pm}\Bigl(
   \int_{\Omega^{(s+1)}_z(m)}|f|\,d\mu\Bigr)^p\le C_se^{-2\rho pm}\,
   \sigma\bigl(\Omega^{(s+1)}_z(m)\bigr)^{p-1}\int_{\Omega^{(s+1)}_z(m)}
   |f|^pw\,d\mu .
\]
Integrating over $B(z,\frac12)$ against $w\,d\mu$ and summing over
$z\in X$,
\[
   \|P^{(s)}_mf\|^p_{L^p(w)}\le C_s\sup_{z}\PP^{(s+1)}_m[w](z)
   \int_{\Bn}|f|^pw\,\sum_{z\in X}\Chi_{\Omega^{(s+1)}_z(m)}\,d\mu
   \le C'_s\sup_z\PP^{(s+1)}_m[w](z)\,\|f\|^p_{L^p(w)},
\]
since $\sum_{z\in X}\Chi_{\Omega^{(s+1)}_z(m)}\le C_s$ by
\cref{lem:shadow}\ref{lem:shadow:d}.

(b) Let $\Omega:=\Omega^{(s)}_z(m)$, $\sigma_N:=\min\{\sigma,N\}$ and
$f:=\sigma_N\Chi_\Omega$.  Then $\|f\|^p_{L^p(w)}\le\int_\Omega\sigma_N
\sigma^{p-1}w\,d\mu=\sigma_N(\Omega)$.  For $x\in B_z$ we have
$\Omega\subset\Omega^{(s+1)}_x(m)$, so that $P^{(s+1)}_mf(x)\ge
c_se^{-2\rho m}\sigma_N(\Omega)$.  Hence $\|P^{(s+1)}_m\|^p\ge
c_s^pe^{-2\rho pm}w(B_z)\sigma_N(\Omega)^{p-1}$, and we let
$N\to\infty$.

(c) Let $j\le|x|$ and $z=a_j(x)$, and put $F(z'):=\int_{\Omega^{(s+1)}_{z'}
(m)}|f|\,d\mu$.  For every $z'\in B(z,\frac12)$ we have
$\Omega^{(s)}_z(m)\subset\Omega^{(s+1)}_{z'}(m)$ by
\cref{lem:shadow}\ref{lem:shadow:e}, so that $e^{2\rho(j+m)}\Sop^{(s)}_{j,m}f(x)\le
F(z')$.  Averaging the $p$-th power over $z'\in B(z,\frac12)$,
\[
   \bigl(\Sop^{(s)}_{j,m}f(x)\bigr)^p\le\frac{e^{-2\rho p(j+m)}}
   {|B(o,\frac12)|}\int_{\Bn}\Chi_{\{d(a_j(x),z')<\frac12\}}\,F(z')^p\,
   d\mu(z').
\]
We multiply by $w(x)$, integrate over $\{x:\ |x|\ge j\}$ and interchange
the order of integration:
\[
   \|\Sop^{(s)}_{j,m}f\|^p_{L^p(w)}\le\frac{e^{-2\rho p(j+m)}}
   {|B(o,\frac12)|}\int_{\Bn}w\bigl(\{x:\ j\le|x|,\ d(a_j(x),z')<
   \tfrac12\}\bigr)\,F(z')^p\,d\mu(z').
\]
If $j\le|x|$ and $d(a_j(x),z')<\frac12$, then $x\in\Omega_{a_j(x)}(j)
\subset\Omega^{(s+1)}_{z'}(j)$ by \cref{lem:shadow}\ref{lem:shadow:f},\ref{lem:shadow:e}; hence the
$w$-measure in the last integral is at most $w(\Omega^{(s+1)}_{z'}(j))$.
By H\"older's inequality, $F(z')^p\le\sigma(\Omega^{(s+1)}_{z'}(m))^{p-1}
\int_{\Omega^{(s+1)}_{z'}(m)}|f|^pw\,d\mu$, and therefore
\begin{align*}
   &\|\Sop^{(s)}_{j,m}f\|^p_{L^p(w)}
   \le Ce^{-2\rho p(j+m)}\sup_{z'}\Bigl[w\bigl(\Omega^{(s+1)}_{z'}(j)\bigr)
   \sigma\bigl(\Omega^{(s+1)}_{z'}(m)\bigr)^{p-1}\Bigr]
   \int_{\Bn}\int_{\Omega^{(s+1)}_{z'}(m)}|f|^pw\,d\mu\,d\mu(z')\\
   &=Ce^{-2\rho p(j+m)}\sup_{z'}\Bigl[w\bigl(\Omega^{(s+1)}_{z'}(j)\bigr)
   \sigma\bigl(\Omega^{(s+1)}_{z'}(m)\bigr)^{p-1}\Bigr]
   \int_{\Bn}|f(y)|^pw(y)\,\bigl|\{z':\ y\in
   \Omega^{(s+1)}_{z'}(m)\}\bigr|\,d\mu(y).
\end{align*}
The last measure is at most $C_s$ by \cref{lem:shadow}\ref{lem:shadow:d}, which gives
the first inequality in (c).  The second is \eqref{eq:slackLambda} with
$s+1$ in place of $s$, and \eqref{eq:Sjmnorm} follows from the first
inequality and \cref{cor:slack}.
\end{proof}

\subsection{Basic properties of the class}
We recall the weights $\ww_{\alpha,\gamma}$ from  \eqref{eq:examples}.

\begin{proposition}\label{prop:basic}
Let $1<p<\infty$.
\begin{enumerate}[label=\textup{(\alph*)}, leftmargin=2.3em]
\item\label{prop:basic:a} Let $\alpha,\gamma\in\R$ and $\eps>0$.  Then
      $\ww_{\alpha,\gamma}\in\Ab_{p,\eps}$ if and only if
      $2\rho(\eps-1)\le\alpha\le2\rho(p-1-\eps)$, where $\gamma\ge0$ is required if $\alpha=2\rho(\eps-1)$, and
      $\gamma\le0$ if $\alpha=2\rho(p-1-\eps)$.  In particular,
      $\ww_{\alpha,\gamma}\in\Ab_p$ if and only if
      $\alpha\in(-2\rho,2\rho(p-1))$.
\item\label{prop:basic:b} $\Ab_{p,\eps}\ne\emptyset$ if and only if $0<\eps\le p/2$.
\item\label{prop:basic:c} If $1<p<q$, then $[w]_{\Ab_{q,\eps}}\lesssim[w]_{\Ab_{p,\eps}}$;
      hence $\Ab_p\subsetneq\Ab_q$.
\item\label{prop:basic:d} \textup{(Openness.)}  If $w\in\Aloc\cap\Ab_p$, then $w\in
      A_{r,\mathrm{loc}}(\Bn)\cap\Ab_r$ for some $r\in(1,p)$.
\end{enumerate}
\end{proposition}

\begin{proof}
(a) Let $c,\delta\in\R$.  On $\Omega_z(i)$ we have
$|y|\in[|z|+i-1,|z|+i+1]$, so that $\ww_{c,\delta}\simeq
e^{c(|z|+i)}(1+|z|+i)^\delta$ there, and \cref{lem:shadow}\ref{lem:shadow:c} gives
$\ww_{c,\delta}(\Omega_z(i))\simeq e^{c(|z|+i)}(1+|z|+i)^\delta
e^{2\rho i}$, with constants depending on $c$ and $\delta$.  The dual
weight of $\ww_{\alpha,\gamma}$ is $\ww_{-\alpha/(p-1),-\gamma/(p-1)}$,
and therefore
\[
   \frac{\ww_{\alpha,\gamma}(\Omega_z(j))\,\sigma(\Omega_z(m))^{p-1}}
   {e^{2\rho(p-\eps)(j+m)}}\simeq e^{-aj-bm}\,
   \Bigl(\frac{1+|z|+j}{1+|z|+m}\Bigr)^{\gamma},
\]
where $a:=2\rho(p-1-\eps)-\alpha$ and $b:=\alpha-2\rho(\eps-1)$.  Since
$(1+m)^{-1}\le\frac{1+|z|+j}{1+|z|+m}\le1+j$, the right-hand side is at
most $e^{-aj-bm}(1+j)^{\gamma}$ if $\gamma\ge0$, and at most
$e^{-aj-bm}(1+m)^{|\gamma|}$ if $\gamma\le0$.  Let $a,b\ge0$.  If
$\gamma\ge0$, the quotient is bounded when $a>0$, and also when $a=0$
provided that $\gamma=0$; if $\gamma\le0$, it is bounded when $b>0$,
and also when $b=0$ provided that $\gamma=0$.  This proves the
sufficiency.  Conversely, take $z=o$.  If $a<0$, or $a=0$ and
$\gamma>0$, we let $m=0$ and $j\to\infty$, and the quotient is
$\simeq e^{|a|j}(1+j)^\gamma\to\infty$; if $b<0$, or $b=0$ and
$\gamma<0$, we let $j=0$ and $m\to\infty$.  The last assertion follows by
taking the union over $\eps>0$.

(b) By \cref{lem:shadow}\ref{lem:shadow:c} and H\"older's inequality,
\[
   ce^{2\rho i}\le|\Omega_z(i)|=\int_{\Omega_z(i)}w^{1/p}w^{-1/p}\,
   d\mu\le w(\Omega_z(i))^{1/p}\sigma(\Omega_z(i))^{1/p'},
\]
that is, $w(\Omega_z(i))\sigma(\Omega_z(i))^{p-1}\ge ce^{2\rho pi}$.
Taking $j=m=i$ in \cref{def:class}, we get $ce^{2\rho pi}\le
[w]_{\Ab_{p,\eps}}e^{4\rho(p-\eps)i}$ for all $i$, which forces
$p\le2(p-\eps)$, i.e.\ $\eps\le p/2$.  Conversely, if $\eps\le p/2$, then
$\alpha:=\rho(p-2)$ satisfies $2\rho(\eps-1)\le\alpha\le2\rho(p-1-\eps)$,
so $\ww_{\rho(p-2),0}\in\Ab_{p,\eps}$ by (a).

(c) Let $t=(q-1)/(p-1)>1$, so that $\sigma_q=\sigma_p^{1/t}$.  H\"older's
inequality with exponents $t,t'$ gives $\sigma_q(\Omega)^{q-1}\le
|\Omega|^{q-p}\sigma_p(\Omega)^{p-1}$.  With $\Omega=\Omega_z(m)$ and
$|\Omega_z(m)|\lesssim e^{2\rho m}$, the quotient of \cref{def:class}
for $q$ is at most $C[w]_{\Ab_{p,\eps}}e^{-2\rho(q-p)j}$.  Moreover
$\ww_{2\rho(p-1),0}\in\Ab_q\setminus\Ab_p$ by (a).

(d) Let $w\in\Ab_{p,\eps}\cap\Aloc$.  By \ref{item_L31},
$w\in A_{r,\mathrm{loc}}(\Bn)$ for $r<p$ close to $p$, and $\sigma=
\sigma_p$ satisfies a local reverse H\"older inequality with exponent
$1+\eta$.  Choose $r\in(\max\{1,p-\eps\},p)$ with
$t:=(p-1)/(r-1)\le1+\eta$, and put $\eps_r:=\eps-(p-r)>0$; then
$\sigma_r=\sigma^t$.  Fix $z$ and $m$, and let $Y$ be a maximal
$\frac12$-separated subset of $\Omega_z(m)$.  The balls
$B_y$, $y\in Y$, cover $\Omega_z(m)$, they are contained in
$\Omega^{(s_0+1)}_z(m)$ by \cref{lem:shadow}\ref{lem:shadow:e}, and they have bounded
overlap, with an overlap constant depending only on $n$ (by
\eqref{eq:ballvol}, a point lies in at most $|B(o,\frac54)|/|B(o,\frac14)|$
of them).  By the reverse H\"older inequality and $|B_y|\simeq1$,
\[
   \sigma_r(\Omega_z(m))\le\sum_{y\in Y}\int_{B_y}\sigma^t\,d\mu\lesssim
   \sum_{y\in Y}\sigma(B_{y})^t\le\Bigl(\sum_{y\in Y}
   \sigma(B_{y})\Bigr)^t\lesssim\sigma\bigl(\Omega^{(s_0+1)}_z
   (m)\bigr)^t,
\]
where in the third inequality we used $\ell^1\subset\ell^t$, and in the
last one the bounded overlap of the balls $B_y$.  Hence
$\sigma_r(\Omega_z(m))^{r-1}\lesssim\sigma(\Omega^{(s_0+1)}_z(m))^{p-1}$,
and by \cref{cor:slack},
\[
   w(\Omega_z(j))\,\sigma_r(\Omega_z(m))^{r-1}\lesssim[w]_{\Ab_{p,\eps}}
   e^{2\rho(p-\eps)(j+m)}=[w]_{\Ab_{p,\eps}}\,
   e^{2\rho(r-\eps_r)(j+m)} .\qedhere
\]
\end{proof}

\begin{remark}[One slice is not enough]\label{rem:onewindow}
For every $\alpha\in\R$ and $w=\ww_{\alpha,0}$, the computation in (a)
gives $\avg w{\Omega_z(i)}\avg\sigma{\Omega_z(i)}^{p-1}\simeq1$, whereas $M$ is bounded on $L^p(\ww_{\alpha,0})$ only for
$\alpha\in(-2\rho,2\rho(p-1))$ (\cref{thmC}).
\end{remark}

\begin{remark}\label{rem:openness}
The proof of (d) is the only place where the local condition enters the
large-scale class,
and it cannot be
avoided: $\Ab_p$ alone is not open.
Let $w(x):=|x|^{n(p-1)}\log(e/|x|)^{2(p-1)}$ on $B_o$ and $w:=1$
elsewhere.  Then $\sigma=|x|^{-n}\log(e/|x|)^{-2}$ on $B_o$, and
since $t\mapsto t^{-n}\log(e/t)^{-2}$ is decreasing on $(0,1]$ and
equals $1$ at $t=1$, we have $\sigma=1+b$ with $b\ge0$ supported in
$B_o$; moreover $b$ is integrable, since by \eqref{eq:polar} and
\eqref{eq:density}, $\int_{B_o}\sigma\,d\mu\simeq\int_0^1\frac{dt}{t
\log^2(e/t)}<\infty$.  As $w\le C$, we get
$w\in\Ab_{p,\eps}$ for $\eps\le\min\{1,p-1\}$,
and $\Minf$ is bounded on $L^p(w)$ by \cref{thm:largescale}.  But for
$r<p$ the function $\sigma_r=w^{-1/(r-1)}$ is not integrable near $o$,
so that $w\notin\Ab_r$ and $\Minf$ is unbounded on $L^r(w)$
(\cref{lem:testing}).  Of course $w\notin\Aloc$, since
$\avg w{B(o,t)}\avg\sigma{B(o,t)}^{p-1}\simeq\log(e/t)^{p-1}$ for
$t\le1$. 
\end{remark}

\section{General weights: proof of Theorem~\ref{thmA}}\label{sec:thmA}

The sufficiency of the two-layer condition for $\Minf$ follows directly
from the decomposition of spheres and the norm bounds of
\cref{sec:directional} (\cref{thm:largescale}).  The main work concerns
the necessity, which we derive from a testing condition for $\Minf$ in
three steps: the ancestor tails decay (\cref{lem:Qtail}), the
descendant averages decay (\cref{prop:Pgeom}), and the two decays are
combined by the H\"older link \eqref{eq:PQlink} (\cref{prop:necessity}).
None of these steps uses the local part of $M$.

\subsection{The testing condition at large scales}
Let $1<p<\infty$.  If $\Minf$ is bounded on $L^p(w)$, then testing it on
the functions $\sigma\Chi_B$, as in Sawyer's theorem, bounds
$\|\Minf(\sigma\Chi_B)\|^p_{L^p(w)}$ by a multiple of $\sigma(B)$.  We
only use balls of radius at least $1$, and define
\begin{equation}\label{eq:testing}
   [w]_{\mathbb{T}_p}:=\sup\left\{\frac{\|\Minf(\sigma\Chi_B)\|^p_{L^p(w)}}
   {\sigma(B)}:\ B=B(z,r),\ z\in\Bn,\ r\ge1\right\}
   \quad\text{if }\sigma\in L^1_{\mathrm{loc}}(\Bn),
\end{equation}
and $[w]_{\mathbb{T}_p}:=\infty$ otherwise.  Let $\sigma_N:=\min\{\sigma,
N\}$, so that $\int_B\sigma_N^pw\,d\mu\le\sigma_N(B)$.

\begin{lemma}\label{lem:testing}
Let $1<p<\infty$ and $C_0>0$, and suppose that
$\|\Minf(\sigma_N\Chi_B)\|^p_{L^p(w)}\le C_0\,\sigma_N(B)$ for every
ball $B$ of radius $r\ge1$ and every $N\in\N$.  Then
$\sigma\in L^1_{\mathrm{loc}}(\Bn)$ and $[w]_{\mathbb{T}_p}\le C_0$.
In particular $[w]_{\mathbb{T}_p}\le\|\Minf\|^p_{L^p(w)\to L^p(w)}$, and
$[w]_{\mathbb{T}_p}\lesssim1$ under \cref{thmA}\ref{thmA:iv}.
\end{lemma}

\begin{proof}
Let $B$ have radius $r\ge1$.  For $x\in B$ we have $B\subset
B(x,\max\{2r,3\})$, so that $\Minf(\sigma_N\Chi_B)\ge c_r\sigma_N(B)/|B|$
on $B$; hence $\avg wB\avg{\sigma_N}B^{p-1}\le c_r^{-p}C_0$, so
$\sigma(B)<\infty$, and monotone convergence gives the claim.  The
hypothesis holds with $C_0=\|\Minf\|^p$, and under
\cref{thmA}\ref{thmA:iv} (understood, as usual, with $\sigma_N$ in place of
$\sigma$ if $\sigma(B)=\infty$).
\end{proof}

The large-scale operator controls the $A_p$ condition at unit scales:

\begin{lemma}\label{lem:unitAp}
Let $1<p<\infty$, $[w]_{\mathbb{T}_p}<\infty$ and $R_0\ge1$.  Then:
\begin{enumerate}[label=\textup{(\alph*)}, leftmargin=2.3em]
\item\label{lem:unitAp:a} $\avg{w}{B}\avg{\sigma}{B}^{p-1}\le C[w]_{\mathbb{T}_p}$ for every
      ball $B$ of radius $r\in[1,R_0]$;
\item\label{lem:unitAp:b} $w(B)\le C[w]_{\mathbb{T}_p}\,w(B_y)$ and $\sigma(B)^{p-1}\le
      C[w]_{\mathbb{T}_p}\,\sigma(B_y)^{p-1}$ whenever $B_y\subset
      B=B(x,R)$ and $R\le R_0$.
\end{enumerate}
Here $C$ depends only on $n$, $p$ and $R_0$.
\end{lemma}

\begin{proof}
(a) Let $B=B(z,r)$ with $1\le r\le R_0$, and let $x\in B$.  Then
$B\subset B(x,r')$ with $r':=\max\{2r,3\}>2$, and $|B(x,r')|\le C|B|$ by
\eqref{eq:ballvol}.  Hence $\Minf(\sigma\Chi_B)\ge c\,\sigma(B)/|B|$
on $B$, and the testing condition gives
$c^p\sigma(B)^pw(B)|B|^{-p}\le[w]_{\mathbb{T}_p}\sigma(B)$, that is,
$\avg{w}{B}\avg{\sigma}{B}^{p-1}\le c^{-p}[w]_{\mathbb{T}_p}$.

(b) By H\"older's inequality and (a),
\[
   |B_y|^p\le w(B_y)\,\sigma(B_y)^{p-1}\le w(B_y)\,\sigma(B)^{p-1}\le
   C[w]_{\mathbb{T}_p}\,w(B_y)\frac{|B|^p}{w(B)},
\]
which gives the first inequality since $|B|\le C|B_y|$.  Similarly,
\[
   |B_y|^p\le w(B)\,\sigma(B_y)^{p-1}\le C[w]_{\mathbb{T}_p}|B|^p\,
   \sigma(B)^{1-p}\sigma(B_y)^{p-1},
\]
which gives the second inequality.
\end{proof}

\begin{remark}\label{rem:large-scale-ap}
Thus $[w]_{\mathbb{T}_p}<\infty$ gives the $A_p$ condition, and the
doubling of $w$ and $\sigma$, at unit scales, but no condition at
scales $r<1$ (\cref{ex:loc}); this is why $w\in\Aloc$ has to be added
for $M$ in \cref{thmA}.
\end{remark}

\subsection{The ancestor tails}\label{subsec:Qtail}
For $x\in\Bn$, $s\ge s_0$ and $k\in\Zp$ we consider the tails of the
sums of the quantities $\QQ^{(s)}_j[w](x)$ of \eqref{eq:alpha},
\[
   \Phi^{(s)}_k(x):=\sum_{j\ge k}\QQ^{(s)}_j[w](x)=\sigma(B_x)^{p-1}
   \sum_{j\ge k}e^{-2\rho pj}\,w\bigl(\Omega^{(s)}_x(j)\bigr),
\]
and we write $\Phi_k(x)$ for the same quantity defined with the slices
$\Omega_x(j)$ of \eqref{eq:shadow}.

\begin{lemma}\label{lem:Qtail}
Let $1<p<\infty$ and $[w]_{\mathbb{T}_p}<\infty$.  Then there are
$C<\infty$ and $\eps_0>0$, depending only on $n$, $p$ and
$[w]_{\mathbb{T}_p}$, such that
\[
   \QQ_j[w](x)\le \Phi_j(x)\le C\,e^{-2\rho\eps_0j}\qquad\text{for all }
   x\in\Bn,\ j\in\Zp .
\]
\end{lemma}

\begin{proof}
We first show that the sums are bounded.  Let $s\ge s_0$ and
$x\in\Bn$.  If $y\in\Omega^{(s)}_x(j)$, then $B_x\subset B(y,j+s+1)$
and $j+s+1>2$, so that $\Minf(\sigma\Chi_{B_x})(y)\ge
c_se^{-2\rho j}\sigma(B_x)$.  Since $|y|-|x|-s\le j\le|y|-|x|+s$ for
such $j$, a point $y$ belongs to $\Omega^{(s)}_x(j)$ for at most
$2s+1$ values of $j$.  Therefore, by \eqref{eq:testing},
\[
   \sigma(B_x)^p\sum_{j\ge0}e^{-2\rho pj}w\bigl(\Omega^{(s)}_x(j)\bigr)\le
   C_s\|\Minf(\sigma\Chi_{B_x})\|^p_{L^p(w)}\le C_s[w]_{\mathbb{T}_p}\,
   \sigma(B_x),
\]
that is,
\begin{equation}\label{eq:Qbound}
   \sup_{x\in\Bn}\Phi^{(s)}_0(x)\le C_s[w]_{\mathbb{T}_p}<\infty .
\end{equation}
We claim the following inequality holds
\begin{equation}\label{eq:Qcover}
   \Phi_{k+1}(x)\le C_0\,\QQ^{(2s_0+4)}_k[w](x),\qquad x\in\Bn,\
   k\in\Zp,
\end{equation}
with $C_0$ depending only on $n$, $p$ and $[w]_{\mathbb{T}_p}$.  To
prove it, let $m\ge1$ and $x'\in\Omega_x(k+m)$.  Since $|x'|\ge|x|+k+
m-1\ge|x|+k$, we may consider the point $y':=\gamma_{x'}(|x|+k)$ of
$[0,x']$, at distance $|x'|-|x|-k\in[m-1,m+1]$ from $x'$.  We claim
that every $y\in B(y',\frac12)$ belongs to the set
\begin{equation}\label{eq:Qcover-set}
   \bigl\{y\in\Omega^{(2s_0+3)}_x(k):\ x'\in\Omega^{(s_0)}_y(m)\bigr\}.
\end{equation}
First, $|x'|-|y|\in[m-\frac32,m+\frac32]$ and $d(y,x')\le d(y',x')+
\frac12\le m+\frac32$, so that $x'\in\Omega^{(s_0)}_y(m)$.  Next,
$|y|-|x|\in[k-\frac12,k+\frac12]$.  Finally, by \cref{lem:shadow}\ref{lem:shadow:a},
$\gp x{x'}\ge|x|-s_0$, so that the point $a_{s_0}(x)=\gamma_x
((|x|-s_0)_+)$ satisfies $d(a_{s_0}(x),\gamma_{x'}((|x|-s_0)_+))\le2$
by \eqref{eq:fellow}.  Since $(|x|-s_0)_+\le|x|\le|y'|$, going from $x$
to $a_{s_0}(x)$, then to $\gamma_{x'}((|x|-s_0)_+)$, and then along
$[0,x']$ to $y'$, we get
\[
   d(x,y')\le s_0+2+\bigl(|y'|-(|x|-s_0)_+\bigr)\le s_0+2+(k+s_0)
   =k+2s_0+2 ,
\]
and so $d(x,y)\le k+2s_0+\frac52$.  Hence $y\in\Omega^{(2s_0+3)}_x(k)$,
which proves the claim.  Thus, for every $x'\in\Omega_x(k+m)$, the set
\eqref{eq:Qcover-set} contains a ball of radius $\frac12$, and by
Fubini's theorem,
\[
   w\bigl(\Omega_x(k+m)\bigr)\le\frac1{|B(o,\frac12)|}
   \int_{\Omega^{(2s_0+3)}_x(k)}w\bigl(\Omega^{(s_0)}_y(m)\bigr)\,d\mu(y) .
\]
Multiplying by $e^{-2\rho p(k+m)}$, summing over $m\ge1$ and using
\eqref{eq:Qbound} with $s=s_0$, we get
\[
   \begin{aligned}
   \sum_{m\ge1}e^{-2\rho p(k+m)}w\bigl(\Omega_x(k+m)\bigr)
   &\le\frac{e^{-2\rho pk}}{|B(o,\frac12)|}\int_{\Omega^{(2s_0+3)}_x(k)}
   \Phi^{(s_0)}_1(y)\,\sigma(B_y)^{1-p}\,d\mu(y)\\
   &\le C\,e^{-2\rho pk}\int_{\Omega^{(2s_0+3)}_x(k)}\sigma(B_y)^{1-p}\,
   d\mu(y).
   \end{aligned}
\]
By H\"older's inequality, $\sigma(B_y)^{1-p}\le|B_y|^{-p}w(B_y)$;
since the $1$-neighbourhood of $\Omega^{(2s_0+3)}_x(k)$ lies in
$\Omega^{(2s_0+4)}_x(k)$, Fubini's theorem bounds the integral of
$w(B_y)$ over $\Omega^{(2s_0+3)}_x(k)$ by $|B_o|\,w(\Omega^{(2s_0+4)}_x
(k))$.  Multiplying by $\sigma(B_x)^{p-1}$ we obtain \eqref{eq:Qcover}.

Let now $s:=2s_0+4$.  Since $\Phi_k(x)$ decreases in $k$, averaging
\eqref{eq:Qcover} over $k$ gives, for all $N\in\Zp$ and $J\in\N$,
\begin{equation}\label{eq:Qaverage}
   J\,\Phi_{N+J}(x)\le\sum_{k=N}^{N+J-1}\Phi_{k+1}(x)\le
   C_0\sum_{k=N}^{N+J-1}\QQ^{(s)}_k[w](x)\le C_0\,\Phi^{(s)}_N(x).
\end{equation}
The right-hand side involves wider slices, and we return to the slices
\eqref{eq:shadow} by moving the anchor: let $h_s=\ip s+2$ and $z'=a_{h_s}(x)$
be as in \cref{lem:slack}.  By that lemma,
$w(\Omega^{(s)}_x(k))\le\sum_{|i|\le2\ip s+4}w(\Omega_{z'}(k+h_s+i))$,
where terms with $k+h_s+i<0$ are omitted.  Each index $k+h_s+i$ arises from
at most $4\ip s+9$ pairs $(k,i)$, and $e^{-2\rho pk}\le
e^{2\rho p(3\ip s+6)}e^{-2\rho p(k+h_s+i)}$.  Moreover $B_x\subset
B(z',h_s+1)$, so that $\sigma(B_x)^{p-1}\le C[w]_{\mathbb{T}_p}\,
\sigma(B_{z'})^{p-1}$ by \cref{lem:unitAp}\ref{lem:unitAp:b} with $R_0=h_s+1$.  Since
$k+h_s+i\ge k-h_s$, we conclude that
\begin{equation}\label{eq:Qshift}
   \Phi^{(s)}_N(x)\le C_1\,\Phi_{\max\{0,N-h_s\}}(z'),
\end{equation}
with $C_1$ depending only on $n$, $p$ and $[w]_{\mathbb{T}_p}$.

Finally, let $g_k:=\sup_{x\in\Bn}\Phi_k(x)$, which is finite by
\eqref{eq:Qbound} and nonincreasing in $k$, and choose $J\in\N$ with
$J\ge2C_0C_1$.  By \eqref{eq:Qaverage} and \eqref{eq:Qshift},
$g_{N+J}\le\frac12g_{N-h_s}$ for $N\ge h_s$.  Taking $N=h_s+i(J+h_s)$, we get
$g_{(i+1)(J+h_s)}\le\frac12g_{i(J+h_s)}$ for all $i\in\Zp$, and hence
$g_k\le2g_0e^{-2\rho\eps_0k}$ for all $k$, where $e^{-2\rho\eps_0}:=
2^{-1/(J+h_s)}$.  Since $\QQ_j[w]\le \Phi_j$, the proof is complete.
\end{proof}

\subsection{The descendant averages}
\begin{lemma}\label{lem:Pdecay}
Let $1<p<\infty$, $[w]_{\mathbb{T}_p}<\infty$ and $s\ge s_0$.  Then
\[
   \sup_{z\in\Bn}\PP^{(s)}_m[w](z)\le C\,(m+1)^{-1/p'},\qquad m\in\Zp,
\]
where $C$ depends only on $n$, $p$, $s$ and $[w]_{\mathbb{T}_p}$.
\end{lemma}

\begin{proof}
Fix $z\in\Bn$ and $m\in\Zp$, and put $R:=\Omega^{(s)}_z(m)$, so that
$R\subset B(z,m+s)$.  For every $s'\ge s_0$ and $x\in B_z$ we have
$B(z,m+s')\subset B(x,m+s'+1)$, so that $\Minf(\sigma\Chi_{B(z,m+s')})
\ge c_{s'}e^{-2\rho m}\sigma(B(z,m+s'))$ on $B_z$, and the testing
condition \eqref{eq:testing} gives
\begin{equation}\label{eq:alpha-unif}
   w(B_z)\,\sigma\bigl(B(z,m+s')\bigr)^{p-1}\le C_{s'}[w]_{\mathbb{T}_p}
   \,e^{2\rho pm};\quad\text{in particular}\quad\PP^{(s')}_m[w](z)\le
   C_{s'}[w]_{\mathbb{T}_p} .
\end{equation}
This proves the lemma for $m\le\ip s+1$.  Let now $m\ge\ip s+2$, and
for $\ip s+2\le j\le m$ put $g_j:=P^{(s+1)}_{m-j}(\sigma\Chi_R)$.  For
$x\in B_z$ we have $R\subset\Omega^{(s+1)}_x(m)$ by
\cref{lem:shadow}\ref{lem:shadow:e}, so that, by \cref{lem:shadow}\ref{lem:shadow:c} and
\cref{lem:semigroup} (applied with $s+1$ in place of $s$, which is
allowed since $j\ge s+1$),
\[
   c_s\frac{\sigma(R)}{|R|}\le P^{(s+1)}_m(\sigma\Chi_R)(x)\le
   C_s\,P^{(3s+6)}_j(g_j)(x).
\]
Since $\Omega^{(3s+6)}_x(j)\subset\Omega^{(3s+7)}_z(j)$ by
\cref{lem:shadow}\ref{lem:shadow:e}, H\"older's inequality gives
\[
   \Bigl(\frac{\sigma(R)}{|R|}\Bigr)^p\le C_se^{-2\rho pj}\,
   \sigma\bigl(\Omega^{(3s+7)}_z(j)\bigr)^{p-1}\int_{\Omega^{(3s+7)}_z(j)}
   g_j^pw\,d\mu ,
\]
and the right-hand side does not depend on $x\in B_z$.  Multiplying by
$w(B_z)$ and using $|R|\le C_se^{2\rho m}$ and \eqref{eq:alpha-unif}
(with $3s+7$ in place of $s'$ and $j$ in place of $m$), we get
\begin{equation}\label{eq:Pstep}
   \PP^{(s)}_m[w](z)\,\sigma(R)\le C_s\,\PP^{(3s+7)}_j[w](z)
   \int_{\Omega^{(3s+7)}_z(j)}g_j^pw\,d\mu\le C_s[w]_{\mathbb{T}_p}
   \int_{\Omega^{(3s+7)}_z(j)}g_j^pw\,d\mu .
\end{equation}
For every $y\in\Bn$ we have $g_j(y)\le C_se^{-2\rho(m-j)}\sigma(R\cap
B(y,m-j+s+1))\le C_s\Minf(\sigma\Chi_{B(z,m+s)})(y)$, because
$m-j+s+1>2$.  Since a point belongs to $\Omega^{(3s+7)}_z(j)$ for at
most $6s+15$ values of $j$, summing \eqref{eq:Pstep} over $\ip s+2\le
j\le m$ gives
\[
   (m-\ip s-1)\,\PP^{(s)}_m[w](z)\,\sigma(R)\le C_s[w]_{\mathbb{T}_p}\,
   \|\Minf(\sigma\Chi_{B(z,m+s)})\|^p_{L^p(w)}\le C_s[w]_{\mathbb{T}_p}^2
   \,\sigma\bigl(B(z,m+s)\bigr).
\]
On the other hand, \eqref{eq:alpha-unif} with $s'=s$ and the definition
of $\PP^{(s)}_m[w](z)$ give $(\sigma(B(z,m+s))/\sigma(R))^{p-1}\le
C_s[w]_{\mathbb{T}_p}/\PP^{(s)}_m[w](z)$.  Combining the two estimates,
$(m-\ip s-1)\,\PP^{(s)}_m[w](z)^{p'}\le C_s[w]_{\mathbb{T}_p}^{2+1/(p-1)}$,
which proves the lemma for $m\ge\ip s+2$.
\end{proof}

\begin{proposition}\label{prop:Pgeom}
Let $1<p<\infty$, $[w]_{\mathbb{T}_p}<\infty$ and $s\ge s_0$.  Then
there are $C<\infty$ and $\eps_1>0$, depending only on $n$, $p$, $s$ and
$[w]_{\mathbb{T}_p}$, with
\[
   \|P^{(s)}_m\|_{L^p(w)\to L^p(w)}\le C\,e^{-2\rho\eps_1m}\quad\text{and}
   \quad\sup_{z\in\Bn}\PP^{(s)}_m[w](z)\le C\,e^{-2\rho p\eps_1m},\qquad
   m\in\Zp .
\]
\end{proposition}

\begin{proof}
Let $s'=3s+3$.  By \cref{lem:normbounds}\ref{lem:normbounds:a} and
\cref{lem:Pdecay} (applied with $s+1$ and $s'+1$),
$\sup_m\|P^{(s)}_m\|<\infty$ and $\|P^{(s')}_k\|\to0$ as
$k\to\infty$.  Choose an integer $k_0\ge s$ with
$C_s\|P^{(s')}_{k_0}\|\le\frac12$, where $C_s$ is the constant of
\cref{lem:semigroup}.  Then $\|P^{(s)}_{k_0+m}\|\le\frac12\|P^{(s)}_m
\|$ for all $m$, which gives the first estimate with
$e^{-2\rho\eps_1}:=2^{-1/k_0}$.  The second estimate follows from
\cref{lem:normbounds}\ref{lem:normbounds:b} and the first estimate with $s+1$ in place of
$s$ (after replacing $\eps_1$ by the smaller of the two numbers so
obtained).
\end{proof}

\subsection{Proof of Theorem~\ref{thmA}}

\begin{proposition}\label{prop:necessity}
Let $1<p<\infty$ and $[w]_{\mathbb{T}_p}<\infty$.  Then
$w\in\Ab_{p,\eps}$ for some $\eps=\eps(n,p,[w]_{\mathbb{T}_p})>0$, and
$[w]_{\Ab_{p,\eps}}\le C(n,p,[w]_{\mathbb{T}_p})$.
\end{proposition}

\begin{proof}
Let $z\in\Bn$ and $j,m\in\Zp$.  By the H\"older link
\eqref{eq:PQlink}, and since $\PP_m[w](z)\le\PP^{(s_0)}_m[w](z)$,
\[
   w\bigl(\Omega_z(j)\bigr)\,\sigma\bigl(\Omega_z(m)\bigr)^{p-1}\le
   \frac{e^{2\rho p(j+m)}}{|B_o|^p}\,\QQ_j[w](z)\,\PP^{(s_0)}_m[w](z).
\]
By \cref{lem:Qtail}, $\QQ_j[w](z)\le Ce^{-2\rho\eps_0j}$, and by
\cref{prop:Pgeom}, $\PP^{(s_0)}_m[w](z)\le Ce^{-2\rho p\eps_1m}$,
uniformly in $z$.  Hence, with $\eps:=\min\{\eps_0,p\eps_1\}>0$, the
right-hand side is at most $Ce^{2\rho(p-\eps)(j+m)}$, that is,
$[w]_{\Ab_{p,\eps}}\le C$.
\end{proof}

\begin{theorem}\label{thm:largescale}
Let $1<p<\infty$ and let $w$ be a weight on $\Bn$.  Then $\Minf$ is
bounded on $L^p(w)$ if and only if $[w]_{\mathbb{T}_p}<\infty$, if and
only if $w\in\Ab_p$, if and only if $\|\ann_\ell\|_{L^p(w)\to L^p(w)}
\lesssim e^{-\eps\ell}$, $\ell\in\N$, for some $\eps>0$.  Moreover, for
every weight $w$, every $\eps\in(0,p/2]$ and every $\ell\in\N$,
\begin{equation}\label{eq:annulardecay}
   \|\ann_\ell f\|_{L^p(w)}\le C(n,p)\,[w]_{\Ab_{p,\eps}}^{1/p}\,(\ell+1)\,
   e^{-2\rho\eps\ell/p}\,\|f\|_{L^p(w)},
\end{equation}
and $\|\Minf\|_{L^p(w)\to L^p(w)}\le C(n,p)\,\eps^{-2}\,
[w]^{1/p}_{\Ab_{p,\eps}}$.
\end{theorem}

\begin{proof}
By \eqref{eq:decomp}, Minkowski's inequality and \eqref{eq:Sjmnorm} with
$s=s_0$,
\[
   \|\ann_\ell f\|_{L^p(w)}\le C\sum_{(j,m)\in\Gamma_\ell}\|\Sop^{(s_0)}_{j,m}f\|
   _{L^p(w)}\le C\sum_{(j,m)\in\Gamma_\ell}\bigl([w]_{\Ab_{p,\eps}}
   e^{-2\rho\eps(j+m)}\bigr)^{1/p}\|f\|_{L^p(w)}.
\]
The constant in \eqref{eq:Sjmnorm} does not depend on $\eps\le p$ (see
the proof of \cref{cor:slack}).  Since $j+m\ge \ell-3$ on $\Gamma_\ell$ and
$\#\Gamma_\ell\le4(\ell+1)$ (\cref{lem:decomp}), this gives
\eqref{eq:annulardecay}.  By \eqref{eq:ballfromannuli} and Minkowski's
inequality, $\|\Minf f\|\le C\sum_{\ell\ge1}\|\ann_\ell f\|\le
C[w]^{1/p}_{\Ab_{p,\eps}}\sum_\ell(\ell+1)e^{-2\rho\eps\ell/p}\|f\|$, and
$\sum_\ell(\ell+1)e^{-2\rho\eps\ell/p}\le C(n,p)\eps^{-2}$ for $\eps\le p/2$.
Thus $\Ab_p$ implies the annular decay (absorbing $\ell+1$ into a smaller
exponent), and the annular decay implies the boundedness of $\Minf$ by
the same argument.  Finally, the boundedness of $\Minf$ gives
$[w]_{\mathbb{T}_p}<\infty$ by \cref{lem:testing}, and
$[w]_{\mathbb{T}_p}<\infty$ gives $w\in\Ab_p$ by \cref{prop:necessity}.
\end{proof}

\begin{proof}[\textbf{Proof of \cref{thmA}}]
(i)$\Rightarrow$(iv) is immediate.  (iv)$\Rightarrow$(ii): by
\cref{prop:locnec}, $w\in\Aloc$; moreover $[w]_{\mathbb{T}_p}<\infty$ by
\cref{lem:testing}, and \cref{prop:necessity} gives $w\in\Ab_p$.
(ii)$\Rightarrow$(iii) is \cref{thm:largescale}.  (iii)$\Rightarrow$(i):
$\Minf$ is bounded by \cref{thm:largescale}, and $\Mz$ is bounded by
\ref{item_L1}.



(i)$\Rightarrow$(v).  By the above, $w\in\Aloc$ and
$\|\ann_\ell\|\le Ce^{-\eps\ell}$ with, say, $\eps<2\rho$.  For measurable
$E,F$, by Fubini's theorem and H\"older's inequality,
\begin{align*}
    \int_F|\{x\in E:\ \ell-1<d(x,y)\le\ell\}|\,w(y)\,d\mu(y) &=\int_F|B(y,\ell)|\,\ann_\ell\Chi_E(y)\,w(y)\,d\mu(y)\\
   & \le C e^{2\rho\ell} \|\ann_\ell\Chi_E\|_{L^p(w)}w(F)^{1/p'} \\
&\le C
e^{(2\rho-\eps)\ell}w(E)^{1/p}w(F)^{1/p'}.
\end{align*}  Summing over
$\ell\le\ip r+1$ and dividing by $|B(o,r)|\simeq e^{2\rho r}$ gives
\eqref{eq:AOcond} with $\beta=1-\eps/(2\rho)$.

(v)$\Rightarrow$(i).  Condition \eqref{eq:AOcond} is condition (1.3) of
\cite[Theorem~1.2]{AO25} with $\alpha=1$ and $\beta<\alpha<p$; 
and the condition is required for every $r\ge1$, as here.  Together
with $w\in\Aloc$, \cite[Theorem~1.2]{AO25} (see also
\cite[Theorem~1.2(1)]{GRS25}) yields the boundedness of $M$.  

The characterisation of $\Minf$ and \eqref{eq:buckley-intro} are
\cref{thm:largescale}, the optimality of $1/p$ is \cref{prop:sharp}
below, and the openness follows from \cref{prop:basic}\ref{prop:basic:c},\ref{prop:basic:d}.
\end{proof}

\begin{remark}\label{rem:invariance}
The class $\Ab_p$ is defined through shadows seen from the base point
$o$, and nothing in its definition shows that it does not depend on
$o$.  However, isometries commute with $\Minf$, so that
$\|\Minf\|_{L^p(w\circ g)}=\|\Minf\|_{L^p(w)}$ for every M\"obius
transformation $g$ of $\Bn$, and \cref{thm:largescale} shows that
$w\in\Ab_p$ if and only if $w\circ g\in\Ab_p$.  Thus, belonging to
$\Ab_p$ is an intrinsic property of the weight, independent of the
choice of base point.
\end{remark}

\begin{proposition}[Optimality of the exponent $1/p$]\label{prop:sharp}
Let $1<p<\infty$ and fix $\eps\in(0,p/2]$.  If $\|\Minf\|_{L^p(w)\to
L^p(w)}\le C[w]^t_{\Ab_{p,\eps}}$ for all $w\in\Ab_{p,\eps}$, then
$t\ge1/p$.
\end{proposition}

\begin{proof}
Let $v:=\ww_{\rho(p-2),0}=e^{\rho(p-2)|x|}$, so that $v\in\Ab_{p,\eps}$
by \cref{prop:basic}\ref{prop:basic:a}.  For $\delta\in(0,1)$ let $v_\delta:=\delta v$
on $B_o$ and $v_\delta:=v$ elsewhere.  Then $v_\delta\le v$ and
$\sigma_{v_\delta}=  \delta^{-\frac{1}{p-1}} \sigma_v \chi_{B_o} + \sigma_v \chi_{B_o^c} \le\sigma_v+\delta^{-1/(p-1)}\sigma_v\Chi_{B_o}$.  If
$\Omega_z(m)$ meets $B_o$, then $|z|+m\le2$, and $v(\Omega_z(j))\lesssim
e^{\rho pj}\le e^{2\rho(p-\eps)(j+m)}$ by \cref{lem:shadow}\ref{lem:shadow:c} and
$\eps\le p/2$.  Hence $[v_\delta]_{\Ab_{p,\eps}}\lesssim
[v]_{\Ab_{p,\eps}}+\delta^{-1}\lesssim\delta^{-1}$.  With
$f=\Chi_{B_o}$ we have $\|f\|^p_{L^p(v_\delta)}=\delta v(B_o)$, while
for $2\le|x|\le3$ the ball $B(x,4)$ contains $B_o$, so that
$\Minf f(x)\ge   A_4(\chi_{B_o})(x) = \frac{|B_o|}{|B(x,4)|}\geq c$.
By definition of operator norm $$\|{M}^\infty\| \geq  \frac{\|M^\infty f\|_{L^p(v_{\delta})}}{\|f\|_{L^p(v_{\delta})}} \ge \frac{c}{\delta v(B_o)} \gtrsim \delta^{-1}.$$
Hence
$\delta^{-1}\lesssim\|\Minf\|^p_{L^p(v_\delta)}\lesssim\delta^{-tp}$ for
all $\delta$. Taking $\delta \to 0^+$, we get \ $t\ge1/p$.
\end{proof}

The next example shows that the local condition is independent of 
large-scale condition $\Ab_p$.
\begin{example}\label{ex:loc}
Let $1<p<\infty$ and $t_0\ge2$, choose disjoint intervals
$I_k\subset[t_0,t_0+1]$ of length $2^{-k}$, $k\in\N$, split each $I_k$
into its lower and upper halves $I_k^-$ and $I_k^+$, and let
$w=w_0(|\cdot|)$ with $w_0=2^{-k(p-1)/2}$ on $I_k^-$ and $w_0=1$
elsewhere.  Then $w\le1$, and $\sigma-1\ge0$ is supported in
$\{t_0\le|x|\le t_0+1\}$ with $\int_{\Bn}(\sigma-1)\,d\mu\simeq
\sum_k2^{-k}2^{k/2}<\infty$.  Hence $w(\Omega_z(j))\lesssim e^{2\rho j}$
and $\sigma(\Omega_z(m))\lesssim e^{2\rho m}+1$, so that
$w\in\Ab_{p,\eps}$ for $\eps\le\min\{1,p-1\}$, and $\Minf$ is bounded on
$L^p(w)$ by \cref{thm:largescale}.  On the other hand, let $x_k$ be a point
such that $|x_k|$ is the midpoint of $I_k$.  About half of the ball
$B(x_k,2^{-k-2})$ lies in $\{|y|\in I_k^-\}$, so that
$\avg{w}{B(x_k,2^{-k-2})}\avg{\sigma}{B(x_k,2^{-k-2})}^{p-1}\gtrsim
2^{k(p-1)/2}\to
\infty$.  Thus $w\notin\Aloc$, and $M$ is unbounded on $L^p(w)$ by
\cref{prop:locnec}.
\end{example}

As on the tree \cite[Theorem~4.3]{ORR22}, the global testing condition
in \cref{thmA}\ref{thmA:iv} cannot be replaced by its local form.

\begin{proposition}
\label{prop:localsawyer}
Let $1<p<\infty$ and $w=\ww_{2\rho(p-1),0}$, so that
$\sigma=\ww_{-2\rho,0}$.  Then $M\sigma\le C\sigma$, hence
$\int_BM(\sigma\Chi_B)^p\,w\,d\mu\le C\sigma(B)$ for every ball $B$, and
$M$ is of weak type $(p,p)$ with respect to $w$; but
$\|M(\sigma\Chi_{B_o})\|_{L^p(w)}=\infty$.
\end{proposition}

\begin{proof}
Let $|x|=t$ and $r>0$.  If $r<1$, then $\sigma\simeq e^{-2\rho t}$ on
$B(x,r)$.  Let $r\ge1$.  If $y\in B(x,r)\setminus\{o\}$ and $|y|=s$, then
$\gp xy\ge\frac12(t+s-r)$, so that $\ang yx\le9e^{-(t+s-r)/2}$ by
\cref{lem:angle}\ref{lem:angle:c}, and the set of such directions has $\nu$-measure
$\lesssim e^{-\rho(t+s-r)_+}$ by \eqref{eq:caps}.  By \eqref{eq:polar}
and $\Jd(s)\lesssim e^{2\rho s}$,
\[
   \int_{B(x,r)}\sigma\,d\mu\lesssim\int_{(t-r)_+}^{t+r}e^{-2\rho s}\,
   e^{2\rho s}\,e^{-\rho(t+s-r)_+}\,ds\lesssim
   \begin{cases}e^{-2\rho(t-r)},& r\le t,\\ 1+r-t,& r>t.\end{cases}
\]
Dividing by $|B(x,r)|\simeq e^{2\rho r}$ gives $A_r\sigma(x)\lesssim
e^{-2\rho t}$ in both cases, since $(1+r-t)e^{-2\rho(r-t)}\lesssim1$.
Thus $M\sigma\le C\sigma$, and
$\int_BM(\sigma\Chi_B)^pw\,d\mu\le C\int_B\sigma^pw\,d\mu=C\sigma(B)$.  The
weak type is \cref{thmC}\ref{thmC:c}.  Finally, if $|x|=t\ge2$, then
$B(x,t+1)\supset B_o$, so that $M(\sigma\Chi_{B_o})(x)\gtrsim
e^{-2\rho t}$, and $\int_{\Bn}M(\sigma\Chi_{B_o})^pw\,d\mu\gtrsim
\int_2^\infty e^{-2\rho pt}e^{2\rho(p-1)t}e^{2\rho t}\,dt=\infty$.
\end{proof}

\section{The weak type: proof of Theorem~\ref{thmB}}\label{sec:weak}

On the tree the weak type is characterised for radial weights in
\cite[Theorem~4.1]{YYZ26} by two conditions, which control the ancestor
maximal operator and the decay of the descendant averages; conditions
\WQ and \WP are their counterparts for general weights on $\Bn$.

\begin{lemma}[Ancestor maximal operator]\label{lem:Hstar}
Let $1\le p<\infty$ and $w\in A_{p,\mathrm{loc}}(\Bn)$.  If $w$
satisfies \WQ, then $\Hst:L^p(w)\to L^{p,\infty}(w)$ with
$\|\Hst\|^p\lesssim[w]_{\mathcal W_Q}$.
\end{lemma}

\begin{proof}
Let $Y$ be a maximal $\frac14$-separated subset of $\Bn$; the balls
$B(y,\frac34)$, $y\in Y$, have bounded overlap by \eqref{eq:ballvol}.
Let $f\in L^p(w)$ and $\lambda>0$.  If $Q_j f(x)>\lambda$ for some
$j\le|x|$, choose $y\in Y$
with $d(a_j(x),y)<\frac14$.  Then $B(a_j(x),\frac12)\subset
B(y,\frac34)$, so that $b_y:=\avg{f}{B(y,\frac34)}>c\lambda e^{2\rho j}$
with $c:=|B(o,\frac12)|/|B(o,\frac34)|$; moreover $x\in\Omega_y(j)$,
since $d(a_j(x),x)=j$ gives $|x|-|y|\in(j-\frac14,j+\frac14)$ and
$d(y,x)<j+\frac14$.  For each $y\in Y$ with $b_y>c\lambda$ let $N_y$ be the
largest $j$ with $b_y>c\lambda e^{2\rho j}$.  Then, by \WQ,
\begin{align*}
   w(\{\Hst f>\lambda\})&\le\sum_{y}\sum_{j\le N_y}w(\Omega_y(j))
   \le[w]_{\mathcal W_Q}\sum_ye^{2\rho p N_y}w(B_y)\\
   &\le\frac{[w]_{\mathcal W_Q}}{(c\lambda)^p}\sum_yb_y^p\,w(B_y)
   \lesssim\frac{[w]_{\mathcal W_Q}}{\lambda^p}\|f\|^p_{L^p(w)},
\end{align*}
because, by the doubling of $w$ at bounded scales \ref{item_L21} and H\"older's
inequality,
\[
   b_y^pw(B_y)\lesssim\bigl(\avg{w}{B}\avg{\sigma}{B}^{p-1}\bigr)
   \int_B|f|^pw\,d\mu\qquad\Bigl(p=1:\ b_yw(B_y)\lesssim
   \frac{\avg{w}{B}}{\essinf_Bw}\int_B|f|w\,d\mu\Bigr),
\]
with $B=B(y,\frac34)$, and the balls $B(y,\frac34)$ have bounded overlap.
\end{proof}

The next lemma shows that, under the local condition, \WP does not
depend on the width of the slices, and that it is equivalent to the
exponential decay of the norms of the descendant averages.

\begin{lemma}\label{lem:WPwidth}
Let $1\le p<\infty$, $w\in A_{p,\mathrm{loc}}(\Bn)$, and suppose that
\WP holds with the constant $\eps>0$.  Then for every $s\ge s_0$ there
is $C_s<\infty$ such that
\[
   \sup_{z\in\Bn}\PP^{(s)}_m[w](z)\le C_s\,e^{-2\rho p\eps m}\qquad
   \text{and}\qquad\|P^{(s)}_m\|_{L^p(w)\to L^p(w)}^p\le C_s\,
   e^{-2\rho p\eps m},\qquad m\in\Zp .
\]
Conversely, if $\|P^{(s)}_m\|_{L^p(w)\to L^p(w)}^p\le C\,e^{-2\rho p
\eps m}$ for some $s\ge s_0+1$ and $\eps>0$, then \WP holds.
\end{lemma}

\begin{proof}
Let $h_s=\ip s+2$ and $z'=a_{h_s}(z)$ be as in \cref{lem:slack}.  Since $B_z\subset
B(z',h_s+1)$ and $w$ is doubling at scales $\le h_s+1$ by \ref{item_L21}, we
have $w(B_z)\le Cw(B_{z'})$.  For $p>1$, \cref{lem:slack} and \eqref{eq:unionsplit} give
\[
   \begin{aligned}
   \PP^{(s)}_m[w](z)&\le C\,e^{-2\rho pm}\,w(B_{z'})\,\max_{|i|\le2\ip s+4}
   \sigma\bigl(\Omega_{z'}(m+h_s+i)\bigr)^{p-1}\\
   &\le C\max_{|i|\le2\ip s+4}e^{2\rho p(h_s+i)}\PP_{m+h_s+i}[w](z'),
   \end{aligned}
\]
where indices with $m+h_s+i<0$ are omitted, and \WP bounds the right-hand
side by $C_se^{-2\rho p\eps m}$.  For $p=1$ the same argument applies,
since the essential supremum of $w^{-1}$ over $\Omega^{(s)}_z(m)$ is at
most the maximum of those over the slices $\Omega_{z'}(m+h_s+i)$.  The second assertion follows from the
first and \cref{lem:normbounds}\ref{lem:normbounds:a}, and the converse
from \cref{lem:normbounds}\ref{lem:normbounds:b}, since
$\PP_m[w]\le\PP^{(s_0)}_m[w]$.
\end{proof}

\begin{proof}[\textbf{Proof of \cref{thmB}}]
(a) The local part $\Mz$ is of weak type $(p,p)$ by \ref{item_L1}.  By
\cref{lem:factor}\ref{lem:factor:b},
\begin{equation}\label{eq:MinfH}
   \Minf f\le C_1\sum_{m\in\Zp}\Hst\bigl(P^{(s_0+1)}_mf\bigr)
\end{equation}
pointwise, and by \cref{lem:WPwidth}, $\|P^{(s_0+1)}_mf\|^p_{L^p(w)}\le
Ce^{-2\rho p\eps m}\|f\|^p_{L^p(w)}$.  Let $\lambda>0$ and
$\lambda_m:=(1-e^{-\rho\eps})e^{-\rho\eps m}\lambda$, so that
$\sum_m\lambda_m=\lambda$.  By \eqref{eq:MinfH} and \cref{lem:Hstar},
\[
   w(\{\Minf f>C_1\lambda\})\le\sum_{m\in\Zp}
   w\bigl(\{\Hst(P^{(s_0+1)}_mf)>\lambda_m\}\bigr)
   \lesssim\sum_{m\in\Zp}\frac{e^{-2\rho p\eps m}}{\lambda_m^p}
   \|f\|^p_{L^p(w)}\lesssim\frac{\|f\|^p_{L^p(w)}}{\lambda^p},
\]
since $e^{-2\rho p\eps m}\lambda_m^{-p}\simeq e^{-\rho p\eps m}
\lambda^{-p}$.  The weak type of $M=\max\{\Mz,\Minf\}$ follows.

(b) Suppose that $M:L^p(w)\to L^{p,\infty}(w)$ is bounded.  Then
$w\in A_{p,\mathrm{loc}}(\Bn)$ by \cref{prop:locnec}; in particular
$\sigma\in L^1_{\mathrm{loc}}(\Bn)$ if $p>1$.

\emph{Condition \WQ.}  Let $y\in\Bn$, $B:=B_y$ and $N\in\Zp$.  If
$x\in\Omega_y(j)$, then $d(x,y)\le j+s_0$, so that $B\subset
B(x,j+s_0+1)$, and by \eqref{eq:ballvol}, for every nonnegative
measurable function $f$,
\begin{equation}\label{eq:MsigmaB}
   M(f\Chi_B)(x)\ge\frac{\int_Bf\,d\mu}{|B(x,j+s_0+1)|}\ge
   c\,e^{-2\rho j}\int_Bf\,d\mu .
\end{equation}
Let $p>1$ and $f=\sigma$.  Then $\bigcup_{j\le N}\Omega_y(j)\subset
\{M(\sigma\Chi_B)\ge ce^{-2\rho N}\sigma(B)\}$, and the weak type
inequality gives
\begin{equation}\label{eq_w(O)_le_}
   w\Bigl(\bigcup_{j\le N}\Omega_y(j)\Bigr)\le C\bigl(e^{-2\rho N}
   \sigma(B)\bigr)^{-p}\|\sigma\Chi_B\|^p_{L^p(w)}
   =Ce^{2\rho pN}\sigma(B)^{1-p}\le Ce^{2\rho pN}\frac{w(B)}{|B|^p},
\end{equation}
where in the last step we used H\"older's inequality $|B|^p\le
w(B)\sigma(B)^{p-1}$.  Since a point belongs to $\Omega_y(j)$ for at
most three values of $j$, we have $\sum_{j=0}^Nw(\Omega_y(j))\le
3\,w(\bigcup_{j=0}^N\Omega_y(j))$, and \eqref{eq_w(O)_le_} gives \WQ.
For $p=1$ we take $f=1$ in \eqref{eq:MsigmaB}, and obtain
$w(\bigcup_{j\le N}\Omega_y(j))\le Ce^{2\rho N}|B|^{-1}w(B)$ directly.

\emph{Uniform form of \WP.}  Let $z\in\Bn$ and $m\in\Zp$, so that
$\Omega_z(m)\subset B(z,m+s_0)$.  For $x\in B_z$ we have $B(z,m+s_0)
\subset B(x,m+s_0+1)$, so that $M(f\Chi_{B(z,m+s_0)})\ge ce^{-2\rho m}
\int_{B(z,m+s_0)}f\,d\mu$ on $B_z$ for $f\ge0$.  For $p>1$ we take
$f=\sigma$; the weak type inequality gives $w(B_z)\le Ce^{2\rho pm}
\sigma(B(z,m+s_0))^{1-p}$, that is, $e^{-2\rho pm}w(B_z)\sigma(B(z,m+
s_0))^{p-1}\le C$, and a fortiori
$\PP_m[w](z)\le C$.  For $p=1$ let $\delta>0$ and let $E\subset
\Omega_z(m)$ with $|E|>0$ and $w\le(1+\delta)\essinf_{\Omega_z(m)}w$ on
$E$.  Taking $f=\Chi_E$ we get $w(B_z)\le Ce^{2\rho m}w(E)/|E|\le
C(1+\delta)e^{2\rho m}\essinf_{\Omega_z(m)}w$, that is,
$\PP_m[w](z)\le C(1+\delta)$.

(c) If $p>1$ and $M$ is bounded on $L^p(w)$, then $[w]_{\mathbb T_p}<
\infty$ by \cref{lem:testing}, and \cref{prop:Pgeom} gives \WP with
$\eps=\eps_1$, since $\PP_m[w]\le\PP^{(s_0)}_m[w]$.  The
statement for radial weights is part of \cref{thm:radialweak} below.
\end{proof}

\begin{remark}\label{rem:weakcompare}
\Cref{thmB}\ref{thmB:a} can be compared with the case $\beta=\alpha$ of
\cite[Theorem~1.2]{AO25} and \cite[Theorem~1.2(2)]{GRS25}.  For
$w=\ww_{2\rho(p-1),0}$, for which the weak type holds but the strong type
fails \cite[Example~4.1(2)]{AO25}, conditions \WQ and \WP hold
(\cref{thm:radialweak}).
\end{remark}

\section{Radial weights: proof of Theorem~\ref{thmC}}\label{sec:radial}

Throughout this section $w=w_0(|\cdot|)$ is radial.  For $t\ge0$, let us define
\begin{align}\label{def_of_R_n}
\mathcal R_t:=\{x:\ t\le|x|<t+1\},    
\end{align}
and for $N\in\Zp$ let
$W_N:=w(\mathcal R_N)$ and $\Sigma_N:=\sigma(\mathcal R_N)$.  On the
tree a radial weight is constant on each level, so that the level
masses of $w$ determine those of $\sigma$; on $\Bn$ the required
relation is supplied by the local condition
(\cref{lem:radialshadow}\ref{lem:radialshadow:b}).

\begin{lemma}\label{lem:radialshadow}
Let $1\le p<\infty$ and let $w$ be a radial weight in
$A_{p,\mathrm{loc}}(\Bn)$.  Let $R_0\ge1$ and $s\ge s_0$.  The
constants below depend only on $n$, $p$, $R_0$, $s$ and the local
constant of $w$.
\begin{enumerate}[label=\textup{(\alph*)}, leftmargin=2.3em]
\item\label{lem:radialshadow:a} $w(\mathcal R_t)\simeq w(\mathcal R_{t'})$
      whenever $t,t'\ge0$ and $|t-t'|\le R_0$; in particular
      $W_{N+1}\simeq W_N$ and $w(\mathcal R_t)\simeq W_{\ip t}$.
\item\label{lem:radialshadow:b} If $p>1$, then $W_N\Sigma_N^{p-1}\simeq e^{2\rho pN}$.  If $p=1$,
      then $\essinf_{[N,N+1)}w_0\gtrsim e^{-2\rho N}W_N$.
\item\label{lem:radialshadow:c} Let $z\in\Bn$ and $N:=\ip{|z|}$, so that $N\le|z|<N+1$.  Then,
      for all $m\in\Zp$ and $\frac14\le r\le R_0$,
      \[
         w\bigl(\Omega_z(m)\bigr)\simeq w\bigl(\Omega^{(s)}_z(m)\bigr)\simeq
         e^{-2\rho N}\,W_{N+m}\qquad\text{and}\qquad
         w\bigl(B(z,r)\bigr)\simeq e^{-2\rho N}\,W_N .
      \]
\end{enumerate}
If $p>1$, then \textup{(a)} and \textup{(c)} also hold for $\sigma$ and
$\Sigma_N$ in place of $w$ and $W_N$.
\end{lemma}

\begin{proof}
We prove the assertions for $w$; those for $\sigma$ are proved in the
same way.  For $t\ge0$ and $\theta\in\Sph$ we use the polar boxes
\eqref{eq:polarboxz} over unit intervals,
\[
   E_t(\theta):=\mathcal E_{\gamma_\theta(t)}([t,t+1);10)
   =\bigl\{y\in\Bn\setminus\{o\}:\ t\le|y|<t+1,\
   \angle(\theta_y,\theta)<10e^{-t}\bigr\},
\]
where $\gamma_\theta$ is the geodesic ray \eqref{eq:geodesic}.  Since
$E_t(\theta)$ is the product of the cap $\SCap(\theta,\min\{10e^{-t},
\pi\})$ and the interval $[t,t+1)$ of radii, \eqref{eq:polarboxint} and
\eqref{eq:caps} give $w(E_t(\theta))\simeq e^{-2\rho t}w(\mathcal R_t)$,
and in the same way $|E_t(\theta)|\simeq1$.  Moreover,
\begin{equation}\label{eq:Etheta}
   B\bigl(\gamma_\theta(t+\tfrac12),\tfrac12\bigr)\subset E_t(\theta)
   \subset B\bigl(\gamma_\theta(t),7\bigr).
\end{equation}
Indeed, let $y\in B(\gamma_\theta(t+\frac12),\frac12)$.  Then
$t<|y|<t+1$ and $\gp{\gamma_\theta(t+\frac12)}y\ge\frac12(t+\frac12+|y|
-\frac12)>t$, so that $\angle(\theta_y,\theta)\le9e^{-t}$ by
\cref{lem:angle}\ref{lem:angle:c}, and $y\in E_t(\theta)$.  Next, let
$y\in E_t(\theta)$.  If $t=0$, then $d(\gamma_\theta(t),y)=|y|<1$.  If
$t>0$, then $\min\{|\gamma_\theta(t)|,|y|\}=t$, and
\cref{lem:angle}\ref{lem:angle:a} gives
$e^{-2\gp{\gamma_\theta(t)}y}\le2(1+100)e^{-2t}$; taking logarithms,
$2\gp{\gamma_\theta(t)}y\ge2t-\log202>2t-6$, and hence
$d(\gamma_\theta(t),y)=t+|y|-2\gp{\gamma_\theta(t)}y<|y|-t+6<7$.

(a) Let $|t-t'|\le R_0$ and $\theta\in\Sph$.  By \eqref{eq:Etheta},
both $E_t(\theta)$ and $E_{t'}(\theta)$ are contained in the ball
$B(\gamma_\theta(t),R_0+7)$ and contain a ball of radius $\frac12$.
Since $w$ is doubling at bounded scales by \ref{item_L21},
$w(E_t(\theta))\simeq w(B(\gamma_\theta(t),R_0+7))\simeq
w(E_{t'}(\theta))$, that is, $e^{-2\rho t}w(\mathcal R_t)\simeq
e^{-2\rho t'}w(\mathcal R_{t'})$, and so $w(\mathcal R_t)\simeq
w(\mathcal R_{t'})$.

(b) Let $E:=E_N(\theta)$.  By \eqref{eq:Etheta}, $E$ lies in a ball $B$
of radius $7$ with $|E|\ge c|B|$, so that \ref{item_L21} gives
$\avg{w}{E}\avg{\sigma}{E}^{p-1}\lesssim1$.  Since $w$ is radial and $E$
is the product of a cap and the interval $[N,N+1)$, the angular measure
cancels in the averages, and $\avg{w}{E}=W_N/V([N,N+1))$,
$\avg{\sigma}{E}=\Sigma_N/V([N,N+1))$.  Hence
$W_N\Sigma_N^{p-1}\lesssim V([N,N+1))^p\simeq e^{2\rho pN}$; the lower
bound is H\"older's inequality on $\mathcal R_N$.  For $p=1$,
$\avg{w}{E}\lesssim\essinf_Ew=\essinf_{[N,N+1)}w_0$.

(c) By \eqref{eq:radialslice}, $w(\Omega^{(s)}_z(m))\lesssim
e^{-2\rho|z|}\int_{I^{(s)}_{z,m}}w_0\,dV$.  The interval
$I^{(s)}_{z,m}$ has length at most $2s$ and is covered by at most
$2s+2$ unit intervals $[t,t+1)$ with $|t-(|z|+m)|\le s+1$; by (a), the
mass of $w$ over each of the corresponding rings is comparable to
$W_{N+m}$.  Thus $w(\Omega^{(s)}_z(m))\lesssim e^{-2\rho N}W_{N+m}$.
Conversely, \eqref{eq:radialslice} and (a) give $w(\Omega_z(m))\gtrsim
e^{-2\rho|z|}w(\mathcal R_{\max\{|z|+m-1,0\}})\simeq e^{-2\rho N}
W_{N+m}$, and since $\Omega_z(m)\subset\Omega^{(s)}_z(m)$, the estimate
for the slices follows.  For the balls, let $\frac14\le r\le R_0$ and
$\theta:=\theta_z$ ($\theta$ arbitrary if $z=o$).  Both $B(z,r)$ and
$E_{|z|}(\theta)$ are contained in $B(z,R_0+7)$ and contain balls of
radius $\frac14$, so that, by \ref{item_L21}, $w(B(z,r))\simeq
w(E_{|z|}(\theta))\simeq e^{-2\rho|z|}w(\mathcal R_{|z|})\simeq
e^{-2\rho N}W_N$.
\end{proof}

\begin{lemma}\label{lem:gaps}
\begin{enumerate}[label=\textup{(\alph*)}, leftmargin=2.3em]
\item\label{lem:gaps:a} \Ro holds iff there are $C,\kappa>0$ with
      $W_{N-i}\le Ce^{-\kappa i}W_N$ for $0\le i\le N$.
\item\label{lem:gaps:b} If \Rt holds, there are $C,\eta>0$ with
      $W_{N+i}\le Ce^{(2\rho p-\eta)i}W_N$ for all $N,i$.
\end{enumerate}
\end{lemma}

\begin{proof}
(a) Suppose that \Ro holds, so that the partial sums
$S_N:=\sum_{i\le N}W_i$ satisfy $S_N\le CW_N$ for some $C>1$.  Then
$S_{N-1}=S_N-W_N\le(1-C^{-1})S_N$, and hence $W_{N-i}\le S_{N-i}\le
(1-C^{-1})^iCW_N$, which is the claim with $e^{-\kappa}:=1-C^{-1}$.
The converse follows by summing a geometric series.

(b) Let $r_i(N):=e^{-2\rho pi}W_{N+i}/W_N$, so that
$r_{i+j}(N)=r_i(N)r_j(N+i)$, and by \Rt there is $C>1$ with
$\sum_ir_i(N)\le C$ for all $N$.  The tail sums $S_i(N):=\sum_{j\ge0}
r_{i+j}(N)=r_i(N)\sum_{j\ge0}r_j(N+i)$ satisfy $S_i(N)\le Cr_i(N)$, and
$S_{i+1}(N)=S_i(N)-r_i(N)\le(1-C^{-1})S_i(N)$.  Hence $r_i(N)\le
S_i(N)\le(1-C^{-1})^iS_0(N)\le C(1-C^{-1})^i$, which is the claim with
$e^{-\eta}:=1-C^{-1}$.
\end{proof}

\begin{lemma}[Test functions]\label{lem:radialtests}
Let $N\in\Zp$.  Then
\begin{enumerate}[label=\textup{(\alph*)}, leftmargin=2.3em]
\item\label{lem:radialtests:a} $\Minf\Chi_{\mathcal R_N}(x)\ge c$ for $|x|<N$;
\item\label{lem:radialtests:b} $\Minf\Chi_{\mathcal R_N}(x)\ge ce^{-2\rho(|x|-N)}$ for $|x|\ge N$.
\end{enumerate}
\end{lemma}

\begin{proof}
(a) Let $|x|<N$.  By \cref{lem:cone} with $\kappa=0$ and $t=N$, the
polar box $\mathcal E_x([N,N+1);1)\subset\mathcal R_N$ has measure
$\simeq e^{2\rho(N-|x|)}$ and lies in $B(x,r)$, $r:=N-|x|+1+2\log2>2$,
whose measure is also $\simeq e^{2\rho(N-|x|)}$.  Hence
$\Minf\Chi_{\mathcal R_N}(x)\ge A_r\Chi_{\mathcal R_N}(x)\ge
|\mathcal E_x([N,N+1);1)|/|B(x,r)|\ge c$.
(b) Let $z:=\gamma_{\theta_x}(N+\frac12)$; then $B(z,\frac12)\subset
\mathcal R_N\cap B(x,r)$ with $r:=\max\{|x|-N+1,3\}$, and
$|B(x,r)|\simeq e^{2\rho(|x|-N)}$.  Hence $\Minf\Chi_{\mathcal R_N}(x)
\ge|B(z,\frac12)|/|B(x,r)|\ge ce^{-2\rho(|x|-N)}$.
\end{proof}

The interval conditions of \cref{thmC} are handled by the polar boxes
$E(I,\Theta)$ defined in \eqref{eq:polarbox}.  By \eqref{eq:polarboxint},
the angular measure cancels in averages of radial functions: for every
radial $F=F_0(|\cdot|)\ge0$,
\begin{equation}\label{eq:polarboxavg}
   \avg{F}{E(I,\Theta)}=\frac1{V(I)}\int_IF_0\,dV .
\end{equation}
Let $B$ be a ball with $\{|y|:\ y\in B\}\subset I$, and
$\Theta:=\{\theta_y:\ y\in B\setminus\{o\}\}$.  Every $y\in E(I,\Theta)$
has the direction of some $y'\in B$, and $d(y,y')=||y|-|y'||\le|I|$;
hence
\begin{equation}\label{eq:boxball}
   B\setminus\{o\}\subset E(I,\Theta)\subset\{y:\ d(y,B)\le|I|\}.
\end{equation}
For $1\le p<\infty$ and an interval $I$ we write $[w_0]_{A_p(V,I)}$ for
the product in \eqref{eq:ApV} over the single interval $I$, with the
second factor replaced by $\|w_0^{-1}\|_{L^\infty(I)}$ if $p=1$.

\begin{lemma}\label{lem:ApV}
Let $w=w_0(|\cdot|)$ be a radial weight.
\begin{enumerate}[label=\textup{(\alph*)}, leftmargin=2.3em]
\item\label{lem:ApV:a} Let $1\le p<\infty$.  Then $w\in A_{p,\mathrm{loc}}(\Bn)$ if and
      only if $\sup\{[w_0]_{A_p(V,I)}:\ |I|\le1\}<\infty$.
\item\label{lem:ApV:b} Let $1<p<\infty$.  Then $w_0\in A_p(\R_+,V)$ if and only if
      $w\in\Aloc$ and \Ro and \Rt hold.
\end{enumerate}
\end{lemma}

\begin{proof}
(a) Suppose that $w\in A_{p,\mathrm{loc}}(\Bn)$.  Let $I=[a,b]$ with
$b-a\le1$, $y_0:=\gamma_\theta(\frac{a+b}2)$ for some direction
$\theta$, $B:=B(y_0,\frac{b-a}2)$, and $E=E(I,\Theta)$ as in
\eqref{eq:boxball}.  The radii of $B$ fill $(a,b)$, and
$B\setminus\{o\}\subset E\subset B(y_0,2(b-a))$ with
$|B(y_0,2(b-a))|\le C|B|$.  Hence, by \eqref{eq:polarboxavg},
$[w_0]_{A_p(V,I)}=\avg{w}{E}\avg{\sigma}{E}^{p-1}\le C$ by
\ref{item_L21} if $p>1$, and $V(I)^{-1}\int_Iw_0\,dV=\avg{w}{E}\le
C\avg{w}{B(y_0,2(b-a))}\le C\essinf_Bw=C\essinf_Iw_0$ if $p=1$.
Conversely, let $r\le\frac12$, $I:=[\max\{|x|-r,0\},|x|+r]$, and
$E=E(I,\Theta)$ be as in \eqref{eq:boxball} with $B(x,r)$ in place of
$B$.  Then $E\subset B(x,3r)$ and $|B(x,3r)|\le C|B(x,r)|$, so that
$\avg{w}{B(x,r)}\le C\avg{w}{E}$, and similarly for $\sigma$; moreover
$\essinf_{B(x,r)}w\ge\essinf_Iw_0$, since $|y|\in I$ for $y\in B(x,r)$.
By \eqref{eq:polarboxavg}, the $A_p$ (resp.\ $A_1$) product of $w$ on
$B(x,r)$ is at most $C[w_0]_{A_p(V,I)}$ with $|I|\le1$.  This
suffices, since the class $A_{p,\mathrm{loc}}(\Bn)$ does not depend on
the threshold $L$ (\cref{subsec:local}).

(b) The local condition follows from (a).  The
intervals $[N,N+1]$ and H\"older's inequality give $W_N\Sigma_N^{p-1}
\simeq V([N,N+1])^p\simeq e^{2\rho pN}$.  For $I=[0,N+1]$ we have
$V(I)\simeq e^{2\rho N}$ and
\[
   \Bigl(\sum_{i\le N}W_i\Bigr)\Sigma_N^{p-1}\le
   \Bigl(\sum_{i\le N}W_i\Bigr)\Bigl(\sum_{i\le N}\Sigma_i\Bigr)^{p-1}
   \le[w_0]_{A_p(V)}V(I)^p\lesssim W_N\Sigma_N^{p-1};
\]
cancelling $\Sigma_N^{p-1}$ we obtain \Ro, and in the same way
$\sum_{i\le N}\Sigma_i\lesssim\Sigma_N$.  The proof of
\cref{lem:gaps}\ref{lem:gaps:a} applies to $(\Sigma_N)$ and gives
$\kappa>0$ with $\Sigma_N\le Ce^{-\kappa i}\Sigma_{N+i}$; hence
$e^{-2\rho pi}W_{N+i}/W_N\simeq(\Sigma_N/\Sigma_{N+i})^{p-1}\lesssim
e^{-\kappa(p-1)i}$, and \Rt follows.

Next to prove sufficiency, we first observe  intervals of length $\le1$ are covered by (a).  Let
$I=[a,b]$ with $b-a>1$, and $N\in\Zp$ with $N<b\le N+1$.  Then
$N\ge1$, $I\subset[0,N+1]$ and $V(I)\ge V([b-1,b])\gtrsim e^{2\rho N}$.
By \Ro, $\int_Iw_0\,dV\le\sum_{i\le N}W_i\lesssim W_N$.  By
\cref{lem:gaps}\ref{lem:gaps:b} and \cref{lem:radialshadow}\ref{lem:radialshadow:b},
$\Sigma_{N+i}^{p-1}\simeq e^{2\rho p(N+i)}W_{N+i}^{-1}\gtrsim e^{\eta i}
\Sigma_N^{p-1}$, so that $\int_I\sigma_0\,dV\le\sum_{i\le N}\Sigma_i
\lesssim\Sigma_N$.  Hence $[w_0]_{A_p(V,I)}\lesssim e^{-2\rho pN}W_N
\Sigma_N^{p-1}\simeq1$.
\end{proof}

The next lemma shows that the local part of the interval condition is
already forced by the boundedness of $M$ on radial functions.

\begin{lemma}\label{lem:radiallocal}
Let $1<p<\infty$, and suppose that $M$ is bounded on the radial
functions of $L^p(w)$, with norm $\|M\|_{\mathrm{rad}}$.  Then
$[w_0]_{A_p(V,I)}\le C\|M\|_{\mathrm{rad}}^p$ for every interval
$I\subset\R_+$ with $|I|\le1$; in particular $w\in\Aloc$.
\end{lemma}

\begin{proof}
Let $I=[a,b]$, $r:=b-a\in(0,1]$, $N\in\N$ and $\sigma_N:=\min\{\sigma,
N\}$, with radial profile $\sigma_{N,0}$.

\emph{Case $a\le b/2$.}  Then $b\le2r\le2$.  Let $f:=\sigma_N
\Chi_{B(o,b)}$, which is radial.  For $|x|\le b$ we have
$B(o,b)\subset B(x,2b)$ and $|B(x,2b)|\lesssim|B(o,b)|$, so that
$Mf\gtrsim\avg{\sigma_N}{B(o,b)}$ on $B(o,b)$, and therefore
$|B(o,b)|\avg{w}{B(o,b)}\avg{\sigma_N}{B(o,b)}^p\lesssim\int_{B(o,b)}
(Mf)^pw\,d\mu\le\|M\|_{\mathrm{rad}}^p\sigma_N(B(o,b))$.  Thus $\avg{w}{B(o,b)}\avg{\sigma_N}{B(o,b)}^{p-1}\le
C\|M\|_{\mathrm{rad}}^p$.  Since $\Jd$ is increasing, $V([a,b])\ge
V([b/2,b])\ge\frac12V([0,b])$, and letting $N\to\infty$ we obtain
$[w_0]_{A_p(V,I)}\le2^p[w_0]_{A_p(V,[0,b])}\le C\|M\|_{\mathrm{rad}}^p$.

\emph{Case $a>b/2$.}  Then $r<b/2<b$ and $\sinh a\ge c\sinh b$ (if
$b\le2$ because $a>b/2$, and if $b>2$ because $a\ge b-1$), so that
$V(I)\simeq r\sinh^{n-1}b$.  Let $f:=\sigma_N\Chi_{\{|y|\in I\}}$,
which is radial, and fix $x$ with $|x|\in I$.  Put
$\phi_0:=2\sinh(r/2)/\sinh b<1$.  If $|y|\in I$ and $\ang yx\le\phi_0$,
then, with $t:=\min\{|x|,|y|\}$, \cref{lem:angle}\ref{lem:angle:b}
gives $\sinh(\frac12d(\gamma_x(t),\gamma_y(t)))=\sinh t\sin(\ang yx/2)
\le\sinh(r/2)$, so that $d(x,y)\le||x|-|y||+r\le2r$.  Hence the polar
box $\{y:\ |y|\in I,\ \ang yx\le\phi_0\}$ is contained in $B(x,2r)$,
and by \eqref{eq:polarboxint}, \eqref{eq:caps} and \eqref{eq:ballvol},
\[
   Mf(x)\ge A_{2r}f(x)\ge\frac{\nu(\SCap(\theta_x,\phi_0))}{|B(o,2r)|}
   \int_I\sigma_{N,0}\,dV\gtrsim\frac{(r/\sinh b)^{n-1}}{r^n}
   \int_I\sigma_{N,0}\,dV\simeq\frac1{V(I)}\int_I\sigma_{N,0}\,dV .
\]
Since $\|f\|^p_{L^p(w)}\le\int_I\sigma_{N,0}\,dV$, the hypothesis gives
\[
   \Bigl(\frac1{V(I)}\int_I\sigma_{N,0}\,dV\Bigr)^p\int_Iw_0\,dV\le
   C\|M\|_{\mathrm{rad}}^p\int_I\sigma_{N,0}\,dV,
\]
and letting $N\to\infty$ we obtain $[w_0]_{A_p(V,I)}\le
C\|M\|_{\mathrm{rad}}^p$.  The last assertion follows from
\cref{lem:ApV}\ref{lem:ApV:a}.
\end{proof}

\begin{theorem}[Strong type]\label{thm:radialstrong}
Let $1<p<\infty$ and let $w$ be a radial weight.  Then conditions
\textup{(i)--(iv)} of \cref{thmC}\ref{thmC:a} are equivalent, and in
that case $w\in\Ab_{p,\eps}$ for some $\eps>0$.
\end{theorem}

\begin{proof}
(i)$\Rightarrow$(ii) is trivial, and (iii)$\Leftrightarrow$(iv) is
\cref{lem:ApV}\ref{lem:ApV:b}.

\emph{(ii)$\Rightarrow$(iii).}  $w\in\Aloc$ by \cref{lem:radiallocal}.
By \cref{lem:radialtests}\ref{lem:radialtests:a} and the hypothesis,
\[
   \sum_{i<N}W_i=w(B(o,N))\le c^{-p}\|\Minf\Chi_{\mathcal R_N}\|^p_{L^p(w)}
   \le c^{-p}\|M\|_{\mathrm{rad}}^p\|\Chi_{\mathcal R_N}\|^p_{L^p(w)}
   \lesssim W_N,
\]
which is \Ro.  By \cref{lem:radialtests}\ref{lem:radialtests:b},
$\sum_{i\ge0}e^{-2\rho p(i+1)}W_{N+i}\lesssim\|M\Chi_{\mathcal R_N}
\|^p_{L^p(w)}\lesssim W_N$, which is \Rt.

\emph{(iii)$\Rightarrow$(i).}  Let $z\in\Bn$ and $N=\ip{|z|}$.  By
\cref{lem:radialshadow}\ref{lem:radialshadow:c},\ref{lem:radialshadow:b},
\begin{equation}\label{eq_w_sigma_p-1}
   w(\Omega_z(j))\,\sigma(\Omega_z(m))^{p-1}\simeq e^{-2\rho pN}\,
   W_{N+j}\,\Sigma_{N+m}^{p-1}\simeq e^{2\rho pm}\frac{W_{N+j}}{W_{N+m}} .
\end{equation}
Let $\eta,\kappa>0$ be given by \cref{lem:gaps}, and let
$0<\eps\le\min\{p/2,\ \eta/(2\rho),\ \kappa/(2\rho)\}$.  If $j\ge m$,
then \cref{lem:gaps}\ref{lem:gaps:b} with $N+m$ in place of $N$ and
$i=j-m$ bounds the right-hand side of \eqref{eq_w_sigma_p-1} by
$Ce^{2\rho pm}e^{(2\rho p-\eta)(j-m)}=Ce^{2\rho pj-\eta(j-m)}$, and
\[
   2\rho(p-\eps)(j+m)-\bigl(2\rho pj-\eta(j-m)\bigr)
   =(\eta-2\rho\eps)(j-m)+2\rho(p-2\eps)m\ \ge0 .
\]
If $j<m$, then \cref{lem:gaps}\ref{lem:gaps:a} with $N+m$ in place of
$N$ and $i=m-j$ bounds it by $Ce^{2\rho pm}e^{-\kappa(m-j)}$, and
\[
   2\rho(p-\eps)(j+m)-\bigl(2\rho pm-\kappa(m-j)\bigr)
   =2\rho(p-2\eps)j+(\kappa-2\rho\eps)(m-j)\ \ge0 .
\]
In both cases $w(\Omega_z(j))\sigma(\Omega_z(m))^{p-1}\lesssim
e^{2\rho(p-\eps)(j+m)}$, uniformly in $z$.  Hence
$w\in\Aloc\cap\Ab_{p,\eps}$, and \cref{thmA} gives (i).
\end{proof}

We now turn to the weak type.  The next lemma identifies the interval
condition \eqref{eq:Aprw}.

\begin{lemma}\label{lem:Aprw}
Let $1\le p<\infty$ and let $w=w_0(|\cdot|)$ be a radial weight.
\begin{enumerate}[label=\textup{(\alph*)}, leftmargin=2.3em]
\item\label{lem:Aprw:a} If $[w_0]_{A^{\mathrm{KT}}_p(V)}<\infty$, then
      \Ro and \Rtw hold.
\item\label{lem:Aprw:b} If $w\in A_{p,\mathrm{loc}}(\Bn)$ and \Ro and
      \Rtw hold, then $[w_0]_{A^{\mathrm{KT}}_p(V)}<\infty$.
\end{enumerate}
\end{lemma}

\begin{proof}
(a) For \Ro take $I=[N,N+1]$ and $J=[0,N+1]$; then $V(I)\simeq V(J)$ by
\eqref{eq:density}, and $\sum_{i\le N}W_i=\int_Jw_0\,dV\lesssim
\int_Iw_0\,dV=W_N$.  For \Rtw take $I=[N,N+1]$ and $J=[N,N+i+1]$; then
$(V(I)/V(J))^p\gtrsim e^{-2\rho pi}$ and $W_{N+i}\le\int_Jw_0\,dV$, so
that $e^{-2\rho pi}W_{N+i}\lesssim\int_Iw_0\,dV=W_N$.

(b) Let $I=[a,b]\subset J=[a',b']$ with $b-a\ge1$; then $b'\ge1$.  On
the one hand, $V(I)\le V([0,b])\lesssim e^{2\rho b}$ and $V(J)\ge
V([b'-1,b'])\gtrsim e^{2\rho b'}$.  On the other hand, by \Ro,
$\int_Jw_0\,dV\le\sum_{i\le\ip{b'}}W_i\lesssim W_{\ip{b'}}$, while
$\int_Iw_0\,dV\ge w(\mathcal R_{b-1})\simeq W_{\ip b}$ by
\cref{lem:radialshadow}\ref{lem:radialshadow:a}.  By \Rtw,
$W_{\ip{b'}}\lesssim e^{2\rho p(\ip{b'}-\ip b)}W_{\ip b}$.  Hence
\[
   \Bigl(\frac{V(I)}{V(J)}\Bigr)^p\frac{\int_Jw_0\,dV}{\int_Iw_0\,dV}\lesssim
   e^{2\rho p(b-b')}e^{2\rho p(\ip{b'}-\ip b)}\le e^{2\rho p} .\qedhere
\]
\end{proof}

\begin{theorem}[Weak type]\label{thm:radialweak}
Let $1\le p<\infty$ and let $w$ be a radial weight.  The following are
equivalent:
\begin{enumerate}[label=\textup{(\roman*)}, leftmargin=2.3em]
\item\label{thm:radialweak:i} $M:L^p(w)\to L^{p,\infty}(w)$ is bounded;
\item\label{thm:radialweak:ii} $w\in A_{p,\mathrm{loc}}(\Bn)$, and \Ro
      and \Rtw hold;
\item\label{thm:radialweak:iii} $w\in A_{p,\mathrm{loc}}(\Bn)$ and
      $[w_0]_{A^{\mathrm{KT}}_p(V)}<\infty$;
\item\label{thm:radialweak:iv} $w\in A_{p,\mathrm{loc}}(\Bn)$, and $w$
      satisfies \WQ and \WP.
\end{enumerate}
\end{theorem}

\begin{proof}
(ii)$\Leftrightarrow$(iii) is \cref{lem:Aprw}, and (iv)$\Rightarrow$(i)
is \cref{thmB}\ref{thmB:a}.

\emph{(i)$\Rightarrow$(ii).}  The local condition is \cref{prop:locnec}.
By \cref{lem:radialtests}\ref{lem:radialtests:a} and the weak type at
the level $c/2$, $\sum_{i<N}W_i\le w(\{M\Chi_{\mathcal R_N}>c/2\})
\lesssim W_N$, which is \Ro.  By
\cref{lem:radialtests}\ref{lem:radialtests:b}, $M\Chi_{\mathcal R_N}\ge
ce^{-2\rho(i+1)}$ on $\{N\le|x|<N+i+1\}$, and the weak type gives
$W_{N+i}\le w(\{N\le|x|<N+i+1\})\lesssim e^{2\rho p(i+1)}W_N$, which is
\Rtw.

\emph{(ii)$\Rightarrow$(iv).}  Let $y\in\Bn$, $k:=\ip{|y|}$ and
$N\in\Zp$.  By \cref{lem:radialshadow}\ref{lem:radialshadow:c}, \Ro and
\Rtw,
\[
   \frac{\sum_{j\le N}w(\Omega_y(j))}{w(B_y)}\simeq\sum_{j\le N}
   \frac{W_{k+j}}{W_k}\lesssim\frac{W_{k+N}}{W_k}\lesssim e^{2\rho pN},
\]
which is \WQ.  For \WP let $z\in\Bn$, $k:=\ip{|z|}$ and $m\in\Zp$.  If
$p>1$, then by \cref{lem:radialshadow}\ref{lem:radialshadow:c},\ref{lem:radialshadow:b}
and \cref{lem:gaps}\ref{lem:gaps:a},
\[
   \PP_m[w](z)\simeq e^{-2\rho pm}e^{-2\rho k}W_k\,\bigl(e^{-2\rho k}
   \Sigma_{k+m}\bigr)^{p-1}\simeq e^{-2\rho pm}e^{-2\rho pk}W_k\,
   \frac{e^{2\rho p(k+m)}}{W_{k+m}}=\frac{W_k}{W_{k+m}}\lesssim
   e^{-\kappa m}.
\]
If $p=1$, the radii of the points of $\Omega_z(m)$ lie in
$[k+m-1,k+m+2]$, and
\cref{lem:radialshadow}\ref{lem:radialshadow:a},\ref{lem:radialshadow:b}
give $\esssup_{\Omega_z(m)}w^{-1}\lesssim e^{2\rho(k+m)}/W_{k+m}$, so
that again $\PP_m[w](z)\lesssim W_k/W_{k+m}\lesssim e^{-\kappa m}$.
This is \WP with $\eps:=\kappa/(2\rho p)$.
\end{proof}

\begin{proof}[\textbf{Proof of \cref{thmC}}]
Part \ref{thmC:a} is \cref{thm:radialstrong}, and part \ref{thmC:b} is
\cref{thm:radialweak}, which also contains the radial part of
\cref{thmB}\ref{thmB:c}.

For \ref{thmC:c}, note that the weights $\ww_{\alpha,\gamma}$ and
$\ww^\circ_{\alpha,\gamma}$ are comparable away from the origin, and
that $\ww_{\alpha,\gamma}$ is continuous and positive.  Hence
$\ww_{\alpha,\gamma}\in A_{p,\mathrm{loc}}(\Bn)$ for all $\alpha,\gamma$,
while, since $B(o,2)$ is bi-Lipschitz to a Euclidean ball and
$|x|\simeq\|x\|$ there, $\ww^\circ_{\alpha,\gamma}\in
A_{p,\mathrm{loc}}(\Bn)$ if and only if the Euclidean power weight
$\|x\|^\gamma$ belongs to $A_p$, that is, iff $\gamma\in(-n,n(p-1))$ for
$p>1$ and $\gamma\in(-n,0]$ for $p=1$.  If $\gamma>-n$, the two weights
have comparable ring masses, $W_N\simeq e^{(\alpha+2\rho)N}(1+N)^\gamma$
for $N\in\Zp$.  Since \Ro, \Rt and \Rtw depend only on the ring masses,
\cref{thm:radialstrong,thm:radialweak} reduce the assertions for
$\ww^\circ_{\alpha,\gamma}$ to those for $\ww_{\alpha,\gamma}$, together
with the local condition.

For the strong type of $\ww_{\alpha,\gamma}$ we use \cref{thmA} and
\cref{prop:basic}\ref{prop:basic:a}: $M$ is bounded on
$L^p(\ww_{\alpha,\gamma})$ iff $\ww_{\alpha,\gamma}\in\Ab_p$, iff
$\alpha\in(-2\rho,2\rho(p-1))$.  For the weak type we use
\cref{thm:radialweak}.  If $\alpha>-2\rho$, then $\sum_{i\le
N}W_i/W_N\lesssim\sum_ke^{-(\alpha+2\rho)k}(1+k)^{|\gamma|}$; if
$\alpha<-2\rho$, then $W_N\to0$ and $\sum_{i\le N}W_i/W_N\ge
W_0/W_N\to\infty$; if $\alpha=-2\rho$, then $\sum_{i\le N}W_i/W_N\simeq
\sum_{i\le N}((1+i)/(1+N))^{\gamma}\gtrsim_\gamma N$.  So \Ro holds iff
$\alpha>-2\rho$.  Next, $e^{-2\rho pi}W_{N+i}/W_N\simeq
e^{(\alpha+2\rho-2\rho p)i}((1+N+i)/(1+N))^\gamma$, so that \Rtw holds
if $\alpha<2\rho(p-1)$ and fails if $\alpha>2\rho(p-1)$; if
$\alpha=2\rho(p-1)$, then $\sup_{N,i}((1+N+i)/(1+N))^\gamma$ is $1$ for
$\gamma\le0$ and infinite for $\gamma>0$.
\end{proof}

\begin{remark}[The ball condition]\label{rem:ballAp}
That the ball condition \eqref{def_class_Ap} is not necessary for the
boundedness of $M$ was observed on the tree in \cite[Theorem~1.3]{ORR22}
and on $\Hn$ in \cite[Example~4.1(3)]{AO25}, for weights comparable to
$\ww_{-2\rho\gamma,0}$, $0<\gamma<1$.  For balls centred at $o$ the
product in \eqref{def_class_Ap} stays bounded for $\ww_{\alpha,0}$,
$\alpha\in(-2\rho,2\rho(p-1))$; the failure is detected by off-centred
balls, and we include a proof.  Let $\alpha<-\rho$, so that
$\sigma=\ww_{\beta,0}$ with $\beta:=-\alpha/(p-1)$, let $|x|=2R$, $R\ge3$,
and $B:=B(x,R)$.  Then $B\supset B(a_{R-1/2}(x),\frac12)$, so
$\avg wB\gtrsim e^{(\alpha-2\rho)R}$.  If $|y|\in[3R-3,3R-2)$ and
$\ang yx<e^{-2R}$, then $\gp xy\ge2R-\log2$ by \cref{lem:angle}\ref{lem:angle:c} and
$d(x,y)<R$; thus $B\supset\mathcal E_x([3R-3,3R-2);1)$, of volume
$\simeq e^{2\rho R}$, and $\avg\sigma B\gtrsim e^{3\beta R}$.  Hence
$\avg wB\avg\sigma B^{p-1}\gtrsim e^{-2(\alpha+\rho)R}\to\infty$, and by
duality the same happens for $\alpha>\rho(p-1)$.  Thus for
$\alpha\in(-2\rho,-\rho)\cup(\rho(p-1),2\rho(p-1))$, $M$ is bounded on
$L^p(\ww_{\alpha,0})$ although the ball condition fails.
\end{remark}

\section{The range $0<p\le1$: proof of Theorem~\ref{thmD}}
\label{sec:small}

For $p<1$ the obstruction is local: $L^p(w)$ contains functions which
are not locally integrable.  At $p=1$ it comes from the Euclidean small
scales, and the large-scale operator $\Minf$ alone still admits a
two-layer characterisation, as for $p>1$.

\begin{proof}[\textbf{Proof of \cref{thmD}\ref{thmD:a},\ref{thmD:b}}]
(a) Since $w<\infty$ a.e., there are $C_0>0$ and $z_0\in\Bn$ with
$|\{w\le C_0\}\cap B_{z_0}|>0$.  As $\mu$ has no atoms, we can choose
disjoint sets $E_k\subset\{w\le C_0\}\cap B_{z_0}$, $k\in\N$, with
$|E_k|=c\,2^{-k/(1-p)}$ for some small $c>0$, and let
$f:=\sum_k|E_k|^{-1}\Chi_{E_k}$.  Then $\|f\|^p_{L^p(w)}\le
C_0\sum_k|E_k|^{1-p}<\infty$, but $\int_{B_{z_0}}f\,d\mu=\infty$.  Since
every point of $\Bn$ is the centre of a ball of radius larger than $2$
containing $B_{z_0}$, $\Minf f\equiv\infty$, and $\Minf$ is not of weak
type $(p,p)$ with respect to $w$.

(b) Let $x_0$ be a Lebesgue point of $w$ with $0<w(x_0)<\infty$, and
choose $r_1\le1$ such that $|B(x_0,\frac r2)|\le\frac13|B(x_0,r)|$ and
$\frac9{10}w(x_0)\le\avg{w}{B(x_0,r)}\le\frac{11}{10}w(x_0)$ for all
$r\le r_1$.  For $\delta<r_1/2$ let $f:=\Chi_{B(x_0,\delta)}$, and
consider the annuli $D_k:=B(x_0,2^k\delta)\setminus
B(x_0,2^{k-1}\delta)$, $1\le k\le k_\delta:=\ip{\log_2(r_1/\delta)}$.
On $D_k$ we have $Mf\ge c2^{-kn}$, and $w(D_k)\ge\frac12w(x_0)
|B(x_0,2^k\delta)|$.  Hence $\|Mf\|_{L^1(w)}\ge cw(x_0)\delta^nk_\delta$,
whereas $\|f\|_{L^1(w)}\le Cw(x_0)\delta^n$.  Since $k_\delta\to\infty$
as $\delta\to0$, $M$ is unbounded on $L^1(w)$.
\end{proof}

For part (c) let $K(t):=e^{-2\rho t}$, so that condition
\ref{thmD:c}\ref{thmD:c:ii} reads $T_Kw\le Cw$ a.e., where
$T_Kf(x)=\int_{\Bn}K(d(x,y))f(y)\,d\mu(y)$.  Since $T_K$ is positive
with a symmetric kernel, $\|T_K\|_{L^1(w)\to L^1(w)}=\|T_Kw/w\|_
{L^\infty}$.  The self-improvement from $K$ to $e^{\eps t}K$ rests on
the following lower bound for the iterated kernels.

\begin{lemma}\label{lem:convpowers}
Let $\mathcal K_1(x,y):=e^{-2\rho d(x,y)}$ and $\mathcal K_{i+1}(x,y):=
\int_{\Bn}\mathcal K_i(x,z)\,e^{-2\rho d(z,y)}\,d\mu(z)$, $i\ge1$, so
that $\mathcal K_i$ is the kernel of $T_K^i$.  Put
$c_*:=e^{-2\rho}|B(o,\frac12)|$.  Then, with the convention
$\binom ab=0$ for $b>a$,
\begin{equation}\label{eq_K_i>}
   \mathcal K_i(x,y)\ge c_*^{\,i-1}\binom{\ip{d(x,y)}}{i-1}
   e^{-2\rho d(x,y)},\qquad x,y\in\Bn,\ i\ge1 .
\end{equation}
\end{lemma}

\begin{proof}
For $i=1$ there is nothing to prove.  Assume \eqref{eq_K_i>} for $i$,
let $x,y\in\Bn$, and let $\gamma$ be the unit speed geodesic from $x$
to $y$.  The balls $B(\gamma(k),\frac12)$, $1\le k\le\ip{d(x,y)}$, are
pairwise disjoint.  If $z\in B(\gamma(k),\frac12)$, then
$k-\frac12<d(x,z)<k+\frac12$ and $d(z,y)<d(x,y)-k+\frac12$, so that
$\ip{d(x,z)}\ge k-1$ and, by the induction hypothesis,
\[
   \mathcal K_i(x,z)\,e^{-2\rho d(z,y)}\ge c_*^{\,i-1}\binom{k-1}{i-1}
   e^{-2\rho(k+\frac12)}e^{-2\rho(d(x,y)-k+\frac12)}
   =c_*^{\,i-1}e^{-2\rho}\binom{k-1}{i-1}e^{-2\rho d(x,y)}.
\]
Integrating over $B(\gamma(k),\frac12)$, summing over $1\le
k\le\ip{d(x,y)}$ and using $\sum_{k=1}^{\ip{d(x,y)}}\binom{k-1}{i-1}=
\binom{\ip{d(x,y)}}i$, we obtain \eqref{eq_K_i>} for $i+1$.
\end{proof}

\begin{proof}[\textbf{Proof of \cref{thmD}\ref{thmD:c}}]
\ref{thmD:c:i}$\Rightarrow$\ref{thmD:c:ii}.  Let $y_0$ be a Lebesgue
point of $w$, $0<\delta<1$ and $f=\Chi_{B(y_0,\delta)}$.  If
$d(x,y_0)\ge2$, then by the triangle inequality the ball
$B(x,d(x,y_0)+\delta)$, whose radius exceeds $2$, contains
$B(y_0,\delta)$, and its volume is at most $Ce^{2\rho d(x,y_0)}$ by
\eqref{eq:ballvol}, since $\delta<1$; if $d(x,y_0)<2$, the ball $B(x,3)$
contains $B(y_0,\delta)$.  In both cases $\Minf f(x)\ge
c|B(y_0,\delta)|e^{-2\rho d(x,y_0)}$.  Hence
\[
   c\int_{\Bn}e^{-2\rho d(x,y_0)}w(x)\,d\mu(x)\le\|\Minf\|_{L^1(w)\to
   L^1(w)}\,\avg{w}{B(y_0,\delta)}\longrightarrow\|\Minf\|\,w(y_0)
   \qquad(\delta\to0).
\]

\ref{thmD:c:ii}$\Rightarrow$\ref{thmD:c:iii}.  Assume that
$T_Kw\le C_1w$ a.e.  Since $T_K$ is positive, we have $T^i_Kw\le
C_1^iw$ a.e.\ for all $i$, that is,
$\int_{\Bn}\mathcal K_i(x,y)\,w(y)\,d\mu(y)\le C_1^i\,w(x)$.  Let
$t_0:=(2C_1)^{-1}$.  By \cref{lem:convpowers} and the binomial theorem,
\[
   \sum_{i\ge1}t_0^{i-1}\mathcal K_i(x,y)\ge e^{-2\rho d(x,y)}
   \sum_{i\ge1}\binom{\ip{d(x,y)}}{i-1}(t_0c_*)^{i-1}
   =e^{-2\rho d(x,y)}(1+t_0c_*)^{\ip{d(x,y)}},
\]
and $(1+t_0c_*)^{\ip{d(x,y)}}\ge(1+t_0c_*)^{-1}e^{\eps d(x,y)}$ with
$\eps:=\log(1+t_0c_*)$.  Therefore
\[
   \int_{\Bn}e^{-(2\rho-\eps)d(x,y)}w(y)\,d\mu(y)\le(1+t_0c_*)
   \sum_{i\ge1}t_0^{i-1}C_1^iw(x)=(2C_1+c_*)\,w(x).
\]

\ref{thmD:c:iii}$\Rightarrow$\ref{thmD:c:iv}.  Since $|B(x,\ell)|\ge
ce^{2\rho\ell}$ and $e^{-(2\rho-\eps)d(x,y)}\ge e^{-(2\rho-\eps)\ell}$ if
$d(x,y)\le\ell$, Fubini's theorem gives
\[
   \|\ann_\ell f\|_{L^1(w)}=\int_{\Bn}|f(y)|\int_{\{\ell-1<d(x,y)\le\ell\}}
   \frac{w(x)}{|B(x,\ell)|}\,d\mu(x)\,d\mu(y)\le Ce^{-\eps\ell}\|f\|_{L^1(w)} .
\]

\ref{thmD:c:iv}$\Rightarrow$\ref{thmD:c:i} follows from
\eqref{eq:ballfromannuli}.

\ref{thmD:c:iii}$\Rightarrow$\ref{thmD:c:v}.  Let $z\in\Bn$ and
$j,m\in\Zp$.  If $x\in\Omega_z(j)$ and $y\in\Omega_z(m)$, then
$d(x,y)\le j+m+2s_0$.  By \ref{thmD:c:iii}, for a.e.\
$x\in\Omega_z(j)$, $w(x)\ge C^{-1}e^{-(2\rho-\eps)(j+m+2s_0)}
w(\Omega_z(m))$.  Hence $w(\Omega_z(m))\|w^{-1}\|_{L^\infty(\Omega_z(j))}
\lesssim e^{2\rho(1-\eps/(2\rho))(j+m)}$, which is \ref{thmD:c:v} with
$\eps/(2\rho)$, since the supremum in the definition of $\Ab_{1,\eps}$
is symmetric in $j$ and $m$.

\ref{thmD:c:v}$\Rightarrow$\ref{thmD:c:iii}.  The proof of
\cref{cor:slack} applies verbatim with $\sigma(F)^{p-1}$ replaced by
$\|w^{-1}\|_{L^\infty(F)}$, and gives, for $s\ge s_0$,
\[
   w\bigl(\Omega^{(s)}_z(m)\bigr)\,\|w^{-1}\|_{L^\infty(\Omega^{(s)}_z(j))}
   \le C_s[w]_{\Ab_{1,\eps}}e^{2\rho(1-\eps)(j+m)} .
\]
Let $X\subset\Bn$ be countable and dense.  Outside a null set $E$, the
inequality $w(\Omega^{(s_0+1)}_{z'}(m))\le C_{s_0}[w]_{\Ab_{1,\eps}}
e^{2\rho(1-\eps)(j+m)}w(x)$ holds for all $z'\in X$, $j,m\in\Zp$ and
$x\in\Omega^{(s_0+1)}_{z'}(j)$.  Let $x\notin E$, $y\in\Bn$ and
$\ell-1<d(x,y)\le\ell$.  By \cref{lem:decomp} there are $j\le|x|$ and
$m$ with $\ell-3\le j+m\le\ell$ and $y\in\Omega_{a_j(x)}(m)\subset
\Omega^{(s_0)}_{a_j(x)}(m)$.  Choose $z'\in X$ with $d(z',a_j(x))<1$.
By \cref{lem:shadow}\ref{lem:shadow:e},\ref{lem:shadow:f}, $\Omega^{(s_0)}_{a_j(x)}(m)\subset
\Omega^{(s_0+1)}_{z'}(m)$ and $x\in\Omega_{a_j(x)}(j)\subset
\Omega^{(s_0+1)}_{z'}(j)$.  Summing over the at most $4(\ell+1)$ pairs
in $\Gamma_\ell$, we get $w(\{y:\ \ell-1<d(x,y)\le\ell\})\le
C(\ell+1)e^{2\rho(1-\eps)\ell}w(x)$, and therefore, for
$\eps'<2\rho\eps$,
\[
   \int_{\Bn}e^{-(2\rho-\eps')d(x,y)}w(y)\,d\mu(y)\le C\sum_{\ell\ge1}
   (\ell+1)\,e^{-(2\rho-\eps')(\ell-1)}e^{2\rho(1-\eps)\ell}w(x)\le C'w(x).
\]

Finally, for $w=\ww_{\alpha,\gamma}$, \cref{lem:radialshadow}\ref{lem:radialshadow:c} and
$w\simeq e^{\alpha(N+m)}(1+N+m)^\gamma$ on $\Omega_z(m)$, $N=\ip{|z|}$,
show that the quotient defining $\Ab_{1,\eps}$ is comparable to
$e^{(2\rho+\alpha)j-\alpha m-2\rho(1-\eps)(j+m)}((1+N+j)/(1+N+m))^{\gamma}$,
which is bounded for small $\eps$ iff $-2\rho<\alpha<0$.  For
$w^\circ:=\ww^\circ_{\alpha,\gamma}$ with $\gamma>-n$, the quotient
$w^\circ/\ww_{\alpha,\gamma}$ is bounded above and below off $B_o$.  If
$\gamma>0$, then $w^\circ(x)\to0$ as $x\to o$, while the left-hand side
of \ref{thmD:c:ii} stays bounded below, so that \ref{thmD:c:ii} fails.
If $\gamma\le0$, then $w^\circ\ge\ww_{\alpha,\gamma}$, and
$\int_{B_o}e^{-2\rho d(x,y)}w^\circ(y)\,d\mu(y)\le Ce^{-2\rho|x|}$,
which is $\lesssim\ww_{\alpha,\gamma}(x)$ when $\alpha>-2\rho$; hence
\ref{thmD:c:ii} holds for $w^\circ$ if and only if it holds for
$\ww_{\alpha,\gamma}$.
\end{proof}

The class $\Ab_1$ contains no condition at scales below $1$, in
accordance with \cref{thmD}\ref{thmD:b}.

\begin{remark}\label{rem:coarse}
The behaviour of the tree is recovered by excluding the functions which
are not locally integrable.  Fix a partition $\cells$ of $\Bn$ into
cells $S$ of diameter at most $D$ and measure comparable to $1$, and let
$\Mcoarse$ be $M$ acting on the functions $f=\sum_{S\in\cells}b_S\Chi_S$
which are constant on each cell; we may assume $b_S\ge0$.  A ball about
$x$ which meets $S$ has radius at least $d(x,S)$, while
$B(x,d(x,S)+D)\supset S$; hence $M\Chi_S(x)\simeq e^{-2\rho d(x,S)}$.
Since $M$ is sublinear and $(\sum_Sa_S)^p\le\sum_Sa_S^p$ for $a_S\ge0$
and $0<p\le1$,
\[
   \int_{\Bn}(Mf)^pw\,d\mu\le\sum_{S\in\cells}b_S^p\int_{\Bn}
   (M\Chi_S)^pw\,d\mu\lesssim\sum_{S\in\cells}b_S^p\int_{\Bn}w(x)\,
   e^{-2\rho p\,d(x,S)}\,d\mu(x),
\]
while $\|f\|^p_{L^p(w)}=\sum_Sb_S^pw(S)$.  Together with the test
functions $f=\Chi_S$, this shows that $\Mcoarse$ is bounded on
$L^p(w)$, $0<p\le1$, if and only if
\begin{equation}\label{eq:coarse}
   \sup_{S\in\cells}\frac1{w(S)}\int_{\Bn}w(x)\,e^{-2\rho p\,d(x,S)}\,
   d\mu(x)<\infty .
\end{equation}
Let $K_p(t):=e^{-2\rho pt}$, so that $\widehat{K_p}(0)\simeq
\int_0^\infty(1+t)e^{\rho(1-2p)t}\,dt$ by \eqref{eq:phieta}.  If
$p>\frac12$, then $\widehat{K_p}(0)<\infty$ and
$T_{K_p}\varphi_0=\widehat{K_p}(0)\varphi_0$ by \eqref{eq:meanvalueK}.
Since $\varphi_0$ is comparable to a constant on each cell and
$d(x,S)\ge d(x,y_S)-D$ for any $y_S\in S$, the integral in
\eqref{eq:coarse} with $w=\varphi_0$ is at most $CT_{K_p}\varphi_0(y_S)
\simeq\varphi_0(S)$, so that $\varphi_0$ satisfies \eqref{eq:coarse}.
If $p\le\frac12$, no weight satisfies \eqref{eq:coarse}.  Indeed, for
$\bar w:=\sum_{S\in\cells}\avg{w}{S}\Chi_S$ and $x\in S'\in\cells$,
\eqref{eq:coarse} gives
\[
   T_{K_p}\bar w(x)\lesssim\sum_{S\in\cells}w(S)\,e^{-2\rho p\,d(S,S')}
   \lesssim\int_{\Bn}w(y)\,e^{-2\rho p\,d(y,S')}\,d\mu(y)\lesssim w(S')
   \simeq\bar w(x),
\]
so that $T_{K_p}$ is bounded on $L^1(\bar w)$, whereas
$\widehat{K_p}(0)=\infty$, contradicting
\cref{thm:herz}\ref{thm:herz:a}.  Thus the critical index $\frac12$ of
\cite[Theorem~1.3]{YYZ26} reappears for $\Mcoarse$.
\end{remark}

\section{Applications}\label{sec:applications}

\subsection{Consequences of exponential decay}\label{subsec:decayapps}
Let $1<p<\infty$ and let $\Minf$ be bounded on $L^p(w)$, so that
$\|\ann_\ell\|_{L^p(w)\to L^p(w)}\le Ce^{-\eps\ell}$ for some $\eps>0$
(\cref{thm:largescale}); where $M$ is concerned, let also $w\in\Aloc$.

\begin{theorem}[Vector-valued inequalities]\label{thm:FS}
\begin{enumerate}[label=\textup{(\alph*)}, leftmargin=2.3em]
\item\label{thm:FS:a} For every $q\in[1,\infty]$ and every sequence $(f_k)$,
\[
   \Bigl\|\Bigl(\sum_k(\Minf f_k)^q\Bigr)^{1/q}\Bigr\|_{L^p(w)}
   \le C\,\Bigl\|\Bigl(\sum_k|f_k|^q\Bigr)^{1/q}\Bigr\|_{L^p(w)},
\]
with $C$ independent of $q$ (usual modification for $q=\infty$).
\item\label{thm:FS:b} For $q\in(1,\infty]$ the same holds with $M$ in place of
$\Minf$, with $C$ depending on $q$.
\item\label{thm:FS:c} For $q=1$ and $w\equiv1$, \ref{thm:FS:b} fails.
\item\label{thm:FS:d} Conversely, if the inequality in \ref{thm:FS:a} holds for one
$q\in[1,\infty]$, then $\Minf$ is bounded on $L^p(w)$.
\end{enumerate}
\end{theorem}

\begin{proof}
(a) Pointwise $\Minf f\le C\sum_\ell\ann_\ell|f|$ by
\eqref{eq:ballfromannuli}.  For $q<\infty$, Minkowski's inequality in
$\ell^q$ gives $(\sum_k(\sum_\ell\ann_\ell|f_k|)^q)^{1/q}\le\sum_\ell
(\sum_k(\ann_\ell|f_k|)^q)^{1/q}$, and Minkowski's integral inequality
for the average $\ann_\ell$ gives $(\sum_k(\ann_\ell|f_k|)^q)^{1/q}\le
\ann_\ell F$, where $F:=(\sum_k|f_k|^q)^{1/q}$.  Hence
$(\sum_k(\Minf f_k)^q)^{1/q}\le C\sum_\ell\ann_\ell F$; for $q=\infty$,
$\sup_k\Minf f_k\le C\sum_\ell\ann_\ell(\sup_k|f_k|)$ directly.  Since
$\|\sum_\ell\ann_\ell F\|\le\sum_\ell\|\ann_\ell\|\,\|F\|\le
C\sum_\ell e^{-\eps\ell}\|F\|$, (a) follows with $C$ independent of $q$.

(b) By (a) it remains to treat $\Mz$.  Let $\{z_i\}$ be a maximal
$1$-separated set and $B^*_{z_i}:=B(z_i,4)$; the balls $B^*_{z_i}$
have bounded overlap (\cref{sec:prelim}).  For $x\in B_{z_i}$ and
$r\le2$ we have $B(x,r)\subset B^*_{z_i}$, so that $\Mz f\le
M_i(f\Chi_{B^*_{z_i}})$ on $B_{z_i}$, where $M_i$ is the maximal
operator of the space of homogeneous type $(B^*_{z_i},d,\mu)$.  Since
hyperbolic balls are geodesically convex, for $x\in B^*_{z_i}$ and
$r\le8$ the set $B(x,r)\cap B^*_{z_i}$ contains a ball of radius $r/4$,
so that $|B(x,r)\cap B^*_{z_i}|\ge c|B(x,r)|$; hence, by
\ref{item_L21}, the restriction of $w$ to $B^*_{z_i}$ is an $A_p$
weight of this space, uniformly in $i$.  For $1<q<\infty$ the weighted
Fefferman--Stein inequality on $B^*_{z_i}$ follows from the scalar
weighted inequality by Rubio de Francia extrapolation, whose proof
\cite[Theorem~3.9]{CUMP11} applies verbatim on spaces of homogeneous
type \cite{ST89}, with constants depending only on the doubling and
$A_p$ constants.  For
$q=\infty$ it is the scalar inequality.  Summing over $i$ gives (b).

(c) The Euclidean counterexample \cite[Ch.~II, \S5.1]{Ste93} can be
realized inside $B_o$, where the metric and the measure are comparable
to the Euclidean ones: with $N$ slabs $E_k$ of Euclidean width $1/N$
foliating $B_o$, $\sum_kM\Chi_{E_k}\gtrsim\log N$ on $B(o,\frac12)$,
while $\sum_k\Chi_{E_k}=\Chi_{B_o}$.

(d) Apply the inequality to a sequence with a single nonzero term.
\end{proof}

\subsection{Spherical functions and a weighted Herz principle}
\label{subsec:spherical}
Here $w$ is an arbitrary weight, and $K:[0,\infty)\to[0,\infty)$ is a
radial kernel with
\[
   T_Kf(x):=\int_{\Bn}K(d(x,y))f(y)\,d\mu(y),\qquad
   \widehat K(\lambda):=\int_{\Bn}K(|x|)\,\varphi_\lambda(x)\,d\mu(x),
\]
the spherical transform \eqref{eq:sphtransform} of $K(|\cdot|)$.  On
$\Bn\cong SO_0(n,1)/SO(n)$, $T_Kf$ is the convolution of $f$ with
$K(|\cdot|)$.  By the mean value identity \eqref{eq:meanvalue}, if
$\varphi>0$ and $\Delta\varphi=(\eta^2-\rho^2)\varphi$ with $\eta\in\R$,
then
\begin{equation}\label{eq:meanvalueK}
   T_K\varphi=\widehat K(i\eta)\,\varphi .
\end{equation}

\begin{theorem}[Weighted Herz principle]\label{thm:herz}
Let $K$ be a radial kernel and $p\in[1,\infty)$.
\begin{enumerate}[label=\textup{(\alph*)}, leftmargin=2.3em]
\item\label{thm:herz:a} For every weight $w$,\ $\|T_K\|_{L^p(w)\to L^p(w)}\ge
      \|T_K\|_{L^2\to L^2}=\widehat K(0)$.
\item\label{thm:herz:b} If $\varphi>0$ and $\Delta\varphi=(\eta^2-\rho^2)\varphi$, $\eta\in\R$,
      then $\|T_K\|_{L^p(\varphi^{2-p})\to L^p(\varphi^{2-p})}\le
      \widehat K(i\eta)$.
\item\label{thm:herz:c} $\|T_K\|_{L^p(\varphi_0^{2-p})\to L^p(\varphi_0^{2-p})}=\widehat K(0)$,
      and $T_K$ is bounded on some weighted $L^p$ space if and only if
      $\widehat K(0)<\infty$.
\end{enumerate}
\end{theorem}

\begin{proof}
(a) Since the kernel of $T_K$ is symmetric, $T_K$ has the same norm on
the dual space $L^{p'}(w^{1-p'})$ (on $\{f:\ f/w\in L^\infty\}$ if
$p=1$).  The Stein--Weiss interpolation theorem with change of measures
\cite{SW58} gives $\|T_K\|_{L^2}\le\|T_K\|_{L^p(w)\to L^p(w)}$, since
$w^{1/(2p)}w^{(1-p')/(2p')}=1$.  On $L^2$, $T_K$ is a Fourier
multiplier: by \eqref{eq_helgaosn_F} and the mean value identity,
$\widetilde{T_Kf}(\lambda,\xi)=\widehat K(\lambda)\widetilde f(\lambda,\xi)$,
so by the Plancherel theorem $\|T_K\|_{L^2}=\sup_{\lambda\in\R}
|\widehat K(\lambda)|$.  Since $K\ge0$ and $|\varphi_\lambda|\le\varphi_0$,
this supremum is attained at $\lambda=0$, whence
$\|T_K\|_{L^2}=\widehat K(0)$.

(b) Let $d\mu_\varphi:=\varphi^2d\mu$ and $T_{K,\varphi}f:=\varphi^{-1}
T_K(\varphi f)$.  The kernel of $T_{K,\varphi}$ with respect to
$\mu_\varphi$ is $K(d(x,y))/(\varphi(x)\varphi(y))$, which is symmetric,
and $T_{K,\varphi}1=\widehat K(i\eta)$ by \eqref{eq:meanvalueK}.  Jensen's
inequality gives $|T_{K,\varphi}f|^p\le\widehat K(i\eta)^{p-1}
T_{K,\varphi}(|f|^p)$, and $\int T_{K,\varphi}(|f|^p)\,d\mu_\varphi=
\widehat K(i\eta)\int|f|^p\,d\mu_\varphi$ by symmetry.  Hence
$\|T_{K,\varphi}\|_{L^p(\mu_\varphi)}\le\widehat K(i\eta)$, and the claim
follows because $f\mapsto f/\varphi$ maps $L^p(\varphi^{2-p})$
isometrically onto $L^p(\mu_\varphi)$ and $T_Kf=\varphi\,
T_{K,\varphi}(f/\varphi)$.

(c) Take $\varphi=\varphi_0$, $\eta=0$, in (b), and combine with (a).
\end{proof}

For $p=2$, (c) is Herz's \emph{principe de majoration} \cite{Her70}.

\begin{corollary}\label{cor:extremal}
For every weight $w$, $p\in[1,\infty)$ and $\ell\in\N$,
$\|\ann_\ell\|_{L^p(w)}\ge|B(o,\ell)|^{-1}\int_{\ell-1}^\ell\varphi_0\Jd\,dt\ge
c(1+\ell)e^{-\rho\ell}$, with equality in the first inequality for
$w=\varphi_0^{2-p}$.  In particular the factor $\ell+1$ in
\eqref{eq:annulardecay} is sharp: $\ww_{\rho(p-2),0}\in\Ab_{p,p/2}$
and $\|\ann_\ell\|_{L^p(\ww_{\rho(p-2),0})}\simeq(1+\ell)e^{-\rho\ell}
\simeq\varphi_0(\ell)$.
\end{corollary}

\begin{proof}
The operator $\ann_\ell$ is $T_K$ with $K=|B(o,\ell)|^{-1}
\Chi_{(\ell-1,\ell]}$, so the first assertion is \cref{thm:herz}, and
the lower bound follows from \eqref{eq:phieta} and \eqref{eq:density}.
Finally $\ww_{\rho(p-2),0}\in\Ab_{p,p/2}$
by \cref{prop:basic}\ref{prop:basic:a}, and \eqref{eq:annulardecay} with $\eps=p/2$
gives the matching upper bound.
\end{proof}

\begin{corollary}[Weighted Kunze--Stein phenomenon]\label{cor:KS}
Let $p\in[1,\infty)$.  If $1\le q<2$ and $K\in L^q(\Bn)$ is radial, then
$\|T_Kf\|_{L^p(\varphi_0^{2-p})}\le\|\varphi_0\|_{L^{q'}}\|K\|_{L^q}
\|f\|_{L^p(\varphi_0^{2-p})}$, with $\|\varphi_0\|_{L^{q'}}<\infty$.  If
$2\le q\le\infty$, there is a radial $0\le K\in L^q(\Bn)$ with
$\widehat K(0)=\infty$, so that $T_K$ is bounded on no weighted $L^p$
space.
\end{corollary}

\begin{proof}
By \cref{thm:herz}\ref{thm:herz:c} applied to $|K|$ and H\"older's inequality,
$\|T_K\|\le\widehat{|K|}(0)\le\|K\|_{L^q}\|\varphi_0\|_{L^{q'}}$, and
$\varphi_0\in L^{q'}$ if and only if $q'>2$ by \eqref{eq:phieta}.  For
$q\ge2$ take $K(t)=(1+t)^{-2/q}e^{-2\rho t/q}$.  Thus the Kunze--Stein
phenomenon \cite{KS60,Cow78} survives, for radial kernels, on every
$L^p(\varphi_0^{2-p})$; {see \cite[p.~657]{ADY96},
\cite[Corollary~4.5]{RS09}.}
\end{proof}

\begin{remark}[Heat maximal operator]\label{rem:heat}
The heat kernel of $\Bn$ is $h_t(d(x,y))$, where $h_t>0$,
$\int_{\Bn}h_t(|x|)\,d\mu(x)=1$, and $r\mapsto h_t(r)$ is decreasing
\cite[Theorem~2.1]{DM88}, \cite{GN98}; see \cite[Theorem~5.9]{ADY96} for sharp estimates.  Writing $h_t(d(x,y))=
\int_{d(x,y)}^{\infty}d(-h_t)(s)$, Tonelli's theorem gives
\[
   |e^{t\Delta}f(x)|\le\int_0^\infty\Bigl(\int_{B(x,s)}|f|\,d\mu\Bigr)
   \,d(-h_t)(s)\le Mf(x)\int_0^\infty|B(o,s)|\,d(-h_t)(s)=Mf(x),
\]
where the last integral equals $\int_{\Bn}h_t\,d\mu=1$ by the same
computation with $f\equiv1$ (see \cite[Theorem~2.1.10]{Gra14},
\cite[Ch.~III, \S2.2]{Ste70}).  Hence the heat
maximal operator $\mathcal M_\Delta f:=\sup_{t>0}|e^{t\Delta}f|$
satisfies $\mathcal M_\Delta f\le Mf$, and every strong or weak type
bound for $M$ on $L^p(w)$ holds for $\mathcal M_\Delta$ with the same
constant.  Moreover, $e^{t\Delta}\varphi_0=e^{-\rho^2t}\varphi_0$ by
\eqref{eq:meanvalueK}, so $H_tf:=\varphi_0^{-1}e^{t(\Delta+\rho^2)}
(\varphi_0f)$ is a symmetric diffusion semigroup on
$L^2(\varphi_0^2\,d\mu)$, and Stein's maximal theorem
\cite[Ch.~III]{Ste70} shows that $\sup_t|e^{t(\Delta+\rho^2)}F|$ is
bounded on $L^p(\varphi_0^{2-p})$, $1<p<\infty$; this does not follow
from the first part, since $e^{\rho^2t}h_t$ has mass $e^{\rho^2t}$.
\end{remark}

\subsection{Powers of the Poisson kernel}\label{subsec:horocyclic}
Fix $\xi\in\Sph$ and $s\in\R$.  The weight $P(\cdot,\xi)^s=
e^{2\rho s\langle\cdot,\xi\rangle}$ is constant on horospheres tangent
at $\xi$.  We pass to the upper half-space model
\[
   \Un:=\{x=(x',x_n):\ x'\in\R^{n-1},\ x_n>0\},\qquad
   ds^2=\frac{\|dx'\|^2+dx_n^2}{x_n^2},\qquad d\mu(x)=x_n^{-n}\,dx'\,dx_n,
\]
by a M\"obius transformation mapping $\Bn$ onto $\Un$, $o$ to
$o:=(0,1)$ and $\xi$ to $\infty$ \cite[\S4.5--4.6]{Rat19}.  Since
isometries preserve horocyclic brackets with the base point,
$\langle\cdot,\xi\rangle$ becomes $A(x):=\log x_n$, and the weight
becomes $e^{2\rho sA(x)}=x_n^{2\rho s}$.  The product
$xy:=(x'+x_ny',\,x_ny_n)$ makes $\Un$ a group $NA$ with identity $o$,
whose left translations are isometries; thus $\mu$ is a left Haar
measure, $d(x,xy)=|y|$, and $A$ is a homomorphism.  We use two facts.

(H1) For $y\in\Un$ the substitution $x\mapsto xy$ gives
$\int_{\Un}f(xy)\,d\mu(x)=e^{2\rho A(y)}\int_{\Un}f\,d\mu$, so that
$R_yf:=e^{-2\rho A(y)/p}f(\cdot\,y)$ is an isometry of $L^p(\Un)$.

(H2) For $\lambda_0\in\R$, \eqref{eq:horoeigen} gives $\Delta
e^{\lambda_0A}=\lambda_0(\lambda_0-2\rho)e^{\lambda_0A}$, so that
\eqref{eq:meanvalueK} at $o$, with $\eta=\lambda_0-\rho$, yields
$\int_{\Un}K(|y|)\,e^{\lambda_0A(y)}\,d\mu(y)=\widehat K(i(\lambda_0-\rho))$.

\begin{proof}[\textbf{Proof of \cref{thmE}}]
Part (a) is \cref{thm:herz}\ref{thm:herz:a},\ref{thm:herz:c}.  For (b) we work on $\Un$, where
the weight is $e^{2\rho sA}$.  Let $\eta=\frac{2\rho(1-s)}p-\rho$ and
$\lambda_0:=\eta+\rho$.  Given $F\in L^p(e^{2\rho sA})$, we put
$f:=e^{2\rho sA/p}F$, so that $\|f\|_{L^p}=\|F\|_{L^p(e^{2\rho sA})}$.
Substituting $y\mapsto xy$ and using that $A$ is a homomorphism,
\begin{equation}\label{eq:TKhoro}
   e^{2\rho sA(x)/p}\,T_KF(x)=\int_{\Un}K(|y|)\,e^{2\rho sA(x)/p}F(xy)\,
   d\mu(y)=\int_{\Un}(R_yf)(x)\,e^{\lambda_0A(y)}K(|y|)\,d\mu(y).
\end{equation}
Minkowski's inequality, (H1) and (H2) give $\|T_K\|\le\widehat K(i\eta)$.
Conversely, by monotone convergence we may assume that $K$ is bounded with support in $[0,R]$.  Since $NA$ is solvable, it is amenable and
has Reiter's property $(P_p)$ \cite[Definition~8.3.1 and
Theorem~8.3.2]{RS00}, \cite[Proposition~6.12]{Pie84}: for every
$\delta>0$ there is $\psi\ge0$ with $\|\psi\|_{L^p}=1$ and
$\|\psi(y^{-1}\cdot)-\psi\|_{L^p}\le\delta$ for $|y|\le R$.  The map
$Jf(x):=e^{2\rho A(x)/p}f(x^{-1})$ is an isometry of $L^p(\Un)$, and
$J(R_yf)=(Jf)(y^{-1}\cdot)$.  Hence $f:=J^{-1}\psi$ satisfies
$\|f\|_{L^p}=1$ and $\|R_yf-f\|_{L^p}\le\delta$ for $|y|\le R$.  By
\eqref{eq:TKhoro} and (H2), with $F:=e^{-2\rho sA/p}f$,
\[
   \bigl\|e^{2\rho sA/p}T_KF-\widehat K(i\eta)f\bigr\|_{L^p}
   \le\int_{|y|\le R}\|R_yf-f\|_{L^p}\,e^{\lambda_0A(y)}K(|y|)\,d\mu(y)
   \le\delta\,\widehat K(i\eta),
\]
so that $\|T_KF\|_{L^p(e^{2\rho sA})}\ge(1-\delta)\widehat K(i\eta)
\|F\|_{L^p(e^{2\rho sA})}$, and $\delta\to0$ gives the norm identity.
Applying it to $\ann_\ell=T_K$, $K=|B(o,\ell)|^{-1}\Chi_{(\ell-1,\ell]}$,
we get from \eqref{eq:phieta} and \eqref{eq:density} that
$\|\ann_\ell\|_{L^p(P(\cdot,\xi)^s)}$ is comparable to
$(1+\ell)e^{-\rho\ell}$ if $\eta=0$ and to $e^{-(\rho-|\eta|)\ell}$
otherwise.  If $|\eta|<\rho$, the annular norms are summable, and
\eqref{eq:ballfromannuli} and \ref{item_L1} apply, the weight being
continuous and positive.  If $|\eta|\ge\rho$, they do not decay, and
\cref{thm:largescale}, resp.\ \cref{thmD}\ref{thmD:c}\ref{thmD:c:iv}, shows unboundedness.
Finally, $|\eta|<\rho$ means $1-p<s<1$.
\end{proof}

\section{Concluding remarks and open questions}\label{sec:remarks}

(1) \emph{Two-layer $A_{p,q}$ weights.}  Let $0\le\gamma<1$, $1<p\le
q<\infty$, $\sigma=w^{-p'}$ and $\Minf_\gamma f:=\sup_{r>2}
|B(\cdot,r)|^\gamma A_rf$.  By the proof of \eqref{eq:annulardecay},
$\Minf_\gamma$ maps $L^p(w^p)$ to $L^q(w^q)$ if the products
$w^q(\Omega_z(j))^{1/q}\sigma(\Omega_z(m))^{1/p'}$ are at most
$Ce^{2\rho(1-\gamma-\eps)(j+m)}$ for some $\eps>0$ (see \cite{IFR25}). It will be interesting to know if this condition is necessary?

(2) \emph{Endpoint estimates for the heat maximal operator.}  Since
$\sup_th_t(r)\simeq(1+r)^{-1/2}e^{-2\rho r}$ for $r\ge1$
\cite{DM88,ADY96}, the classes of $\mathcal M_\Delta$ and $M$ can
differ for $\ww_{\alpha,\gamma}$ only at $\alpha\in\{-2\rho,2\rho(p-1)\}$.
Is $\mathcal M_\Delta$ bounded on $L^p(\ww_{2\rho(p-1),\gamma})$ for some
$\gamma$?

(3) \emph{Fourier inequalities with $A_p$-type weights.}  Can the
Pitt type inequalities of \cite{GRT26} for polynomial weights be
extended to $A_p$-type weights such as $\varphi_{i\eta}$?  These
functions drive \cref{thm:herz,thmE}, and are used in \cite{GRT26} to
define general monotone functions in hyperbolic geometry.

(4) \emph{Other spaces.}  
Our geometric input, namely \cref{lem:angle,lem:shadow},  has analogues on Damek--Ricci spaces
\cite{ADY96,GRS25}: they are Gromov hyperbolic, so that the angle at
$o$ can be replaced by a visual quasi-metric \cite[III.H.3]{BH99} and
shadows behave as in \cref{lem:shadow}.  This makes Damek--Ricci spaces
a natural setting for a two-layer theory, as are Jacobi hypergroups
with arbitrary parameters, to which \cref{thm:radialstrong} might
extend.
{\subsection*{Acknowledgments}
The author is supported by the FWO Odysseus 1 grant G.0H94.18N:
Analysis and Partial Differential Equations, by the Methusalem program
of the Ghent University Special Research Fund (BOF), project
BOFMET2021000601, and by a BOF postdoctoral fellowship at Ghent
University, BOF24/PDO/025.
\vspace{1mm}

\noindent \textbf{Declaration on the use of AI:}
The author used AI tools (Claude) to generate the TikZ
figures, refine the writing, and assist in developing the arguments for
\cref{thmD}\ref{thmD:b} and the lower bound in \cref{thmE}\ref{thmE:b}.
All mathematical arguments were written and verified by the author, who
takes full responsibility for the content.
\vspace{1mm}

\noindent\textbf{Data availability statement:}
Data sharing does not apply to this article, as no datasets were
generated or analysed during the current study.\par}

{\small
\bibliographystyle{alphaurl}
\bibliography{Ref_hyperbolic}}

\end{document}